\documentclass[12pt]{article}
\usepackage[export]{adjustbox}
\usepackage{graphicx}
\usepackage{amsfonts}
\usepackage{amscd}
\usepackage{indentfirst}
\usepackage{amssymb,amsthm, amsgen, amstext, amsbsy, amsopn,yfonts}
\usepackage{xcolor}
\usepackage{amsmath}
\usepackage{mathrsfs}
\usepackage[hyphens]{url}
\newtheorem{thm}{Theorem}[section]
\newtheorem{lemma}[thm]{Lemma}
\newtheorem{prop}[thm]{Proposition}
\newtheorem{cor}[thm]{Corollary}
\newtheorem{defin}[thm]{Definition}
\newtheorem{rem}[thm]{Remark}
\newtheorem{exam}[thm]{Example}

\newcommand{\I}{{\mathbb{I}}}

\newcommand{\J}{{\mathbb{J}}}

\newcommand{\N}{{\mathbb{N}}}

\newcommand{\R}{{\mathbb{R}}}
\newcommand{\SP}{{\mathbb{S}}}

\newcommand{\Z}{{\mathbb{Z}}}

\newcommand{\cA}{{\mathcal{A}}}

\newcommand{\cF}{{\mathcal{F}}}
\newcommand{\cG}{{\mathcal{G}}}

\newcommand{\cI}{{\mathcal{I}}}
\newcommand{\cJ}{{\mathcal{J}}}
\newcommand{\cK}{{\mathcal{K}}}

\newcommand{\cM}{{\mathcal{M}}}

\newcommand{\cR}{{\mathcal{R}}}
\newcommand{\cS}{{\mathcal{S}}}

\newcommand{\hM}{{\widehat{M}}}

\newcommand{\hX}{{\widehat{X}}}
\newcommand{\hY}{{\widehat{Y}}}

\newcommand{\hl}{{\widehat{l}}}

\newcommand{\bL}{{\overline{L}}}

\newcommand{\hxi}{{\hat{\xi}}}

\newcommand{\hlambda}{{\widehat{\lambda}}}

\newcommand{\hDelta}{{\widehat{\Delta}}}
\newcommand{\hLambda}{{\widehat{\Lambda}}}

\newcommand{\hSigma}{{\widehat{\Sigma}}}

\newcommand{\tg}{{\widetilde{g}}}
\newcommand{\tih}{{\widetilde{h}}}
\newcommand{\tp}{{\widetilde{p}}}

\newcommand{\tF}{{\widetilde{F}}}
\newcommand{\tG}{{\widetilde{G}}}
\newcommand{\tH}{{\widetilde{H}}}

\newcommand{\tM}{{\widetilde{M}}}
\newcommand{\tN}{{\widetilde{N}}}

\newcommand{\tX}{{\widetilde{X}}}
\newcommand{\tY}{{\widetilde{Y}}}

\newcommand{\lmin}{{\it l_{min}}}

\newcommand{\graph}{{\it graph\,}}

\newcommand{\Image}{{\rm Im\,}}

\newcommand{\Ker}{{\rm Ker\,}}

\newcommand{\Qed}{\ \hfill \qedsymbol \bigskip}

\newcommand{\Sh}{{\it Sh}}
\newcommand{\spec}{{\it spec}}

\newcommand{\supp}{{\it supp\,}}

\newcommand{\Op}{{\it Op}}

\newcommand{\sgrad}{{\it sgrad\,}}

\begin{document}

\title{{Legendrian persistence modules and dynamics}}

\author{\textsc Michael Entov$^{1}$,\ Leonid Polterovich$^{2}$}

\footnotetext[1]{Partially supported by the Israel Science Foundation grant 1715/18.}

\footnotetext[2]{ Partially supported by  the Israel Science Foundation grant 1102/20.}

\date{\today}

\maketitle

\begin{abstract}
{We relate the machinery of persistence modules to the Legendrian contact homology theory and to Poisson bracket invariants, and use it to show the existence of connecting trajectories of contact and symplectic Hamiltonian flows.}
\end{abstract}

\bigskip
\begin{small}
\centerline{{\it  To Claude Viterbo on the occasion of his 60th birthday}}
\end{small}

\tableofcontents

\section{Introduction and main results}
\label{sec-introduction}

\subsection{Interlinking}
\label{subsec-interlinking}

In the present paper we discuss a new facet of a method, introduced in \cite{BEP,EP-tetragons},
of finding orbits of Hamiltonian systems connecting a pair of disjoint subsets $(X_0,X_1)$ in the phases space. The method manifests a dynamical phenomenon called {\it interlinking}, which involves
a quadruple of subsets $(X_0,X_1,Y_0,Y_1)$ in a symplectic manifold $(M,\omega)$ satisfying
$X_0 \cap X_1 = Y_0 \cap Y_1 =\emptyset$. A Hamiltonian $H:M\times\SP^1 \to \R$
{\it separates} $Y_0,Y_1$ if
\begin{equation}\label{eq-sep-vsp}
\Delta(H,Y_0,Y_1):= \inf_{Y_1\times\SP^1} H - \sup_{Y_0\times\SP^1}H > 0.
\end{equation}

Let $\mu>0$. According to the definition in \cite{EP-tetragons},
$(Y_0,Y_1)$  {\it $\mu$-interlinks} (respectively, {\it autonomously $\mu$-interlinks}) $(X_0,X_1)$, if
for every Hamiltonian (respectively, every autonomous Hamiltonian) $H$ separating $Y_0,Y_1$ and generating a flow $\{ \phi_t\}_{t\in\R}$ on $M$,
there exist $t_0\in \R$, $x\in M$ and a positive $T \leq \mu/\Delta$, so that $\phi_{t_0} x\in X_0$ and $\phi_{t_0+T} x \in X_1$.
The piece of the trajectory $\{\phi_t x\}$, $t \in [t_0,t_0+T]$ is called a {\it chord of time-length $T$} connecting $X_0$ and $X_1$.

The pair $(Y_0,Y_1)$  {\it interlinks} (respectively, {\it autonomously interlinks}) the pair $(X_0,X_1)$ if it $\mu$-interlinks (respectively, autonomously
$\mu$-interlinks) it for some $\mu>0$.

\bigskip
\noindent
{\bf IMPORTANT REMARK:} In this paper we will consider only autonomous interlinking and for bre\-vi\-ty will omit the word ``autonomous". {\bf Thus, ``interlinking" further on in this paper is the same as ``autonomous interlinking" in \cite{EP-tetragons}.}
\bigskip

We will focus on quadruples of a special form lying in the symplectization $S\Sigma$ of
a contact manifold $(\Sigma,\xi)$, and its fillings. Let $\lambda$ be a contact form on $\Sigma$.
Recall that $S\Sigma$ is $\Sigma \times \R_+ (s)$ equipped with the symplectic form $\omega=d(s\lambda)$. We set $Y_0 = \{s=s_-\}$ and $Y_1= \{s=s_+\}$ for some $0<s_-\leq s_+$, and take $X_0$ and $X_1$
to be disjoint Lagrangian cobordisms whose boundaries project to Legendrian submanifolds
in $\Sigma$. We mainly concentrate on the case of cylindrical cobordisms
\begin{equation}\label{eq-cyl-vsp}
X_i= \Lambda_i \times [s_-,s_+]
\end{equation}
for some disjoint Legendrians $\Lambda_0,\Lambda_1$.  The main result of the present paper provides a sufficient condition for interlinking in terms of contact topology of this pair of Legendrians.

A very rough sketch of the method of \cite{BEP,EP-tetragons} is as follows. By using
Poisson bracket invariants coming from function theory on symplectic manifolds, one deduces the desired interlinking from obstructions to some special deformations of the symplectic form $\omega$ constant near $Z:= X_0\cup X_1\cup Y_0 \cup Y_1$. The constraints on such deformations are often provided by pseudo-holomorphic curves with the boundary on $Z$. The main new idea of the present paper is to extract such curves from filtered Legendrian contact homology of  $\Lambda_0$ and $\Lambda_1$. Thus our main point is to deduce ``dynamical interlinking" from ``contact-topological interlinking".

The realization of this idea requires tools from the relative symplectic field theory (RSFT)
\cite{EGH} combined with the theory of persistence modules which originated in topological data analysis \cite{ELZ,ZC} and brought to symplectic topology in \cite{PS}. RSFT associates to a (non-de\-ge\-ne\-rate) pair formed by a Legendrian submanifold and a contact form on the manifold an algebraic object, called the Legendrian contact homology, and to an exact Lagrangian cobordism between Legendrian submanifolds a morphism between the corresponding Legendrian contact homologies. In the simplest setting (say, where a contact manifold is the contactization of an exact symplectic manifold) {the enhancement of the RSFT using the action filtration} gives rise to the {\it filtered relative symplectic field theory (FRSFT)}. Namely, the Legendrian contact homology can be viewed as a persistence module formed by vector spaces (over $\Z_2$), while an exact Lagrangian cobordism between Legendrian submanifolds defines a morphism between the corresponding persistence modules,
see papers \cite{DR-S2,DR-2020} by Dimitroglou Rizell and M.~Sullivan. In the present paper we rely on some special instances of this theory only leaving more general results for a forthcoming manuscript \cite{EP-big}.

In fact, theory of persistence modules enables us, in certain situations, to detect
chords of contact flows even in the absence of interlinking, provided the
contact Hamiltonians have a sufficiently small oscillation in the uniform norm. This {\it robustness} of the existence mechanism for contact chords with respect to $C^0$-small perturbations of contact Hamiltonians is a feature of our approach.

Let us mention also that in a special case when the cobordisms are cylindrical as in \eqref{eq-cyl-vsp} and the Hamiltonian $H: S\Sigma \to \R$ is $\R_+$-homogeneous and positive, interlinking reduces to existence of Reeb chords of {\it a modified} contact form $\lambda/H$ connecting Legendrian submanifolds $\Lambda_0$ and $\Lambda_1$. Existence of such a chord with
a good upper bound on its time-length would follow from a Legendrian contact homology
theory valid for arbitrary contact forms. At the moment such a theory is still under construction, though it sounds likely that Pardon's work \cite{Pardon} on foundations of absolute contact homologies will be eventually extended to the relative (i.e. Legendrian) case. Meanwhile, our results yield existence of Reeb chords for arbitrary forms on certain contact manifolds. Applications of these results include, in particular, contact-topological methods in non-equilibrium thermodynamics \cite{EP-thermodyn}.

\subsection{The pool of contact manifolds}\label{subsec-pool}
In the present paper we focus on contact manifolds of the form $\Sigma = P\times \R (z)$,
where $(P^{2n},d\vartheta)$, $n\in\Z_{> 0}$, is an exact symplectic manifold with a symplectic form $d\vartheta$ and bounded geometry at infinity (see \cite{ALP}, \cite{EG} for the definition of this notion) and the contact structure on $\Sigma$ is defined by the contact form $\lambda = dz + \vartheta$. We will call such contact manifolds {\it nice}: to the best of our knowledge, this is the largest class of contact manifolds for which the details of the Legendrian contact homology theory have been worked out rigorously {in the published literature}. Note that the Reeb vector field $R$ of $\lambda$ is $\partial/\partial z$ and its flow has no periodic orbits.

A specific example of a nice contact manifold is the 1-jet space $J^1 Q = T^* Q (p,q)\times \R (z)$ of a closed manifold $Q$ equipped with the contact form $dz \pm pdq$. The forms
corresponding to different choices of the sign are related by the involution $p \to -p$.
We freely use both forms depending on the context, hoping that this will not cause a confusion.
Note that a neighborhood of the zero section in $J^1Q$  provides the universal
model for a neighborhood of any Legendrian copy of $Q$ in an arbitrary contact manifold \cite[Example 2.5.11]{Geiges}.

Write $ST^* \R^n$ for the space of co-oriented contact elements of $\R^n$, identified
with the unit cotangent sphere bundle of $\R^n$ with respect to the Euclidean metric.
There exists a contactomorphism
\begin{equation}\label{eq-hodograph}
(J^1 \SP^{n-1}, dz-pdq) \to (ST^* \R^n, pdq)
\end{equation}
(known as ``the hodograph map", see e.g. \cite[pp.48-49]{Arnold-top-invs-of-plane-curves-and-caustics})
identifying the standard contact forms on both spaces: it sends a point $(p,q,z)\in J^1 \SP^{n-1}$ to the unit  cotangent vector $q$ in the cotangent space of $zq+p\in\R^n$. (Here $p,q,z$ are local Darboux vector-coordinates on $J^1 \SP^{n-1}$).

\subsection{Sample applications: chords of symplectic Hamiltonians}
\label{subsec-sample-appls-sympl-hams}

Before discussing our results in a general setting, let us give a few basic definitions and present a sample of {dynamical applications. In Remark~\ref{rem-case-n-equals-1-history} below we discuss the relation between these applications and the dynamical results in our previous work.

A (time-dependent) Hamiltonian on a symplectic manifold is called {\it complete} if its Hamiltonian flow is defined for all times. Similarly, a (time-dependent) contact Hamiltonian (with respect to a contact form) on a contact manifold is called {\it complete} if its contact flow is defined for all times.

We extend the definition of a chord of a symplectic Hamiltonian (see Section~\ref{subsec-interlinking}) to the contact setting as follows:
given a (time-dependent) contact Hamiltonian $h$ (with respect to a contact form $\lambda$) on a contact manifold $\Sigma$ and two-disjoint subsets $Z_0, Z_1$ of $\Sigma$, a {\it chord
of $h$ from $Z_0$ to $Z_1$ of time-length $T>0$ (with respect to $\lambda$)} is a trajectory of
the contact flow of $h$ (with respect to $\lambda$) that passes through $Z_0$ at a time $t_0$ and
through $Z_1$ at the time $t_0+T$. If $h>0$, such a chord is a Reeb chord of the contact form $\lambda/h$ of the same time-length. Given a chord $a$, we write $|a|$ for its time-length.

\medskip

Consider $\R^{2n} (p,q) = \R^n (p) \times \R^n (q)$ with the standard symplectic form $dp\wedge dq$. We view $R^{2n}$ as the symplectization of the contact manifold
$ST^*\R^n$ filled by the zero section $\{p=0\}$. Let $|\cdot |$ denote the Euclidean norm on $\R^n (q)$.

Let $0<s_-<s_+$. Let $x_0, x_1\in\R^n (q)$, $x_0\neq x_1$. If $n=1$, assume that $x_0 < x_1$.

Define $X_0,X_1,Y_0,Y_1\subset \R^{2n}$ as follows.

If $n>1$, set
\[
X_0: = \{\ (p,x_0)\in\R^{2n}\ |\ s_-\leq |p|\leq s_+\ \},
\]
\[
X_1:= \{\ (p,x_1)\in\R^{2n}\ |\ s_-\leq |p|\leq s_+\ \},
\]
\[
Y_0:= \{\ |p|=s_-\ \},
\]
\[
Y_1:= \{\ |p|=s_+\ \}.
\]

If $n=1$, set
\[
X_0: = \{\ (p,x_0)\in\R^{2n}\ |\ s_-\leq p\leq s_+\ \},
\]
\[
X_1:= \{\ (p,x_1)\in\R^{2n}\ |\ s_-\leq p\leq s_+\ \},
\]
\[
Y_0:= \{\ (s_-,q)\in\R^{2n}\ |\ x_0 \leq q\leq x_1\ \},
\]
\[
Y_1:= \{\ (s_+,q)\in\R^{2n}\ |\ x_0\leq q\leq x_1\ \}.
\]

Let $H: \R^{2n}\times \SP^1\to\R$ be a complete
Hamiltonian.

Set:
\[
c_{min} := \min_{X_0\times \SP^1} H,\ c_{max} := \max_{X_0\times \SP^1} H.
\]

\begin{thm}
\label{thm-main-R2n}
\

\noindent
A. Assume $\Delta(H;X_1,X_0) >0$ and the following conditions are satisfied:

\begin{equation}
\label{eqn-H-is-time-independent}
H\ \textrm{is time-independent},
\end{equation}

\begin{equation}
\label{eqn-H-geq-0-on-X0}
H|_{X_0}\geq 0,
\end{equation}

\begin{equation}
\label{eqn-intersection-of-supp-H-with-layer-is-compact}
\supp H\cap \{ s_-\leq |p|\leq s_+\}\ \textrm{is compact}.
\end{equation}

Then there exists a chord of $H$ from $Y_0$
to
$Y_1$
of time-length bounded from above by $\displaystyle \frac{|x_0-x_1| (s_+-s_-)}{\Delta(H;X_1,X_0)}$.

In the case $n=1$ the claim holds even without assuming \eqref{eqn-H-is-time-independent}, \eqref{eqn-H-geq-0-on-X0}, \eqref{eqn-intersection-of-supp-H-with-layer-is-compact}.

\bigskip
\noindent
B.
Assume $\Delta(H;Y_0,Y_1) =: \Delta >0$ and for some $0<e<1/2$ the following condition is satisfied:

\begin{equation}
\label{eqn-upper-bound-on-partial-H-partial-t}
\sup_{c_{min} - e\Delta \leq H\leq c_{max} + e\Delta} \left|\frac{\partial H}{\partial t}\right| < \frac{(1-2e)e\Delta^2}{(s_+ - s_-)|x_0-x_1|}.
\end{equation}

Then there exists a chord of $H$ from $X_0$
to
$X_1$ of time-length bounded from above by $\displaystyle \frac{|x_0-x_1| (s_+-s_-)}{(1-2e)\Delta(H;Y_0,Y_1)}$.

In particular, if $H$ is time-independent (and thus \eqref{eqn-upper-bound-on-partial-H-partial-t} holds for all $0<e<1/2$), then the
time-length of the chord is $\displaystyle \leq\frac{|x_0-x_1| (s_+-s_-)}{\Delta(H;Y_0,Y_1)}$.

In the case $n=1$ the claim holds even without assuming \eqref{eqn-upper-bound-on-partial-H-partial-t}.
\end{thm}

\medskip\noindent
For the proof see Section~\ref{sec-the-unit-cotangent-bundle-of-Rn}.
\begin{rem}\label{rem-int}{\rm
In the case of autonomous Hamiltonians,
part B of Theorem~\ref{thm-main-R2n} implies that $(Y_0,Y_1)$ $\mu$-interlinks $(X_0,X_1)$, where
$\mu= |x_0-x_1|(s_+ - s_-)$. Part A of Theorem~\ref{thm-main-R2n} does not imply that $(Y_0,Y_1)$ $\mu$-interlinks $(X_0,X_1)$ but is a somewhat weaker claim.

In the proofs of the two parts we use two slightly different Poisson bracket invariants -- the tools used for detecting interlinking in Section~\ref{subsec-pb-invariant} below. The appearance of the same constant $\mu= |x_0-x_1|(s_+ - s_-)$ in both parts of Theorem~\ref{thm-main-R2n} is due to the fact that the proofs of both claims are based on the same lower bound on both Poisson bracket invariants, coming from areas of certain pseudo-holomorphic curves used in the theory of Legendrian contact homology.
In fact, one can think of $\mu$ as the area of a (pseudo-holomorphic) quadrilateral whose edges lie in $X_0,X_1,Y_0,Y_1$.
}
\end{rem}

\begin{rem}
\label{rem-case-n-equals-1-history}
{\rm
In the case $n=1$ the claim of Theorem~\ref{thm-main-R2n} follows rather directly
from the results in \cite{EP-tetragons} (cf. \cite[Thm. 1.20]{BEP}) -- see the proof of Theorem~\ref{thm-main-R2n} in Section~\ref{sec-the-unit-cotangent-bundle-of-Rn}.

In the case $n>1$ we do not know any way to obtain Theorem~\ref{thm-main-R2n} from the results in \cite{EP-tetragons}. (Basically, the obstacle is
the non-existence of a positive Legendrian isotopy between unit cotangent spheres at different points in the unit cotangent bundle of $\R^n$ --
see \cite{Colin-Ferrand-Pushkar}, \cite{Chernov-Nemirovski}).

Let us note that the existence of chords as in Theorem~\ref{thm-main-R2n} for similarly defined sets in the cotangent bundle of the torus and compactly supported Hamiltonians can be proved along the lines of the proof of \cite[Thm. 1.13]{BEP} using symplectic quasi-states, with the upper bound on the time-length
of the chord depending on the size of the support of the Hamiltonian. The methods of this paper do not allow us to treat this case because
the foundations of the Legendrian contact homology have not been worked out rigorously yet for the relevant setting.
}
\end{rem}

\begin{rem}
\label{rem-mech-Hamiltonian-with-bdd-potential}
{\rm
{For (time-independent) mechanical Hamiltonians $H$ the existence of Hamiltonian chords of $H$ from $X_0$ to $X_1$, as in part B of Theorem~\ref{thm-main-R2n},}
can be obtained by the classical Maupertius's least action principle (see e.g. \cite[p.247]{Arnold-MathMeth}).

Namely, assume that $H$ is a complete mechanical Hamiltonian of the form $H(p,q) = |p|^2/2 + U(q)$, where
$0\leq \sup_{\R^n} |U| < +\infty$. Let $x_0,x_1\in\R^n$, $0<s_- < s_+$ and assume that for some
$C> \sup_{\R^n} |U|$ the level set $\{ H=C\}$ intersects the sets $X_0$, $X_1$.

Consider the Riemannian metric $\tg$ on $\R^n$ of the form $\tg=\sqrt{C-U(q)}g$, where $g$ is the Euclidean metric (the metric $\tg$ is called the Jacobi metric). It is not hard to verify that the metric $\tg$ is complete and therefore, by the Hopf-Rinow theorem, there exists a minimal geodesic of $\tg$ from $x_0$ to $x_1$. By Maupertius's least action principle, the lift of the geodesic to the level set $\{ H=C\}$ of $H$ in $T^* \R^n = \R^{2n}$ is a Hamiltonian chord of $H$ from $X_0$ to $X_1$.
}
\end{rem}

\subsection{Sample applications: contact dynamics}
\label{subsec-sample-appls-contact-dynamics}
Let us present sample applications to contact dynamics.

\subsubsection{Contact interlinking} \label{susubsec-ci}
Let $(\Sigma,\xi)$ be a contact manifold equipped with a contact form $\lambda$.
An ordered pair $(\Lambda_0,\Lambda_1)$ of disjoint Legendrian submanifolds $\Lambda_0,\Lambda_1 \subset \Sigma$ is called {\it $\mu$-interlinked} if there exists a constant $\mu= \mu(\Lambda_0,\Lambda_1,\lambda) >0$ such that every bounded strictly positive contact Hamiltonian $h$ on $\Sigma$ with $h \geq c >0$ possesses an orbit  of time-length $\leq \mu/c$ starting at $\Lambda_0$ and arriving at $\Lambda_1$. The pair $(\Lambda_0,\Lambda_1)$ is called {\it interlinked} if it is $\mu$-interlinked for some $\mu>0$. The pair $(\Lambda_0,\Lambda_1)$ is called {\it robustly interlinked},
if every pair $(\Lambda'_0, \Lambda'_1)$ of Legendrians
obtained from $(\Lambda_0,\Lambda_1)$ by a sufficiently $C^1$-small Legendrian isotopy is interlinked.

Write $R_t$ for the Reeb flow of $\lambda$. Given two Legendrian submanifolds $\Lambda_0,\Lambda_1 \subset \Sigma$, a Reeb chord $R_t x$, $t \in [0,\tau]$ with $x \in \Lambda_0$ and $y:= R_\tau x \in \Lambda_1$ is called {\it non-degenerate} if
\begin{equation}\label{eq-nondeg}
D_xR_\tau(T_x\Lambda_0) \oplus T_y \Lambda_1 = \xi_y\;,
\end{equation}
where $\xi_y$ stands for the contact hyperplane at $y$.

Let $\Sigma$ be the jet space $J^1 Q = T^* Q (p,q)\times \R (z)$ of a closed
manifold $Q$ equipped with the contact form $dz - pdq$. Let $R= \partial/\partial z$ be the Reeb vector field of $\lambda$. Let $\Lambda_0=\{p=0,z=0\}$ be the zero section.

\begin{thm}
\label{thm-1-jet-space-0-section-and-its-Reeb-shift-homol-bonded}
\begin{itemize}
\item[{(i)}] Let $\psi$ be a positive function on $Q$, and let $\Lambda_1:= \{z=\psi(q), p=\psi'(q)\}$ be the graph of its 1-jet. Then the pair $(\Lambda_0, \Lambda_1)$ is robustly interlinked.
\item[{(ii)}]  Assume that $\Lambda_1 \subset \Sigma=J^1 Q$ is a Legendrian submanifold
Legendrian isotopic to $\Lambda_0$,
with
the
  following property: there is
a
  unique chord of the Reeb flow $R_t$ starting on $\Lambda_0$ and ending on $\Lambda_1$, and this chord is non-degenerate. Then the pair $(\Lambda_0,\Lambda_1)$ is interlinked.
    \end{itemize}
\end{thm}

\medskip\noindent The proof is given in Section~\ref{sec-applications-to-contact-dynamics}.

\begin{rem}
{\rm
The assumption in part (ii) of Theorem~\ref{thm-1-jet-space-0-section-and-its-Reeb-shift-homol-bonded} that $\Lambda_1$ is Legendrian isotopic to $\Lambda_0$ can be considerably weakened. The actual assumption that we need for the proof is that the Legendrian contact homology of $\Lambda_1$ is not zero -- see Section~\ref{sec-the-1-jet-space}.
}
\end{rem}

\subsubsection{Beyond interlinking} \label{subsec-twochords}
Let $\psi$ be a positive function on $Q$, and let $\Lambda := \{z=\psi(q), p=\psi'(q)\}$ be the graph of its 1-jet in $T^*Q \times \R$. Let us note that for every critical
point $q$ of $\psi$ there is a Reeb chord of the time length $\psi(q)$ from the zero
section $\Lambda_0$ to $\Lambda$. As we have seen in Theorem~\ref{thm-1-jet-space-0-section-and-its-Reeb-shift-homol-bonded}(i), the pair $(\Lambda_0,\Lambda)$ is interlinked. Note that the order matters: the pair $(\Lambda,\Lambda_0)$ is not interlinked -- indeed, there is no Reeb chord from $\Lambda$ to $\Lambda_0$.
Assume now
that $K$ is Legendrian isotopic to $ \Lambda $ outside $\Lambda_0$, and that
there exist exactly two Reeb chords $A,a$ starting on $ K $ and ending on $\Lambda_0$. Assume further that both chords are non-degenerate, and their time lengths $|A|,|a|$ satisfy
\begin{equation}
\label{eq-twochords}
0 < |A|-|a| < |b|,
\end{equation}
for every Reeb chord $b$ starting and ending on $\Lambda_0 \sqcup  K $.
The next result states that for contact Hamiltonians with a {\it sufficiently small oscillation} (i.e., for small perturbations of the {\it constant} contact Hamiltonian  $1$),
one can establish existence of a chord even in the absence of interlinking.

\begin{figure}[ht]
\begin{minipage}[c][1\width]{
	   0.6\textwidth}
\includegraphics[width=1.2\textwidth,left]{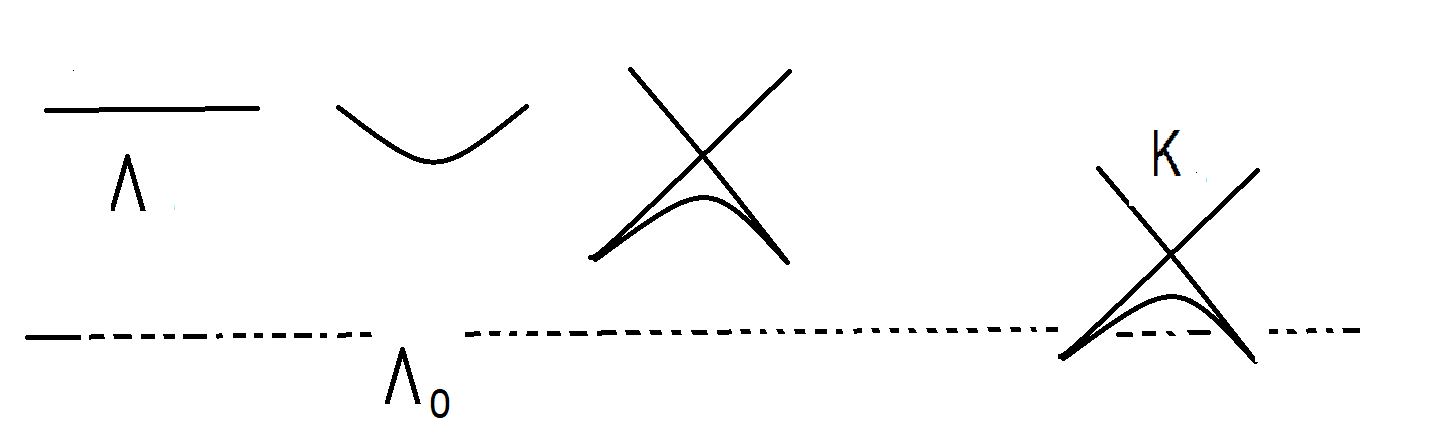}
\end{minipage}
\hfill
\begin{minipage}[c][1\width]{
	   0.8\textwidth}
\includegraphics[width=0.7\textwidth,center]{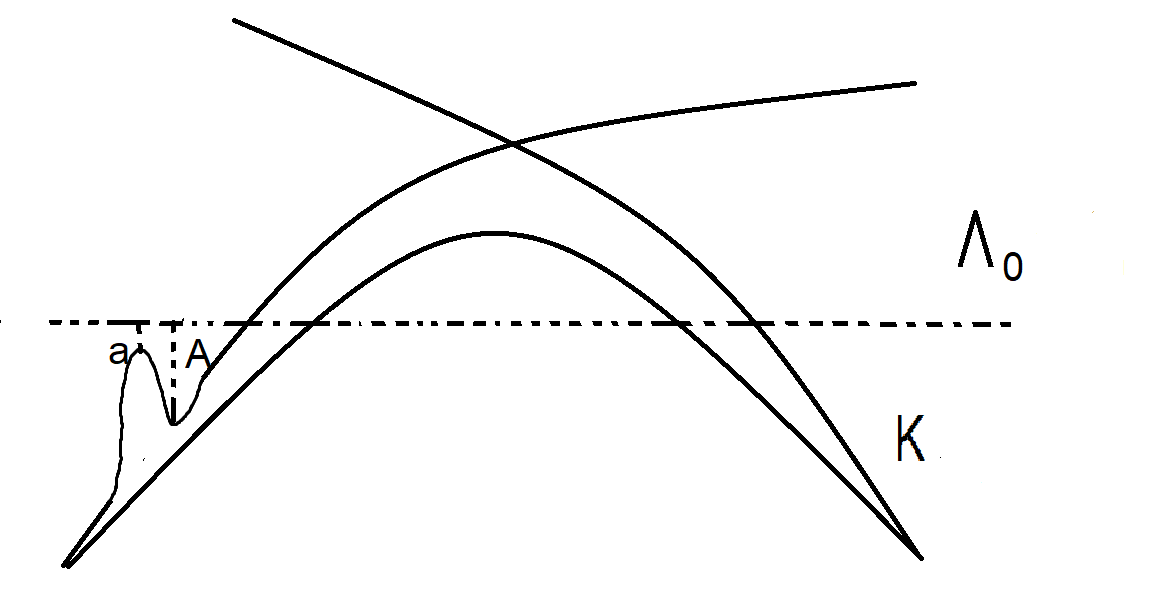}
%\centering
\end{minipage}
\caption{Isotopy from $\Lambda$ to $K$ (left); $K$ zoomed in (right)}
\end{figure}

\begin{thm}\label{thm-twochords} Let $h$ be any positive bounded contact Hamiltonian on $\Sigma$ with $c:= \inf_\Sigma h \leq h \leq \sup_\Sigma h =:C$. If
\begin{equation} \label{eq-constraint}
\frac{C}{c} < \frac{|A|}{|a|}\;,
\end{equation}
then there is a chord of $h$ starting on $K$ and ending on $\Lambda_0$ of time length
$\leq |a|(|A|-|a|)/(|A|c - |a|C)$.
\end{thm}

\medskip
\noindent
The proof is given in
Section~\ref{sec-the-1-jet-space}.
  For an example of a Legendrian submanifold $ K \subset T^*S^1 \times \R$ satisfying the assumption
of Theorem~\ref{thm-twochords} we refer to Figure 1 describing the front projection of $ K $. As we shall see later on, the proof
of this result involves the machinery of persistence modules.

\subsubsection{Chords of contact Hamiltonians }
\label{subsubsec-sample-appls-chords-of-contact-hams-from-one-Leg-to-another}
Next, we relax the setting of contact interlinking and work with contact Hamiltonians which may be time-dependent, unbounded, and may change sign.

Let $l\in \R$, $l>0$.
Let $\Lambda_0$, $\Lambda_1$ be the following Legendrian submanifolds of $(\Sigma,\xi)$: if $\Sigma=ST^*\R^n$, then $\Lambda_0$ is the unit cotangent sphere
at some $x_0\in\R^n$ while
\begin{itemize}
\item[{(i)}] either $\Lambda_1$ is the image of $\Lambda_0$ under the time-$l$ Reeb flow;
\item[{(ii)}]or $\Lambda_1$ is the unit cotangent sphere at some $x_1\in\R^n$, $|x_0 - x_1| = l$, in which case there exists a unique non-degenerate
Reeb chord starting at $\Lambda_0$ and ending on $\Lambda_1$.
\end{itemize}
Note that these pairs $(\Lambda_0,\Lambda_1)$ are interlinked by  Theorem~\ref{thm-1-jet-space-0-section-and-its-Reeb-shift-homol-bonded}(i) and (ii), respectively.
We shall also allow $\Sigma = J^1 Q$, where $\Lambda_0$ is the zero-section and $\Lambda_1$ is its image under the time-$l$ Reeb flow, as in Theorem~\ref{thm-1-jet-space-0-section-and-its-Reeb-shift-homol-bonded}(i).
We shall discuss applications concerning the existence
of chords of contact Hamiltonians from $\Lambda_0$ to $\Lambda_1$.

Assume $h: \Sigma\times \SP^1\to \R$ is a complete (time-periodic) contact Hamiltonian (with respect to $\lambda$). {Write $h_t := h(\cdot,t)$, $t\in\SP^1$}. Denote by $v_t$, $t\in\SP^1$, the (time-periodic) contact Hamiltonian vector field of $h$ with respect to the contact form $\lambda$. If $v_t$ is time-independent, we write just $v$. Let $\{\varphi_t\}$ be the flow of $v_t$ -- that is, the contact Hamiltonian flow of $h$.

\begin{defin}
\label{def-cooperative-ham}
{\rm
Let us say that $h: \Sigma\times\SP^1\to\R$ is {\it $C$-cooperative with  $\Lambda_0$, $\Lambda_1$} for $C>0$ if either of the following conditions holds:

\bigskip
\noindent
(a)
$h < C$ on $\Lambda_1\times \SP^1$ and $dh_t (R)\geq 0$ on $\{ h_t\geq C\}$ for all $t\in\SP^1$.

\bigskip
\noindent
(b)
$h < C$ on $\Lambda_0\times \SP^1$ and  $dh_t (R)\leq 0$ on $\{ h_t\geq C\}$ for all $t\in\SP^1$.

\bigskip
\noindent
We will say that $h$ is {\it cooperative with  $\Lambda_0$, $\Lambda_1$} if it is $C$-cooperative with  $\Lambda_0$, $\Lambda_1$ for some $C>0$.
}
\end{defin}
\bigskip

Note that conditions (a) and (b) hold, in particular, for a sufficiently large $C$ if $\sup_{\Sigma\times\SP^1} h < +\infty$. Conditions (a) and (b) guarantee that a chord of $h$ from $\Lambda_0$ to $\Lambda_1$, if it exists, does not leave the set $\{ h \leq C\}$ at any time.

\begin{thm}[cf. Rem. 1.14 in \cite{EP-tetragons}]
\label{thm-main-existence-of-connecting-Reeb-chords-unit-cot-bundle}
Assume that $h$ is cooperative with $\Lambda_0$, $\Lambda_1$ and that $\inf_{\Sigma\times\SP^1} h >0$.
Assume also that for some $0<e<1/2$
\[
\sup_{\Sigma\times \SP^1} |\partial h/\partial t| < \frac{(1-2e)e\big(\inf_{\Sigma\times \SP^1} h\big)^3} {\big(\max_{\Lambda_0\times \SP^1} h + e\inf_{\Sigma\times \SP^1} h  \big) l }.
\]

Then there exists a chord of $h$ from $\Lambda_0$ to $\Lambda_1$ of time-length bounded from above by
$\displaystyle \frac{l}{(1-2e)\inf_{\Sigma\times \SP^1} h}$.

In particular, if $h$ is time-independent, then the time-length can be bounded from above by $\displaystyle\frac{l}{\inf_{\Sigma\times \SP^1} h}$.
\end{thm}
\bigskip

For the proof of Theorem \ref{thm-main-existence-of-connecting-Reeb-chords-unit-cot-bundle} see Section~\ref{sec-the-unit-cotangent-bundle-of-Rn}.

For other results on the existence of Reeb chords between different Legendrian submanifolds (or equivalently, chords of positive contact Hamiltonians) see the papers of G.Dimitroglou-Rizell and M.Sullivan \cite{DR-S1,DR-S2} (for a comparison of their results with the results in \cite{EP-tetragons} and here, see \cite[Sec. 1.3]{DR-S2}).

Theorem~\ref{thm-main-existence-of-connecting-Reeb-chords-unit-cot-bundle}, together with a basic dynamical assumption,
allows to obtain the following results concerning the chords of contact Hamiltonians {\sl that are not  everywhere positive}.

\begin{cor}
\label{cor-non-positive-hamiltonians-separating-hypersurface-unit-cot-bundle}
{Assume that $h$ is cooperative with $\Lambda_0$, $\Lambda_1$}
and there exists a (possibly non-compact or disconnected) closed codimension-0 submanifold $\Xi\subset \Sigma$ with a (possibly non-compact or disconnected) boundary $\partial\Xi$,
so that

\bigskip
\noindent
(1) $\inf_{\Xi\times\SP^1} h >0$ (but $h$ may be negative outside ${\Xi}\times\SP^1$).

\bigskip
\noindent
(2) $\sup_{\partial\Xi\times\SP^1} h < +\infty$.

\bigskip
\noindent
(3) For each $t\in\SP^1$
the vector field $v_t$ is transverse to $\partial\Xi$ (in particular, $\partial\Xi$ is a convex surface in the sense of contact topology -- see \cite{Giroux-CMH1991})
and either points inside $\Xi$ everywhere on $\partial\Xi$ or points outside $\Xi$ everywhere on $\partial\Xi$.

\bigskip
\noindent
(4)
Both $\Lambda_0$ and $\Lambda_1$ lie in $\Xi$.

\bigskip
\noindent
Assume also that for some $0<e<1/2$
\[
\sup_{\Xi\times \SP^1} |\partial h/\partial t| < \frac{(1-2e)e\big(\inf_{\Xi\times \SP^1} h\big)^3} {\big(\max_{\Lambda_0\times \SP^1} h + e\inf_{\Xi\times \SP^1} h  \big)l }.
\]

Then there exists a chord of $h$ from $\Lambda_0$ to $\Lambda_1$ whose time-length is bounded from above by $\displaystyle \frac{l}{(1-2e)\inf_{\Xi\times \SP^1} h}$.

If $h$ is time-independent, then the time-length of the chord is bounded from above by ${l}/\inf_\Xi h$.

\end{cor}
\bigskip

For the proof of Corollary~\ref{cor-non-positive-hamiltonians-separating-hypersurface-unit-cot-bundle} see Section~\ref{sec-the-unit-cotangent-bundle-of-Rn}.

\begin{rem}
\label{rem-separating-unit-cot-bundle}
{\rm
Assume $h$ is time-independent and $\Xi := \{ h\geq c\}$ for some $c>0$. Then the conditions (1) and (2) are satisfied automatically, while condition (3) is equivalent to $\partial\Xi$ being transverse to the Reeb vector field $R$, because $dh (v) = h dh(R)$.
}
\end{rem}
\bigskip

\begin{exam}
\label{exam-hypersurface-Xi}
{\rm
An example satisfying the assumptions of Corollary~\ref{cor-non-positive-hamiltonians-separating-hypersurface-unit-cot-bundle} can be constructed as follows. Write $R_t$ for the Reeb flow on $J^1Q$. Let $\Lambda_0$ and $\Lambda_1$ be the images of the zero-section under $R_c$ and $R_{c+l}$, respectively, with $c,l>0$.
Note that $\Lambda_0,\Lambda_1\subset \Xi  := \{ z \geq c\}$. Thus condition (4) in Corollary~\ref{cor-non-positive-hamiltonians-separating-hypersurface-unit-cot-bundle} is satisfied.

Consider a time-independent contact Hamiltonian $h = az + g$ on $\Sigma = J^1 Q = T^* Q\times\R$, where $a>0$, $z$ is the coordinate along the $\R$-factor and $g$ is a smooth bounded function on $T^* Q$. Assume that $\inf_\Xi h >0$ -- this can be achieved if $a$ is sufficiently large compared to $||g||_{L^\infty}$.
Then condition (1) in Corollary~\ref{cor-non-positive-hamiltonians-separating-hypersurface-unit-cot-bundle} is satisfied.
Condition (2) is satisfied since $g$ is bounded. Finally, condition (3) is satisfied since the Reeb vector field $R$ of the standard contact form on $J^1 Q$ is $\partial/\partial z$ and therefore $dh (v) = h dh (R) = ah >0$ on $\partial\Xi = \{ z=c\}$.
It is also easy to verify that $h$ is $C$-cooperative with $\Lambda_0,\Lambda_1$ for a sufficiently large $C$.

We have verified that the objects above -- and accordingly their preimages under $\psi$ --
satisfy the assumptions of Corollary~\ref{cor-non-positive-hamiltonians-separating-hypersurface-unit-cot-bundle}. Consequently, Corollary~\ref{cor-non-positive-hamiltonians-separating-hypersurface-unit-cot-bundle} yields the existence of a chord of $h$ from $\Lambda_0$ to $\Lambda_1$ of time-length $\leq l/\inf_\Xi h$.
}
\end{exam}

\subsubsection{Contact flows with large conformal factor}
\label{subsubsec-sample-appls-conf-factor}

Our next result illustrates that contact Hamiltonians separating (in a suitable sense)  certain pairs of Legendrian submanifolds generate contact flows with an arbitrarily large conformal factor. Here $\Lambda_0$ and $\Lambda_1$ are as in the beginning of Section~\ref{subsubsec-sample-appls-chords-of-contact-hams-from-one-Leg-to-another}.

\begin{thm}
\label{thm-conformal-factor-unit-cot-bundle}

Assume that $h$ is time-independent, {compactly supported} and
\[
h|_{\Lambda_0} {\geq} 0,\ h|_{\Lambda_1} {<} 0.
\]

{
Then the conformal factor of $\varphi_t$ takes arbitrarily large values as $t$ varies between $0$ and $+\infty$:
\[
\inf\limits_{t\in (0,+\infty), y\in\Sigma} {\frac{\left(\varphi_t^{-1}\right)^* \lambda \left(\varphi_t \left(y\right)\right)}{\lambda \left(\varphi_t \left(y\right)\right)} = +\infty}.
\]
}
\end{thm}

\medskip\noindent Recall that the conformal factor is an important dynamical characteristic
playing the role of the contact Lyapunov exponent. For the proof see Section~\ref{sec-the-unit-cotangent-bundle-of-Rn}.

\subsection{Scheme of the proof: homologically bonded pairs}
\label{subsec-homologically-bonded-Leg-submfds}

Let us outline a key property of $\Sigma$, $\lambda$, $\Lambda_0$, $\Lambda_1$ above
that allows to prove the results in Sections~\ref{subsec-sample-appls-sympl-hams},\ref{subsec-sample-appls-contact-dynamics}
and outline the general scheme of the proofs.

Let $\Lambda_0,\Lambda_1\subset\Sigma$ be disjoint Legendrian (not necessarily connected) compact submanifolds without boundary of a nice contact manifold (see Section~\ref{subsec-pool}). Let $\Lambda:= \Lambda_0\sqcup\Lambda_1$. Assume that the pair {$(\Lambda,\lambda)$} is {\it non-degenerate} -- that is, there are only finitely many Reeb chords of $\Lambda$ and they are all non-degenerate (this can be always achieved by a $C^\infty$-small Legendrian perturbation of either of the two Legendrian submanifolds). Then one can associate to the pair $(\Lambda = \Lambda_0\sqcup\Lambda_1,\lambda)$ its (full) Chekanov-Eliashberg algebra -- a free non-commutative unital algebra over $\Z_2$ generated by all the Reeb chords of $\Lambda$. The algebra is filtered by the action (the action of a Reeb chord is its time-length; the action of a monomial, or a product of the Reeb chords, is the sum of the actions of its factors).

We consider a vector subspace of the Chekanov-Eliashberg algebra, which we will call the {\it $01$-subspace} -- it is generated by the monomials $a_1\cdot\ldots\cdot a_k$, $k\in\Z_{>0}$, where $a_1$ starts at $\Lambda_0$, $a_k$ ends at $\Lambda_1$, and for each $m=1,\ldots,k-1$ the end of $a_m$ lies in the same component of $\Lambda$ as the origin of $a_{m+1}$.

Recall that the differential $\partial_J$ on the Chekanov-Eliashberg algebra {depends on an almost complex structure $J$ on the symplectization of $(\Sigma, \ker \lambda)$ and} is defined as follows: the differential of a generator (that is, a Reeb chord) is defined using the count of {$J$-holomorphic} disks {in the symplectization} with one positive and possibly several negative punctures on the boundary, whose boundary lies in $\Lambda$ and that converge near the punctures to {cylinders over} Reeb chords of $\Lambda$; the differential is then extended to the whole algebra using the Leibniz rule {and the condition $\partial_J (1) :=0$} (see Section~\ref{subsec-Chekanov-Eliashberg-algebra-pers-mod}).

The differential preserves the $01$-subspace and lowers the filtration. This allows to view the resulting homology of the $01$-subspace -- {\it
the filtered Legendrian contact homology of $(\Lambda:=\Lambda_0\sqcup\Lambda_1,\lambda)$} -- as a persistence module defined over $(-\infty,+\infty)$ and apply the theory of persistence modules to its study. In particular, one can associate to it its barcode -- a collection of intervals, called {\it bars}, lying in $(0,+\infty)$.

For $s\in (1,+\infty)$ let $l_{min,s} (\Lambda_0,\Lambda_1,\lambda)$ denote the smallest left end of a bar of multiplicative length greater than $s$ in the barcode. (The multiplicative length of a bar in $(0,+\infty)$ is the ratio of its right and left ends; note that it may be infinite). Let $l_{min,\infty} (\Lambda_0,\Lambda_1,\lambda)$ denote the smallest left end of an infinite bar. If there are no such bars, set $l_{min,s} (\Lambda_0,\Lambda_1,\lambda) := +\infty$ or, respectively, $l_{min,\infty} (\Lambda_0,\Lambda_1,\lambda) := +\infty$.

If the pair $(\Lambda = \Lambda_0\sqcup\Lambda_1,\lambda)$ is degenerate, then $\Lambda = \Lambda_0\sqcup\Lambda_1$ can be approximated by Legendrian submanifolds $\Lambda'=\Lambda'_0\sqcup\Lambda'_1$ obtained from $\Lambda$ by a $C^\infty$-small Legendrian isotopy, so that the pair $(\Lambda',\lambda)$
is non-degenerate. {Extend the definition of $l_{min,s} (\Lambda_0,\Lambda_1,\lambda)$, $s\in (1,+\infty]$, to all pairs $(\Lambda = \Lambda_0\sqcup\Lambda_1,\lambda)$ as follows:}
\[
l_{min,s} (\Lambda_0,\Lambda_1,\lambda) := \liminf l_{min,s} (\Lambda'_0,\Lambda'_1,\lambda),
\]
where the $\liminf$ is taken over all such $\Lambda'=\Lambda'_0\sqcup\Lambda'_1$ converging to $\Lambda$ (in the $C^\infty$-topology).
{One can show that for non-degenerate pairs $(\Lambda = \Lambda_0\sqcup\Lambda_1,\lambda)$ this definition and the original one yield the same $l_{min,s} (\Lambda_0,\Lambda_1,\lambda)$ -- cf. Remark~\ref{rem-l-min-s-robust}.}

If $l_{min,s} (\Lambda_0,\Lambda_1,\lambda)<+\infty$ for all {$s\in (1,+\infty)$}, we will say that the pair $(\Lambda_0\sqcup\Lambda_1,\lambda)$ is {\it {weakly} homologically bonded}. If $l_{min,\infty} (\Lambda_0,\Lambda_1,\lambda)<+\infty$, we will say that the pair $(\Lambda_0\sqcup\Lambda_1,\lambda)$ is homologically bonded. If the pair is non-degenerate, these conditions mean that the corresponding barcode contains bars of arbitrarily large multiplicative length, or, respectively, an infinite bar.

The key property of the setting in Sections~\ref{subsec-sample-appls-sympl-hams},\ref{subsec-sample-appls-contact-dynamics}
is that
the pair $(\Lambda_0\sqcup\Lambda_1,\lambda)$ is {\it homologically bonded}.

Consider the {\it stabilization of $\Sigma$} which is the manifold $\hSigma := \Sigma\times \R (r) \times\SP^1 (\tau)$ equipped with the contact form $\hlambda := \lambda + rd\tau$. This contact manifold is also nice. For a Legendrian submanifold $\Delta\subset\Sigma$ define a Legendrian submanifold $\hDelta\subset\hSigma$ by
\[
\hDelta:= \Delta\times \{ r=0\}.
\]
For each $s\in (1,+\infty]$ define
\[
\hl_{min,s} (\Lambda_0,\Lambda_1,\lambda):= l_{min,s} (\hLambda_0,\hLambda_1,\hlambda).
\]

The pair $(\Lambda_0\sqcup\Lambda_1,\lambda)$ is said to be {\it stably homologically bonded} if\break
$\hl_{min,\infty} (\Lambda_0,\Lambda_1,\lambda)<+\infty$.

\bigskip
\begin{rem}
\label{rem-hom-bondedness-vs-stable-hom-bondedness}
{\rm
It is likely that homological bondedness implies stable homological bondedness. If $\Sigma$ is the standard contact $\R^3$, this follows from a result of Ekholm-K\'alm\'an \cite[Thm. 1.1]{Ekholm-Kalman-JSG2008}, but, to the best of our knowledge, the case of a general (nice) contact manifold has not been worked out so far.
}
\end{rem}
\bigskip

For the Legendrian submanifolds $\Lambda=\Lambda_0\sqcup\Lambda_1$ in Section~\ref{subsec-sample-appls-contact-dynamics} the pair $(\Lambda=\Lambda_0\sqcup\Lambda_1,\lambda)$ is homologically bonded and stably homologically bonded -- see Sections~\ref{sec-the-1-jet-space}, \ref{sec-the-unit-cotangent-bundle-of-Rn}.

Let us now explain how (stable) homological bondedness is used to prove the results in Sections~\ref{subsec-sample-appls-sympl-hams},\ref{subsec-sample-appls-contact-dynamics}.

For a pair $(\Lambda_0\sqcup \Lambda_1,\lambda)$ in a general nice contact manifold $(\Sigma,\lambda)$,
consider the trivial exact Lagrangian cobordism $(\Lambda_0\sqcup\Lambda_1)\times [s_-,s_+]$ in the trivial exact symplectic cobordism $\left(\Sigma\times [s_-,s_+], d(s\lambda)\right)$. For instance, in the setting of Theorem~\ref{thm-main-R2n} the latter exact symplectic cobordism can be identified with the manifold $\{ (p,q)\in\R^{2n}\ |\ s_0\leq |p|\leq s_1\}$ whose boundary components -- the sets $Y_0$, $Y_1$ -- are identified, respectively, with $\Sigma\times s_-$ and $\Sigma\times s_-$. The parts $\Lambda_0\times [s_-,s_+]$, $\Lambda_1\times [s_-,s_+]$ of the trivial exact Lagrangian cobordism $(\Lambda_0\sqcup\Lambda_1)\times [s_-,s_+]$ are then identified, respectively, with the sets $X_0$, $X_1$.

If the pair $(\Lambda_0\sqcup \Lambda_1,\lambda)$ is non-degenerate, the exact Lagrangian cobordism defines a cobordism map (in the category of the persistence modules) from the persistence module associated to $(\Lambda=\Lambda_0\sqcup\Lambda_1,s_+\lambda)$ to the one associated to $(\Lambda=\Lambda_0\sqcup\Lambda_1,s_-\lambda)$. These persistence modules are multiplicative shifts of each other and the cobordism map is the multiplicative shift by $s_+/s_-$.

If the pair $(\Lambda_0\sqcup\Lambda_1,\lambda)$ is weakly homologically bonded (and, in particular, if it is homologically bonded), then the cobordism map is not the zero morphism between persistence modules and the pseudo-holomorphic curves used to define the map can be also used to prove that a version of the Poisson bracket invariant of quadruples of sets is positive for the following quadruple of sets: $\Lambda_0\times [s_-,s_+]$, $\Lambda_1\times [s_-,s_+]$, $\Sigma\times s_-$, $\Sigma\times s_+$. This is the key result in the paper --
see Section~\ref{sec-pb} for the precise definition of the Poisson bracket invariant (it is a version of the invariant defined previously in \cite{BEP,EP-tetragons})
and Theorem~\ref{thm-lower-bound-on-pb-via-pers-modules} for the precise statement of the result. If the pair $(\Lambda_0\sqcup\Lambda_1,\lambda)$ is degenerate but still weakly homologically bonded, the same result is obtained using a semi-continuity property of the Poisson bracket invariant.

The existence of chords for {\sl time-independent} symplectic Hamiltonians as in Theorem~\ref{thm-main-R2n} follows then from the positivity of the Poisson bracket invariant by Fathi's dynamical Urysohn lemma (see Theorem~\ref{thm-fathi} for its statement).

If the pair $(\Lambda_0\sqcup\Lambda_1,\lambda)$ is stably homologically bonded, then a similar argument yields the existence of chords as in part B of Theorem~\ref{thm-main-R2n} for {\sl time-dependent} Hamiltonians.

The existence of a chord from $\Lambda_0$ to $\Lambda_1$ for a contact Hamiltonian cooperative with $\Lambda_0$, $\Lambda_1$, as in Section~\ref{subsubsec-sample-appls-chords-of-contact-hams-from-one-Leg-to-another}, is deduced then from the existence of the chords of the corresponding homogeneous symplectic Hamiltonian on the symplectization of $\Sigma$.

\begin{rem}
\label{rem-estimate-in-main-thms-on-Ham-Reeb-chords-robust}
{\rm
The claims of Theorem~\ref{thm-main-R2n} (and its analogues for other homologically bonded pairs) on the existence of Hamiltonian chords remain true if $X_0$, $X_1$ are perturbed as exact Lagrangian cobordisms (cylindrical near the boundaries) in the trivial exact symplectic cobordism $\{ (p,q)\in\R^{2n}\ |\ s_0\leq |p|\leq s_1\}$
so that the Legendrian isotopies, induced on the boundaries by the exact Lagrangian isotopies, are sufficiently small -- e.g. sufficiently $C^\infty$-small. (Note that away from the boundary the perturbations may be arbitrarily large, as long as the perturbed $X_0$, $X_1$ are disjoint!). The upper bound on the time-length of the Hamiltonian chords between the perturbed $X_0$, $X_1$ is then only slightly larger than the one for the original $X_0,X_1$.

Similarly, the claims of the results in Section~\ref{subsec-sample-appls-contact-dynamics}
remain true if the Legendrian submanifolds $\Lambda_0$, $\Lambda_1$ are perturbed by sufficiently $C^\infty$-small Legendrian isotopies into Legendrian submanifolds $\Lambda'_0$, $\Lambda'_1$. The upper bound on the time-length of the chord between $\Lambda'_0$, $\Lambda'_1$ is then only slightly larger than the one for the original $\Lambda_0$, $\Lambda_1$ and tends to it as the sizes of the Legendrian isotopies tend to zero.

Let us also remark that the scheme of the proof can be extended to a more general setting and, in particular, to non-trivial exact Lagrangian cobordisms; one can also use the linearized Legendrian contact homology instead of the full one \cite{EP-big}.

For more details and a reference to the proofs see more general Remarks~\ref{rem-l-min-s-robust}, \ref{rem-estimate-on-pb-in-thm-lower-bound-on-pb-via-pers-modules-robust}, \ref{rem-existence-of-chords-of-h-from-Lambda0-to-Lambda1-robust}.
}
\end{rem}
\bigskip

\subsection{Plan of the paper}
\label{subsec-plan}

Let us outline the plan of the paper.

In Section~\ref{sec-pb} we define a Poisson bracket invariant of a quadruple of sets (a modified version of the invariant defined previously in
\cite{BEP,EP-tetragons}) and state a recent theorem of Fathi (a generalization of the result in \cite{Fathi2016}), {which allows to deduce the existence of a Hamiltonian chord from the positivity of the invariant}.

In Section~\ref{sec-persistence-modules} we recall basic facts about persistence modules.

In Section~\ref{sec-lch-intro} we describe the Legendrian contact homology setting that we need and explain how to associate a persistence module, and a corresponding barcode, to a (non-degenerate) pair formed by a Legendrian submanifold and a contact form. Then we show how the existence of bars of sufficiently large multiplicative length in the barcode implies the positivity of the Poisson bracket invariant for appropriate quadruples of sets, which in turn yields the existence of the wanted chords.

In Section~\ref{sec-applications-to-contact-dynamics} we discuss applications of the result proved in Section~\ref{sec-lch-intro} to contact dynamics.

In Sections~\ref{sec-the-1-jet-space},\ref{sec-the-unit-cotangent-bundle-of-Rn} we explain how the results of Section~\ref{sec-lch-intro} can be applied to the Legendrian submanifolds in {$J^1 Q$ and} $ST^* \R^n$, which yields the results of {Sections~\ref{subsec-sample-appls-sympl-hams}, \ref{subsec-sample-appls-contact-dynamics}.}

\bigskip
\noindent

{\bf Acknowledgements:} We thank F.Bourgeois, B.Chantraine, G.Dimi\-tro\-glou Rizell, T.Ekholm, Y.Eliashberg, E.Giroux, M.Sullivan for useful discussions. We also thank
A.Fathi for communicating to us his unpublished note \cite{Fathi2019} containing Theorem~\ref{thm-fathi} and its proof. We thank T.Melistas, J.Zhang, and the anonymous referee
for numerous corrections, and G.Dimitroglou Rizell for finding a mistake in the original version of Theorem \ref{thm-1-jet-space-0-section-and-its-Reeb-shift-homol-bonded}(ii).

\section{A modified $pb^+$-invariant and Hamiltonian chords}
\label{sec-pb}

In this section we discuss a Poisson bracket invariant of quadruples of sets in a symplectic manifold. The proof of the result relating the Poisson bracket invariant to the existence of Hamiltonian chords is based on the theorems of A.Fathi in general smooth/topological dynamics \cite{Fathi2016}, \cite{Fathi2019} and is {similar} to the proof of {the relevant} results in \cite{EP-tetragons} (with the only difference that we can now use the results from \cite{Fathi2019} that were unavailable when \cite{EP-tetragons} was written). Let us recall these results of Fathi.

\subsection{Chords of smooth vec\-tor fields}
\label{subsec-existence-of-chords-for-smooth-vect-fields-Fathi}

In this section let $M$ be any smooth manifold {(without boundary)}, $v$ a complete smooth time-independent
vector field on $M$ (meaning that its flow is defined for all times) and $X_0,X_1 \subset M$ disjoint closed subsets of $M$.
Denote by $T(X_0,X_1;v)$ the infimum of the time-lengths of the chords (that is, integral trajectories) of $v$ from $X_0$ to $X_1$. If there is no such chord, set $T(X_0,X_1;v):=+\infty$.

The following theorem (a ``dynamical Urysohn lemma") was proved by A.Fathi in \cite{Fathi2016} in the case when $X_0,X_1$ are compact and in \cite{Fathi2019} in the case when $X_0,X_1$ are arbitrary closed sets.

\begin{thm}[A.Fathi, \cite{Fathi2016}, \cite{Fathi2019}]
\label{thm-fathi}
Assume $T> T(X_0,X_1;v)$.

Then there exists a smooth function $F: M\to \R$ such that $F|_{X_0}\leq 0$, $F|_{X_1}>1$, and $L_v F < 1/T$.

If $X_0,X_1$ are compact, then $F$  can be chosen to be compactly supported.
\Qed
\end{thm}

Define
\begin{equation}\label{eq-S-vsp}
\cS (X_0,X_1) := \{\ F\in C^\infty (M)\ |\ F|_{X_0}\leq 0,\ F|_{X_1}\geq 1\ \},\end{equation}
\begin{multline}\label{eq-Sprim-vsp}
\cS' (X_0,X_1) := \\ \{\ F\in C^\infty (M)\ |\ \Image (F)\subset [0,1],\ F|_{\Op (X_0)} = 0,\ F|_{\Op (X_1)} = 1\ \},\end{multline}
where $\Op (\cdot)$ denotes some open neighborhood of a set.
Define
\[
\cK (X_0,X_1):= \cS (X_0,X_1)\cap C^\infty_c (M),\ \cK' (X_0,X_1):= \cS' (X_0,X_1)\cap C^\infty_c (M),
\]
where $C^\infty_c (M)$ is the space of compactly supported smooth functions on $M$.

Define
\[
L (X_0,X_1;v) := \inf_{F\in \cS (X_0,X_1)} \sup_M L_v F,
\]
and
\[
L_c (X_0,X_1;v) := \inf_{F\in \cK (X_0,X_1)} \sup_M L_v F.
\]
If $X_1$ is non-compact, the set $K(X_0,X_1)$ is empty. Our convention
is that $\inf \emptyset = +\infty$.
Clearly,
\[
L (X_0,X_1;v)\leq L_c (X_0,X_1;v).
\]

\begin{prop}
\label{prop-L-vs-L-c}
In the definition of $L (X_0,X_1;v)$ and $L_c (X_0,X_1;v)$ one can replace $\cS (X_0,X_1)$ and $\cK (X_0,X_1)$ respectively by
$\cS' (X_0,X_1)$ and $\cK' (X_0,X_1)$:
\[
L (X_0,X_1;v) = \inf_{F\in \cS' (X_0,X_1)} \sup_M L_v F,
\]
and if $X_0$ and $X_1$ are compact,
\[
L_c (X_0,X_1;v) = \inf_{F\in \cK' (X_0,X_1)} \sup_M L_v F.
\]
\end{prop}

\bigskip
\noindent
{\bf Proof of Proposition~\ref{prop-L-vs-L-c}:}
Let us prove the claim for $L (X_0,X_1;v)$ -- the case of $L_c (X_0,X_1;v)$ is similar.

Clearly, $L (X_0,X_1;v) \leq \inf_{F\in \cS' (X_0,X_1)} \sup_M L_v F$ and  thus it suffices to prove that
\begin{equation}
\label{eqn-L-cS-cs-prime}
\inf_{F\in \cS' (X_0,X_1)} \sup_M L_v F \leq L (X_0,X_1;v).
\end{equation}

Let $\delta>0$. Pick a non-decreasing smooth function $\chi:\R\to\R$ so that $\sup_{t\in\R} \chi'(t) \leq 1+\delta$ and for some $\epsilon >0$ we have $\chi (t) = 0$ on $(-\infty,\epsilon]$ and $\chi (t) = 1$ on $[1-\epsilon,+\infty)$.

Then $\chi\circ F\in \cS' (X_0,X_1)$ for any $F\in \cS (X_0,X_1)$ and
\[
L_v (\chi\circ F) = (\chi'\circ F) L_v F \leq (1+\delta) L_v F.
\]
Taking the infimum over $F\in \cS (X_0,X_1)$ in both sides we get
\[
\inf_{F\in \cS' (X_0,X_1)} \sup_M L_v F \leq (1+\delta) L (X_0,X_1;v).
\]
Since this is true for any $\delta>0$, we obtain \eqref{eqn-L-cS-cs-prime} and this finishes the proof of the proposition.
\Qed
\bigskip

The following corollary follows readily from Theorem~\ref{thm-fathi}.

\begin{cor}
\label{cor-fathi-existence-of-chords}
$T(X_0,X_1;v) = 1/L (X_0,X_1;v)$.

Consequently, if $L (X_0,X_1;v) >0$, then for any $\epsilon >0$ there exists a chord of $v$ from $X_0$ to $X_1$
of time-length $\leq 1/L (X_0,X_1;v) + \epsilon$ (if $\supp v$ is compact, then one can drop $\epsilon$ from the bound).

If $X_0,X_1$ are compact, then $L (X_0,X_1;v) = L_c (X_0,X_1;v)$ and $T(X_0,X_1;v) = 1/L_c (X_0,X_1;v)$.

Consequently, if $X_0$ and $X_1$ are compact and $L_c (X_0,X_1;v) >0$, then there exists a chord of $v$ from $X_0$ to $X_1$
of time-length $\leq 1/L_c (X_0,X_1;v)$.
\Qed
\end{cor}
\bigskip

\subsection{Poisson bracket invariant}
\label{subsec-pb-invariant}

Let $(M,\omega)$ be a (not necessarily compact) connected symplectic manifold, possibly with boundary.

We use the following sign conventions in the definitions of a
Hamiltonian vector field and the Poisson bracket on $M$: the
Hamiltonian vector field $\sgrad H$ of a Hamiltonian $H$ is
defined by
\[
i_{\sgrad {H}}
\omega = -dH
\]
and the Poisson bracket of two Hamiltonians $F$, $G$ is given by
\[
\{F,G\} := \omega(\sgrad G,\sgrad F) = dF (\sgrad G) = - dG (\sgrad F) =
\]
\[
= L_{\sgrad G} F = - L_{\sgrad F} G.
\]

Assume $X_0, X_1, Y_0, Y_1$ are closed subsets of $M$ such that
\[
X_0\cap X_1=Y_0\cap Y_1=\emptyset.
\]

Such a collection of sets $X_0, X_1, Y_0, Y_1$ will be called {\it an admissible quadruple}.

Consider the following conditions on pairs $(F,G)\in C^\infty (M)\times C^\infty (M)$:

\bigskip
\noindent
(1) $\sup_M \{F,G\} < +\infty$ and the vector field $F\sgrad G$ on $M$ is complete.

\bigskip
\noindent
(1') $\supp F$ is compact and $\supp (F\sgrad G)$ lies in the interior of $M$.

\bigskip
\noindent
(2) $F|_{X_0}\leq 0,\ F|_{X_1} \geq 1$, $G|_{Y_0}\leq 0,\ G|_{Y_1}\geq 1$.

\bigskip
\noindent
(2') $F|_{X_0}\leq 0,\ F|_{X_1} \geq 1$, $G|_{\Op (Y_0)}\equiv 0,\ G|_{\Op (Y_1)}\equiv 1$. (Here $\Op (\cdot)$ denotes some open neighborhood of a set).

\bigskip
\noindent
(2'') $F|_{\Op (X_0)}\equiv 0,\ F|_{\Op (X_1)} \equiv 1$, $G|_{\Op (Y_0)}\equiv 0,\ G|_{\Op (Y_1)}\equiv 1$.

\bigskip
\noindent
(3) $\Image F\subset [0,1]$.

\bigskip
Note that (1') $\Rightarrow$ (1), since $\supp \{F,G\} \subset\supp (FdG) = \supp (F\sgrad G)$.

Define:
\[
\cF_M (X_0, X_1, Y_0, Y_1) := \{\ (F,G)\in C^\infty (M)\times C^\infty (M)\ |\ (F,G)\ \textrm{satisfies (1) and (2)}\ \},
\]
\[
\cF'_M (X_0, X_1, Y_0, Y_1) := \{\ (F,G)\in C^\infty (M)\times C^\infty (M)\ |\ (F,G)\ \textrm{satisfies (1),(2) and (3)}\ \},
\]
\[
\cF''_M (X_0, X_1, Y_0, Y_1) := \{\ (F,G)\in C^\infty (M)\times C^\infty (M)\ |\ (F,G)\ \textrm{satisfies (1),(2'') and (3)}\ \},
\]
\[
\cG_M (X_0, X_1, Y_0, Y_1) := \{\ (F,G)\in C^\infty (M)\times C^\infty (M)\ |\ (F,G)\ \textrm{satisfies (1') and (2)}\ \},
\]
\[
\cG'_M (X_0, X_1, Y_0, Y_1) := \{\ (F,G)\in C^\infty (M)\times C^\infty (M)\ |\ (F,G)\ \textrm{satisfies (1'), (2')}\ \},
\]
\[
\cG''_M (X_0, X_1, Y_0, Y_1) := \{\ (F,G)\in C^\infty (M)\times C^\infty (M)\ |\ (F,G)\ \textrm{satisfies (1'), (2'') and (3)}\ \}.
\]
For brevity we will omit the sets $X_0, X_1, Y_0, Y_1$ from this notation, when needed.

Clearly,
\[
\cG''_M\subset\cG'_M \subset \cG_M\subset \cF_M,\ \cF''_M \subset \cF'_M\subset \cF_M,
\]
\[
\ \cG'_M\subset \cF'_M,\ \cG''_M\subset \cF''_M.
\]
It is also easy to see that the sets $\cG_M$, $\cG'_M$ and $\cG''_M$ are non-empty if only if $X_1$ is compact.

Set
\[
pb^+_M (X_0,X_1,Y_0,Y_1) := \inf\limits_{\cF_M} \sup_M \{F,G\},
\]
\[
pb^+_{M,comp} (X_0,X_1,Y_0,Y_1) := \inf\limits_{\cG_M} \max_M \{F,G\}.
\]
If the set over which the infimum is taken is empty, we set the infimum to be $+\infty$.

Clearly,
\[
pb^+_M (X_0,X_1,Y_0,Y_1)\leq pb^+_{M,comp} (X_0,X_1,Y_0,Y_1).
\]

\begin{rem}
\label{rem-pb-invariants-history}
{\rm
The quantities $pb^+_M$, $pb^+_{M,comp}$ are versions of the $pb_4$ invariant of quadruples of sets defined originally in \cite{BEP} (where the $C^0$-norm of $\{F,G\}$ was used instead of $\max_M \{F,G\}$) and of the $pb^+_4$ invariant defined in \cite{EP-tetragons} (where the sets $X_0,X_1,Y_0,Y_1$ were assumed to be compact and both $F$ and $G$ were assumed to be compactly supported). See also \cite{Ganor} for a result that allows to define $pb^+_4 (X_0,X_1,Y_0,Y_1)$ in terms of the topology of the set $X_0\cup X_1\cup Y_0\cup Y_1$. Note that, unlike $pb_4$ and $pb^+_4$, the invariants $pb^+_M$, $pb^+_{M,comp}$ are {\sl not} symmetric with respect to the permutation $(X_0,X_1,Y_0,Y_1)\mapsto (Y_0,Y_1,X_1,X_0)$.
}
\end{rem}
\bigskip

Similarly to Proposition~\ref{prop-L-vs-L-c} (also see \cite{BEP}), one can prove that the sets $\cF_M$, $\cG_M$ in the definitions of $pb^+_M$, $pb^+_{M,comp}$ can be replaced, respectively, by $\cF'_M$, $\cF''_M$ and by $\cG'_M$, $\cG''_M$:
\[
pb^+_M (X_0,X_1,Y_0,Y_1) = \inf\limits_{\cF'_M} \sup_M \{F,G\} = \inf\limits_{\cF''_M} \sup_M \{F,G\},
\]
\[
pb^+_{M,comp} (X_0,X_1,Y_0,Y_1) = \inf\limits_{\cG'_M} \max_M \{F,G\} = \inf\limits_{\cG''_M} \max_M \{F,G\}.
\]

We will need the following basic properties of $pb^+_M$, $pb^+_{M,comp}$.

\bigskip
\noindent{\sc Monotonicity:}

\begin{prop}[cf. \cite{BEP}, \cite{EP-tetragons}]
\label{prop-monotonicty-of-pb}
Assume that $M$ is a codimension-zero submanifold (with boundary) of a symplectic manifold $N$ (without boundary),
which is closed as a subset of $N$. Let $U\subset N$ be an open set. Assume that $X_0,X_1,Y_0,Y_1$ is an admissible quadruple lying in $M$, so that $X_0,X_1\subset U\cap M$, $Y_0\cap U, Y_1\cap U\neq\emptyset$ and $\partial M\subset Y_0\cup Y_1$.

Then
\begin{equation}
\label{eqn-pb-monotonicity-1}
pb^+_N (X_0,X_1,Y_0,Y_1)\geq pb^+_M (X_0,X_1,Y_0,Y_1),
\end{equation}
\begin{equation}
\label{eqn-pb-monotonicity-1-comp}
pb^+_{N,comp} (X_0,X_1,Y_0,Y_1)\geq pb^+_{M,comp} (X_0,X_1,Y_0,Y_1),
\end{equation}
\begin{equation}
\label{eqn-pb-monotonicity-2}
pb^+_{M\cap U,comp} (X_0,X_1,Y_0\cap U,Y_1\cap U)\geq pb^+_{M,comp} (X_0,X_1,Y_0,Y_1).
\end{equation}

\end{prop}

\begin{proof}
If $(F,G)\in \cF''_N (X_0, X_1, Y_0, Y_1)$, then it follows easily from the definitions that $(F|_M, G|_M)\in \cF''_M (X_0, X_1, Y_0, Y_1)$. This yields \eqref{eqn-pb-monotonicity-1}. The inequality \eqref{eqn-pb-monotonicity-1-comp} follows similarly.

Let us prove \eqref{eqn-pb-monotonicity-2}. Assume $(F,G)\in \cG''_{M\cap U} (X_0, X_1, Y_0\cap U, Y_1\cap U)$. In particular, this means that
$\supp F\subset M\cap U$ is compact and $G$ is
equal to $0$ and $1$ on some open neighborhoods (in $U$) of, respectively, $Y_0\cap U$ and $Y_1\cap U$. Extend $F$ by zero outside $M\cap U$ to a smooth compactly supported function $\tF: M\to [0,1]$ and extend $G$ to a smooth function $\tG: M\to\R$ so that $\tG$ is equal to $0$ and $1$ on some open neighborhoods (in $M$) of, respectively, $Y_0$ and $Y_1$. Then $(\tF, \tG)\in \cG''_M (X_0, X_1, Y_0, Y_1)$, while $\{ \tF,\tG\} = \{ F,G\}$ (because outside $U$ both Poisson brackets vanish, while on $U$ they coincide since $\tF|_U = F$, $\tG|_U = G$). This immediately yields \eqref{eqn-pb-monotonicity-2}.
\end{proof}

The following property follows from the definitions (cf. \cite{BEP}, \cite{EP-tetragons}).

\bigskip
\noindent{\sc Semi-continuity:}

Suppose that $(X_0, X_1, Y_0, Y_1)$ is an admissible quadruple in $(M,\omega)$, $X_0,X_1$ are compact, and
$\{ X_0^{(j)}\}, \{ X_1^{(j)}\}$, $j\in\N$, are sequences of compact subsets of $M$ converging (in the sense of the Hausdorff distance between sets)
respectively to $X_0$, $X_1$, so that the quadruples $(X_0^{(j)}, X_1^{(j)}, Y_0, Y_1)$ are admissible for all $j\in\N$.

Then
\begin{equation}
\label{eqn-Hausdorff-convergence}
\limsup_{j\to +\infty} pb^+_M (X_0^{(j)}, X_1^{(j)}, Y_0, Y_1) \leq pb^+_M (X_0,X_1,Y_0,Y_1),
\end{equation}
\begin{equation}
\label{eqn-Hausdorff-convergence-comp}
\limsup_{j\to +\infty} pb^+_{M,comp} (X_0^{(j)}, X_1^{(j)}, Y_0, Y_1) \leq pb^+_{M,comp} (X_0,X_1,Y_0,Y_1).
\end{equation}

The next proposition is proved as in \cite{EP-tetragons} using Corollary~\ref{cor-fathi-existence-of-chords}.

\begin{prop}
\label{prop-pb4+-modified-chords}
Assume that $M$ is a codimension-zero submanifold (with boundary) of a symplectic manifold $N$ (without boundary), so that
$M$ is closed as a subset of $N$. Assume that $X_0,X_1,Y_0,Y_1$ is an admissible quadruple lying in $M$, so that $\partial M\subset Y_0\cup Y_1$.

Let $H: N\to\R$ be a complete time-independent Hamiltonian. Then the following claims hold:

\bigskip
\noindent
I. Assume $\Delta(H;Y_0,Y_1)=:\Delta>0$. If $pb^+_M (X_0,X_1,Y_0,Y_1) =: p >0$, then for any $\epsilon >0$ there exists a chord of $H$ from $X_0$ to $X_1$ of time-length $\displaystyle \leq \frac{1}{p\Delta} + \epsilon$. If $X_0$, $X_1$ are compact and $pb^+_{M,comp} (X_0,X_1,Y_0,Y_1) =: p_{comp} >0$, then there exists a chord of $H$ from $X_0$ to $X_1$ of time-length $\displaystyle \leq \frac{1}{p_{comp}\Delta}$.

\bigskip
\noindent
II.
Assume that $X_0$, $X_1$ are compact, $\supp H\cap M$ is compact, $H|_{X_0}\geq 0$ and $H|_{X_1}\leq -\Delta$ for some $\Delta>0$.
Assume also that $pb^+_{M,comp} (X_0,X_1,Y_0,Y_1) =: p_{comp} >0$. Then there exists a chord of $H$ from $Y_0$ to $Y_1$ of time-length $\displaystyle \leq \frac{1}{p_{comp}\Delta}$.

\end{prop}
\bigskip

\bigskip
\noindent
{\bf Proof of Proposition~\ref{prop-pb4+-modified-chords}:}
Let us prove part I. We may assume without loss of generality that $H|_{Y_0}\leq 0$, $H|_{Y_1}\geq 1$ (this can be always achieved by replacing $H$ with $aH+b$ for some $a,b\in\R$, $a\neq 0$).
For any $F\in \cS (X_0,X_1)$ (see \eqref{eq-S-vsp}) satisfying $\sup_N L_{\sgrad H} F = \sup_N \{ F,H\} < +\infty$ we have $(F,H)\in \cF'_N$, and if $\supp F$ is compact, then $(F,H)\in \cG_N$. Indeed, since the vector field $\sgrad H$ is complete, then so is the vector field $F\sgrad H$, since $0\leq F\leq 1$.
Hence, by \eqref{eqn-pb-monotonicity-1} and \eqref{eqn-pb-monotonicity-1-comp},
\[
\sup_N L_{\sgrad H} F = \sup_N \{ F,H\}\geq pb^+_N (X_0,X_1,Y_0,Y_1) \geq
\]
\[
\geq
pb^+_M (X_0,X_1,Y_0,Y_1),
\]
or, if $\supp F$ is compact,
\[
\sup_N L_{\sgrad H} F = \sup_N \{ F,H\}\geq
\]
\[
\geq
pb^+_{N,comp} (X_0,X_1,Y_0,Y_1) \geq pb^+_{M,comp} (X_0,X_1,Y_0,Y_1).
\]
Taking the infimum over all such $F$, we get
\[
L (X_0,X_1;\sgrad H)\geq pb^+_M (X_0,X_1,Y_0,Y_1),
\]
and if $X_0,X_1$ are compact,
\[
L_c (X_0,X_1;\sgrad H)\geq pb^+_{M,comp} (X_0,X_1,Y_0,Y_1).
\]
Now the claims of part I follow from Corollary~\ref{cor-fathi-existence-of-chords}.

Let us prove part II. We may assume without loss of generality that $H|_{X_0}\geq 0$, $H|_{X_1}\leq -1$ (this can be always achieved by replacing $H$ with $H/\Delta$).
For any $G\in \cS' (Y_0,Y_1)$ (see \eqref{eq-Sprim-vsp}) we have $(-H|_M,G|_M)\in \cG'_M$. Recall that here $\supp H\cap M$ is assumed to be compact and $G$ is constant near $Y_0$ and $Y_1$, since $G\in \cS' (Y_0,Y_1)$.
Hence,
\[
\sup_N L_{\sgrad H} G\geq \sup_M L_{\sgrad H} G = \sup_M \{ -H,G\}\geq pb^+_{N,comp} (X_0,X_1,Y_0,Y_1).
\]
Taking the infimum over all $G\in \cS' (Y_0,Y_1)$ and using Proposition~\ref{prop-L-vs-L-c}, we get
\[
L (Y_0,Y_1;\sgrad H)\geq pb^+_{N,comp} (X_0,X_1,Y_0,Y_1)\geq
\]
\[
\geq pb^+_{M,comp} (X_0,X_1,Y_0,Y_1) = p_{comp}.
\]
Now the claim of part II follows from Corollary~\ref{cor-fathi-existence-of-chords}.
\Qed

Let us now discuss an implication of Proposition~\ref{prop-pb4+-modified-chords} for the existence of chords of time-dependent Hamiltonians.

Let $E>0$.
Let $r\in (-E,E)$ and $\tau\in\SP^1 = \R/\Z$ be the coordinates, respectively, on $(-E,E)$ and $\SP^1$. Set
\[
\tM_E := M\times (-E,E)\times \SP^1
\]
and equip $\tM_E$ with the product symplectic form $\omega\oplus dr\wedge d\tau$.
Let $X_0,X_1,Y_0,Y_1\subset M$ be an admissible quadruple, such that $X_0,X_1$ are compact.
Set
\[
\tX_0:= X_0\times \{ r=0\},\  \tX_1:= X_1\times  \{ r=0\},
\]
\[
\tY_0 (E) := Y_0\times (-E,E)\times \SP^1,\ \tY_1 (E) := Y_1\times (-E,E)\times \SP^1.
\]

\begin{prop}
\label{prop-chords-of-time-dep-Ham-s}
Assume that $\partial M = \emptyset$. With an admissible quadruple $X_0,X_1,Y_0,Y_1\subset M$ as above,
let $H: M\times \SP^1\to \R$ be a complete Hamiltonian and $\{ \phi_t\}_{t\in\R}$ its flow. Let $E>0$. Assume that

\bigskip
\noindent
(a) $pb^+_{\tM_E,comp} (\tX_0, \tX_1, \tY_0 (E), \tY_1 (E)) =: \tp_E >0$.

\bigskip
\noindent
(b) $\Delta (H; Y_0, Y_1) =: \Delta > 2E$.

\bigskip
\noindent
(c) $\sup_{t_0\in\R} \big(\sup_{t\in [t_0, t_0+T]} H(\phi_t (x), t) - \inf_{t\in [t_0, t_0+T]} H(\phi_t (x), t)\big) <E$ for any $x\in X_0$, where $\displaystyle T:=\frac{1}{\tp_E (\Delta-2E)}$.

\bigskip
Then there exists a chord of $H$ from $X_0$ to $X_1$ of time-length $\displaystyle \leq T=\frac{1}{\tp_E (\Delta-2E)}$.
\end{prop}

\bigskip
\noindent
{\bf Proof of Proposition~\ref{prop-chords-of-time-dep-Ham-s}:}
In view of (c) we can pick $0<E_1< E_2<E$ so that
\begin{equation}
\label{eqn-E1-E2}
\sup_{t\in [t_0, t_0+T]} H(\phi_t (x), t) - \inf_{t\in [t_0, t_0+T]} H(\phi_t (x), t) < E_1\ \textrm{for all}\ t_0\in\R.
\end{equation}
Pick a smooth cut-off function $\chi:\R\to\R$ such that $\chi (x) = 0$ if $|x|\geq E_2$ and $\chi (x) = 1$ if $|x|\leq E_1$.

Define a {\sl time-independent} Hamiltonian $\tH : \tM_E\to\R$ as $\tH (x,r,\tau) := r + H(x,\tau)$. One easily verifies that
the Hamiltonian $\chi \tH$ is complete
and, in view of (b),
\[
\Delta \left(\chi H; \tY_0 (E), \tY_1 (E)\right) \geq \Delta - 2E >0.
\]
Together with (a), this implies, by part I of Proposition~\ref{prop-pb4+-modified-chords}, that there exists a chord $\gamma$ of $\chi \tH$ from $\tX_0$ to $\tX_1$ of time-length $\displaystyle \leq T=\frac{1}{\tp_E(\Delta-2E)}$.

We claim that $\gamma$ is, in fact, a chord of $\tH$ from $\tX_0$ to $\tX_1$ -- this would imply that the projection of $\gamma$ to $M$ is a chord of
$H$ from $X_0$ to $X_1$ of time-length $\leq T$.

Indeed, note that $\chi\tH = \tH = H$ on $M\times [-E_1,E_1]\times \SP^1\subset \tM_E$ and the projection to $M$ of each trajectory of the Hamiltonian flow of $\tH$ on $\tM_E$ is a trajectory of the Hamiltonian flow of $H$ on $M$. Since the time-independent Hamiltonian $\tH$ is preserved by its own Hamiltonian flow, \eqref{eqn-E1-E2} implies that for any $t_0\in\R$ any time-$[t_0,t_0+T]$ trajectory of the Hamiltonian flow of $\tH$ passing at some moment $t\in [t_0,t_0+T]$ through $\tX_0$ stays in $M\times [-E_1,E_1]\times \SP^1$ for all $t\in [t_0,t_0+T]$. Therefore for any $t_0\in\R$ any time-$[t_0,t_0+T]$ trajectory of the Hamiltonian flow of $\chi\tH$ passing at some moment $t\in [t_0,t_0+T]$ through $\tX_0$ is, in fact, a trajectory of the Hamiltonian flow of $\tH$ for all $t\in [t_0,t_0+T]$. In particular, $\gamma$ is a chord of $\tH$ from $\tX_0$ to $\tX_1$, which proves the claim and finishes the proof of the proposition.
\Qed

\bigskip

\begin{rem}
\label{rem-chords-of-time-dep-Hams-for-noncompact-X0-X1}
{\rm
Proposition~\ref{prop-chords-of-time-dep-Ham-s} admits an analogue for the case when $X_0,X_1$ are not necessarily compact -- in that case one should replace the quantity $pb^+_{\tM_E,comp} (\tX_0, \tX_1, \tY_0 (E), \tY_1 (E))$ in the claim by $pb^+_{\tM_E} (\tX_0, \tX_1, \tY_0 (E), \tY_1 (E)) =: pb^+_{\tM_E}$ and then there would exist a chord of $H$ from $X_0$ to $X_1$ of time-length $\displaystyle \leq T=\frac{1}{pb^+_{\tM_E} (\Delta-2E)}$. The proof of this claim virtually repeats the proof of Proposition~\ref{prop-chords-of-time-dep-Ham-s}.
}
\end{rem}

\section{Persistence modules}
\label{sec-persistence-modules}

In this section we recall basic facts about persistence modules. For a more detailed introduction to persistence modules see e.g. \cite{Edelsbrunner-Harer}, \cite{Chazal-et-al-book}, \cite{Oudot-AMS2015-book}, or \cite{Polt-Ros-Samv-Zhang}.

We work over the base field $\Z_2$.

Let {$\I := (a,+\infty)$, $-\infty\leq a< +\infty$. (In fact, we will be concerned only with $\I = (0,+\infty)$ and $\I = (-\infty,\infty)$).}

\begin{defin}
\label{def-persistence-module}
{\rm
A {\it persistence module over $\I$} is given by a pair
\[
\left(V=\{ V_t\}_{t\in \I}, \pi = \{ \pi_{s,t}\}_{s,t\in \I, s\leq t}\right),
\]
where {all $V_t$, $t\in \I$, are finite-dimensional $\Z_2$-vector spaces}
and $\pi_{s,t} : V_s \to V_t$ are linear maps, so that

\smallskip
\noindent
(i) (Persistence) $\pi_{t,t} = Id$, $\pi_{s,r} = \pi_{t,r} \circ \pi_{s,t}$, for all $s, t, r \in \I$, $s\leq t\leq r$.

\smallskip
\noindent
(ii) (Discrete spectrum and semicontinuity) There exists a (finite or countable) discrete closed set of points
\[
\spec (V) = \left\{ \lmin (V):=t_0<t_1<t_2<\ldots <+\infty\right\}\subset \I,
\]
called {\it the spectrum of $V$}, so that

\begin{itemize}

\item{}
for any $r \in \I\setminus \spec (V)$ there exists a
neighborhood $U$ of $r$ in $\I$ such that $\pi_{s,t}$ is an isomorphism for all $s,t\in U$, $s \leq t$;

\item{}
for any $r\in \spec (V)$ there exists $\epsilon > 0$ such that $\pi_{s,t}$ is an isomorphism of vector spaces
for all $s,t\in (r -\epsilon,r] \cap \I$.

\end{itemize}

\smallskip
\noindent
(iii) (Semi-bounded support) For the smallest point $\lmin (V):=t_0$ of $\spec (V)$ one has $\lmin (V)>a$ and $V_t = 0$ for all $t\leq \lmin (V)$.

}
\end{defin}\bigskip

{In \cite{Polt-Ros-Samv-Zhang} such persistence modules are called {\it persistence modules of finite type}.}

{\it The zero (or trivial) persistence module (over $\I$)} is a persistence module formed by zero vector spaces and trivial maps between them.

The notions of a persistence submodule, the direct sum of persistence modules, a morphism/isomorphism between persistence modules (over $\I$) are defined in a straightforward manner.

Recall that two morphisms $\Phi_i: V_i \to W_i$, $i=1,2$  between per\-sis\-tence mo\-du\-les
are called {\it right-left equivalent} (or, for brevity, simply equivalent) if there exist isomorphisms $\Psi: V_1 \to V_2$ and $\Theta: W_1 \to W_2$ such that $\Theta\Phi_1 = \Phi_2\Psi$.

\begin{exam}
\label{exam-extension-of-pers-modules}
{\rm
A persistence module $V$ over $(0,+\infty)$ can be extended to a persistence module over $(-\infty,+\infty)$ by setting $V_s=0$ and $\pi_{s,t}=0$ for all $s\in (-\infty,0]$.
}
\end{exam}
\bigskip

\begin{exam}
\label{exam-interval-module}
{\rm
Let {$\J \subset \I$ be either of the form $(a_\J, b_\J]$ for some $0<a_\J<b_\J < +\infty$, or of the form $(a_\J, +\infty)$}. Define an {\it interval (persistence) module}
\[
(Q(\J), \pi) := \left(\{ Q(\I)\}_{t\in\I} , \{\pi_{s,t}\}_{s,t\in\I,s\leq t}\right),
\]
over $\I$
as follows: $Q(\J)_t = \Z_2$ for $t\in \J$ and $Q(\J)_t = 0$ for $t\in \I\setminus\J$,
while the morphisms $\pi_{s,t}$ are the identity maps for $s,t\in\J$ and zeroes otherwise.
}
\end{exam}
\bigskip

The following structure theorem for persistence modules can be found in
\cite{ZC}, \cite[Thm. 2.7,2.8]{Chazal-et-al-book}, \cite{Crawley-Boevey}.
Its various versions appeared prior to invention of persistence modules,
see \cite{Azumaya,Gabriel,Auslander,Ringel-Tachikawa,Webb,Baran}.

\begin{thm}
\label{thm-decomposition-persist-modules}
For every persistence
module $(V,\pi)$ over $\I$ there exists a unique (finite or countable) collection of intervals
$\J_j \subset \I$ -- where the intervals may not be distinct but each interval appears in the collection only finitely many times -- so that
$(V,\pi)$ is isomorphic to $\oplus_j Q(\J_j)$:
\[
(V,\pi)=\oplus_j Q(\J_j).
\]
\end{thm}

The collection $\{ \J_j\}$ of the intervals is called the {\it barcode} of $V$. The intervals $\J_j$ themselves are called the {\it  bars of $V$}.
Note that the same bar may appear in the barcode several (but finitely many) times (this number of times is called the {\it multiplicity} of the bar) -- in other words, a barcode is a multiset of bars.

The barcode of the trivial persistence module is empty. Set
\[
V_\infty := (\Z_2)^k,
\]
where $k\in \N\cup \{0,+\infty\}$ is the number of the infinite bars in the barcode of $V$.

\begin{exam}
\label{exam-add-shift}
{\rm
Let $\I = (-\infty,+\infty)$ and let $V$ be a persistence module over $\I$.
Let $c\in\R$.

Define a new persistence module $V^{[+c]}$ over $\I$
by adding $c$ to all indices of $V_t$ and $\pi_{s,t}$ -- in particular,
\[
V^{[+c]}_t:= V_{t+c}.
\]
The barcode of $V^{[+c]}$ is the barcode of $V$ shifted by $c$ to the left.

If $c\geq 0$, one can also define a morphism
\[
\Sh_V [+c]: V\to V^{[+c]},
\]
called {\it the additive shift of $V$ by $c$}, as follows:
\[
\Sh_V [+ c] := \left\{ \pi_{t,t+c}: V_t\to V^{[+c]}_t\right\}_{t\in \I}.
\]
}
\end{exam}
\bigskip

\begin{exam}
\label{exam-mult-shift}
{\rm
Let $\I = (0,+\infty)$ and let $V$ be a persistence module over $\I$.
Let $c>0$.

Let $V= (V,\pi)$ be a persistence module. Define a new persistence module $V^{[\times c]}$ over $\I$
by multiplying all the indices of $V_t$ and $\pi_{s,t}$ by $c$ -- in particular,
\[
V^{[\times c]}_t:= V_{ct}.
\]
The barcode of $V^{[\times c]}$ is the barcode of $V$ divided by $c$.

If $c\geq 1$, one can also define a morphism
\[
\Sh_V [\times c]: V\to V^{[\times c]},
\]
called {\it the multiplicative shift of $V$ by $c$}, as follows:
\[
\Sh_V [\times c] := \left\{ \pi_{t,ct}: V_t\to V^{[\times c]}_t\right\}_{t\in \I}.
\]
}
\end{exam}
\bigskip

\begin{rem}
\label{rem-shifts-isomorphisms}
{\rm
One easily sees that additive/multiplicative shifts of isomorphic persistence modules by the same constant are also isomorphic. Thus, one can speak about additive/multiplicative shifts of isomorphism classes of persistence modules.

}
\end{rem}
\bigskip

\section{Legendrian contact homology}
\label{sec-lch-intro}

Let $(P^{2n},d\vartheta)$, $n\in\Z_{\geq 0}$, be an exact symplectic manifold
with bounded geometry at infinity {(see \cite{ALP} for the definition of this class of manifolds; in particular, this class includes symplectic manifolds that are convex in the sense of \cite{EG}).}
Consider the contactization of $(P,d\vartheta)$: Let
\begin{equation}\label{eq-contactization-vsp}
\Sigma := P\times \R (z),
\end{equation}
and let $\xi=\ker \lambda$ be the contact structure on $\Sigma$ defined by the contact form
\[
\lambda := dz + \vartheta.
\]
Let $\{R_t\}$ be the Reeb flow of $\lambda$ -- each $R_t$ is a shift by $t$ in the coordinate $z$.

\bigskip
{\bf Further on we will always assume that $P$, and consequently $\Sigma$, are connected.}
\bigskip

Consider a compact (not necessarily connected) Legendrian submanifold $\Lambda\subset (\Sigma,\xi)$ without boundary.

If $\Lambda=\Lambda_0\sqcup\Lambda_1$ is a disjoint union of compact (not necessarily connected)
Legendrian submanifolds $\Lambda_0$, $\Lambda_1$ without boundary we call $\Lambda$ {\it a two-part Legendrian submanifold} and $\Lambda_0,\Lambda_1$ {\it the ($0$- and $1$-) parts of $\Lambda$}.

Denote the set of all Reeb chords of $(\Lambda, \lambda)$ by
$\cR (\Lambda, \lambda)$.

If $\Lambda=\Lambda_0\sqcup\Lambda_1$ is a two-part Legendrian submanifold, we say that a Reeb chord of $(\Lambda,\lambda)$ is an {\it $ij$-chord} for $i,j=0,1$ if it starts on $\Lambda_i$ and ends on $\Lambda_j$.
Denote the set of all $ij$-chords of $(\Lambda, \lambda)$ by
$\cR_{ij} (\Lambda, \lambda)$.
The $ii$-chords will be called {\it pure}, while $ij$-chords for $i\neq j$ will be called {\it mixed}.

We say that the pair $(\Lambda,\lambda)$ is {\it non-degenerate} if the following conditions are satisfied:

\begin{itemize}

\item{} For each Reeb chord
$a: [0,T]\to \Sigma$, $a (t) = R_t
(a(0))$, $a(0), a(T)\in\Lambda$,
of $\Lambda$ with respect to $\lambda$,
the tangent spaces of the Legendrian submanifolds $R_T (\Lambda)$ and $\Lambda$ are transversal
inside the contact hyperplane at the point $a(T) = R_T (a(0))$.

\item{} Each trajectory of the Reeb flow $\{ R_t (x)\}$, $-\infty < t < +\infty$, $x\in\Sigma$, intersects $\Lambda$ at most in two points.
    (Equivalently, the images of distinct Reeb chords are disjoint).

\end{itemize}

If $(\Lambda,\lambda)$ is non-degenerate and
$\Lambda$ is compact, then the set $\cR (\Lambda, \lambda)$ is finite.

\subsection{Exact Lagrangian cobordisms}
\label{subsec-exact-Lagr-cobs}

The symplectization of $(\Sigma,\xi)$ can be identified with $\big(\Sigma\times \R_+ (s), d(s\lambda)\big)$.

Let $0 < s_- < s_+$.

Consider the manifold with boundary $\Sigma\times [s_-,s_+]\subset \Sigma\times \R_+$ equipped with the symplectic form $\omega := d(s\lambda)$ -- it is a {\it trivial exact symplectic cobordism} whose {\it positive and negative boundaries} and the restrictions of $s\lambda$ to them are identified, respectively, with $(\Sigma,s_+\lambda)$ and $(\Sigma,s_-\lambda)$.

A differential 1-form $\theta$ on $\Sigma\times [s_-,s_+]$ will be called a {\it cobordism 1-form} if the following conditions are satisfied:

\begin{itemize}

\item $d\theta = \omega$;

\item $\theta$ coincides with $s\lambda$ near the boundaries of $\Sigma\times [s_-,s_+]$;

\end{itemize}

In particular, $s\lambda$ itself is a cobordism 1-form.

Let $\Lambda^\pm\subset (\Sigma,\xi)$ be compact Legendrian submanifolds without boundary, viewed respectively as submanifolds of $\Sigma\times s_\pm$.
A Lagrangian cobordism $L$ in $(\Sigma\times [s_-,s_+],s\lambda)$ between
$\Lambda^\pm$ is a smooth compact cobordism in $\Sigma\times [s_-,s_+]$ between $\Lambda^+\subset \Sigma\times s_+$ and $\Lambda^-\subset \Sigma\times s_-$ which is a Lagrangian submanifold of $(\Sigma\times \R_+, \omega)$ so that there exist $\delta_\pm >0$ for which
\[
L\cap \Sigma\times [s_+ - \delta_+, s_+] = \Lambda^+ \times [s_+ - \delta_+, s_+],
\]
\[
L\cap \Sigma\times [s_-, s_- + \delta_-] = \Lambda^- \times [s_-, s_- + \delta_-].
\]

The sets $L\cap \Sigma\times [s_+ - \delta_+, s_+]$ and $L\cap \Sigma\times [s_-, s_- + \delta_-]$
will be called the {\it positive and the negative collars} of $L$.
The Legendrian submanifolds $\Lambda^-$ and $\Lambda^+$ will be called, respectively, {\it the negative} and {\it the positive boundary of $L$}.

We say that $L$ as above is a {\it two-part Lagrangian cobordism}
if $L$ is a disjoint union of two (not necessarily connected) Lagrangian cobordisms $L_0$ and $L_1$:
\[
L=L_0\sqcup L_1,
\]
where $L_0$ is a Lagrangian cobordism between the Legendrian submanifolds
$\Lambda_0^+ := \Lambda^+ \cap L_0$ and $\Lambda_0^- := \Lambda^- \cap L_0$ and
$L_1$ is a Lagrangian cobordism between the Legendrian submanifolds
$\Lambda_1^+ := \Lambda^+ \cap L_1$ and $\Lambda_1^- := \Lambda^- \cap L_1$.
In particular, $\Lambda^\pm = \Lambda_0^\pm\sqcup \Lambda_1^\pm$ are two-part Legendrian submanifolds.

Note that in our terminology a two-part Lagrangian cobordism includes a numbering of its parts.

Let $\theta$ be a cobordism 1-form.
We say that a two-part Lagrangian cobordism $L=L_0\sqcup L_1$ is {\it $\theta$-exact} if $\theta|_{L_i} = df_i$, $i=0,1$, for a smooth function $f_i: L_i\to\R$
which is zero on the negative collar of $L$ and is identically equal to a constant $C_i$ on the positive collar of $L_i$.

If a two-part Lagrangian cobordism $L$ is $\theta$-exact for some cobordism 1-form $\theta$, we call it just {\it exact}.

The constant $C:=C_1-C_0$ will be called {\it the gap of $L$ with respect to $\theta$}. We will also say that $L$ is {\it $C$-gapped} (with respect to $\theta$).

We say that a two-part Lagrangian cobordism $L$ is {\it belted} if there exists a null-homologous piecewise-smooth closed path $\gamma$ in $\Sigma\times [s_ - , s_+]$ tracing $\Sigma\times s_+$ from $\Lambda^+_0= L_0\cap (\Sigma\times s_+)$ to $\Lambda^+_1= L_1\cap (\Sigma\times s_+)$, then tracing $L_1$ from $\Lambda^+_1 = L_1\cap (\Sigma\times s_+)$ to $\Lambda^-_1 = L_1\cap (\Sigma\times s_-)$, then following possibly several arcs in $L_1$ connecting points in $\Lambda^-_1$, then tracing $\Sigma\times s_-$ from $\Lambda^-_1= L_1\cap (\Sigma\times s_-)$ to $\Lambda^-_0= L_0\cap (\Sigma\times s_-)$, then following possibly several arcs in $L_0$ connecting points in $\Lambda^-_0$, and finally tracing $L_0$ from $\Lambda^-_0 = L_0\cap (\Sigma\times s_-)$ to $\Lambda^+_0 = L_0\cap (\Sigma\times s_+)$. We will call such a $\gamma$ a {\it belt path of $L$}.

We claim that if $L$ is belted, then the gap of $L$ with respect to a cobordism 1-form (with respect to which $L$ is exact) does not depend on the form.

Indeed, assume $\theta$ and $\theta'$ are cobordism 1-forms on $\Sigma\times [s_-,s_+]$ so that $L$ is exact with respect to both $\theta$ and $\theta'$.
Then $\theta-\theta'$ is a closed 1-form vanishing near the positive and the negative boundaries of $\Sigma\times [s_-,s_+]$. Since a belt path $\gamma$ of $L$ is null-homologous, the integral of $\theta-\theta'$ over $\gamma$ vanishes. Since $\theta-\theta'$ vanishes near the positive and the negative boundaries of $\Sigma\times [s_-,s_+]$, the latter zero integral is the sum of the integrals of $\theta-\theta'$ over the parts of $\gamma$ lying in $L_0$ and $L_1$, which readily implies the claim.

{\it A trivial Lagrangian cobordism in $(\Sigma\times [s_-,s_+],s\lambda)$} is a cobordism $\Lambda\times [s_-,s_+]$ where $\Lambda$ is a Legendrian submanifold of $(\Sigma,\xi)$.

A trivial two-part Lagrangian cobordism
\[
L= (\Lambda_0\sqcup\Lambda_1)\times [s_-,s_+]
\]
{is belted: to construct a belt path $\gamma$ for $L$, take a path $\Gamma$ in
$\Sigma$ from $x\in\Lambda_0$ to $y\in\Lambda_1$ (it exists, since, by our assumption, $\Sigma$ is connected). Now   define
$\gamma$ as the path tracing
$\Gamma\times s_+$ from $x\times s_+\in\Lambda_0\times s_+$ to $y\times s_+ \in\Lambda_1\times s_+$, then tracing $y\times [s_-, s_+]$ from $y\times s_+$ to $y\times s_-$, then tracing $\Gamma\times s_-$ from $y\times s_-$ to $x\times s_-$, and finally tracing $x\times [s_-,s_+]$ from $x\times s_-$ to $x\times s_+$. Note also that $L$ is $(s\lambda)$-exact with the gap $0$.}

Given a (two-part exact) Lagrangian cobordism $L\subset \Sigma\times [s_-,s_+]$ between $\Lambda^\pm$, consider the Lagrangian submanifold $\bL\subset \big(\Sigma\times\R_+,d(s\lambda)\big)$ defined as
\[
\bL := \big(\Lambda^-\times (0,s_-] \big) \cup L \cup \big(\Lambda^+\times [s_+,+\infty) \big).
\]
We call $\bL$ {\it the completion of $L$}.

Let $L\subset (\Sigma\times [s_-,s_+],s\lambda)$ be an exact two-part Lagrangian cobordism between two-part Legendrian submanifolds $\Lambda^\pm\subset (\Sigma, \xi)$.
By {\it an exact Lagrangian cobordism isotopy of $L$} we mean a smooth family $\{ L^\tau,\theta^\tau\}_{0\leq\tau\leq T}$, where $\{ L^\tau\}_{0\leq\tau\leq T}$ is a Lagrangian isotopy of $L=L^0$ in $\Sigma\times [s_-,s_+]$ and $\{ \theta^\tau\}_{0\leq\tau\leq T}$ is a smooth family of cobordism 1-forms so that

\smallskip
\noindent
(1) Each $L^\tau$ is a two-part $\theta^\tau$-exact Lagrangian cobordism between its positive and negative boundaries that will be denoted by $\Lambda^\pm_\tau$. In particular, $\{ \Lambda^\pm_\tau\}_{0\leq\tau\leq T}$
are Legendrian isotopies in $(\Sigma^\pm,\xi^\pm)$. We will say that these are {\it the Legendrian isotopies induced by the exact Lagrangian cobordism isotopy
$\{ L^\tau,\theta^\tau\}_{0\leq\tau\leq T}$}.

\smallskip
\noindent
(2) There exist $\delta_\pm >0$ such that for all $\tau\in [0,T]$
\[
L^\tau\cap \Sigma\times [s_+ - \delta_+, s_+] = \Lambda^+_\tau \times [s_+ - \delta_+, s_+],
\]
\[
L^\tau\cap \Sigma\times [s_-, s_- + \delta_-] = \Lambda^-_\tau \times [s_-, s_- + \delta_-].
\]

\bigskip
Note that the gaps $C(\tau)$ of $L^\tau$, $0\leq \tau\leq T$, with respect to $\theta^\tau$ form a smooth function of $\tau$.

Clearly, if $L^0$ is belted, then so are all $L^\tau$, $0\leq \tau\leq T$.

\subsection{Chekanov-Eliashberg algebra and the corresponding persistence modules}
\label{subsec-Chekanov-Eliashberg-algebra-pers-mod}

Let us recall the definition of the Chekanov-Eliashberg algebra associated to a non-degenerate pair $(\Lambda,\lambda)$, where $\Lambda\subset (\Sigma,\xi)$ is a compact Legendrian submanifold without boundary.
See \cite{Ekh-JEMS} (cf. \cite{EES-TAMS},\cite{Ekholm-Honda-Kalman}, \cite{DRizell-lifting}) for more details and \cite{Chekanov}, \cite{Eliashberg-ICM98}, \cite{EGH}
for the original ideas underlying the construction.

Denote by $\cA (\Lambda,\lambda)$ a free non-commutative unital algebra over $\Z_2$ generated by the elements of $\Z_2$ and by $\cR (\Lambda, \lambda)$.

The algebra $\cA (\Lambda,\lambda)$ comes with a filtration defined by the action: the action of a Reeb chord is its time-length and the action of a monomial which is a product of Reeb chords is the sum of the actions of the factors. The action of the constant monomial $1\in\Z_2$ is set to be zero and the action of $0$ is defined as $-\infty$. For $r\in (-\infty,+\infty)$ define $\cA_r (\Lambda,\lambda)$ as the vector subspace of $\cA (\Lambda,\lambda)$ spanned over $\Z_2$ by the monomials whose action is smaller than $r$. It is easy to see that for a finite $r$ the vector space $\cA_r (\Lambda,\lambda)$ is finite-dimensional. For any $r\leq r'$ there is a natural morphism $\cA_r (\Lambda,\lambda)\to \cA_{r'} (\Lambda,\lambda)$ induced by the inclusion of the generators.

For an appropriate (a so-called cylindrical) almost complex structure $J$ on $\Sigma\times\R_+$
one can define a
differential
$\partial_J$
on $\cA (\Lambda,\lambda)$
using a count of $J$-holomorphic maps of a disk with one positive and several (possibly no) negative boundary punctures into $\Sigma\times\R_+$. Such a map should send the boundary of the disk to $\Lambda\times\R_+$ and converge near positive/negative puncture to a positive/negative cylinder over a Reeb chord in $\cR (\Lambda,\lambda)$.

The set $\cJ (\Lambda)$ of $J$ for which $\partial_J$ is well-defined and
$\partial_J^2=0$ is connected and dense in the space of all cylindrical almost complex structures on $\Sigma\times \R_+$
\cite{Ekh-JEMS}, see also
\cite{EES-TAMS}, \cite{Ekholm-Honda-Kalman}, \cite{DRizell-lifting}.

Let $J\in \cJ (\Lambda)$.

A computation using Stokes' theorem and similar to \cite[Lemma 5.16]{BEHWZ} and \cite[Lemma B.3]{Ekh-JEMS} shows that for all $r$
the spaces $\cA_r (\Lambda,\lambda)$ are invariant under $\partial_J$.

For each $r\in {(-\infty,+\infty)}$ define a vector space $V_r (\Lambda,\lambda)$ over $\Z_2$:
\[
V_r (\Lambda,\lambda,J) := \frac{\Ker \partial_J|_{\cA_r (\Lambda,\lambda)}}{\Image \partial_J|_{\cA_r (\Lambda,\lambda)}}.
\]
The inclusion maps $\cA_r (\Lambda,\lambda)\to \cA_{r'} (\Lambda,\lambda)$, $r\leq r'$, induce morphisms $V_r (\Lambda,\lambda,J)\to V_{r'} (\Lambda,\lambda,J)$.

It is easy to see that the vector spaces $V_r (\Lambda,\lambda,J)$, $r\in {(0,+\infty)}$, together with the morphisms between them, form a persistence module $V (\Lambda,\lambda,J)$ over ${(0,+\infty)}$.

One can show (see \cite{EP-big}) that for fixed $\Lambda,\lambda$ a different choice of $J\in \cJ (\Lambda)$ does not change the isomorphism class of the persistence module $V (\Lambda,\lambda,J)$. The isomorphism class of this persistence module will be denoted by $V (\Lambda,\lambda)$.
The vector space $V_\infty (\Lambda,\lambda)$ is then isomorphic to the (non-filtered) {\it Legendrian contact homology of $(\Lambda,\lambda)$} -- that is, $\Ker \partial_J/\Image \partial_J$.
The dimension of the vector space
$V_\infty (\Lambda,\lambda)$ is invariant under
Legendrian isotopies of $\Lambda$ \cite{Ekh-JEMS}, see also
\cite{EES-TAMS}, \cite{Ekholm-Honda-Kalman}, \cite{DRizell-lifting}.

Assume now that $\Lambda=\Lambda_0\sqcup\Lambda_1$ is a two-part Legendrian submanifold.

A sequence of Reeb chords $a_1,\ldots,a_k\in \cR(\Lambda,\lambda)$ is called {\it $ij$-composable} for $i,j=0,1$, if $a_1$ starts at $\Lambda_i$, $a_k$ ends at $\Lambda_j$, and for each $m=1,\ldots,k-1$ the end of $a_m$ lies in the same part of $\Lambda$ as the origin of $a_{m+1}$. Note that an $ij$-composable sequence of Reeb chords must contain at least one chord from $\cR_{ij} (\Lambda, \lambda)$. The corresponding monomial $a_1\cdot\ldots\cdot a_k$ in $\cA (\Lambda,\lambda)$ will be also called {\it $ij$-composable}.

Denote by $\cA (\Lambda_0,\Lambda_1,\lambda)$ the vector subspace of $\cA (\Lambda,\lambda)$ (over $\Z_2$) generated by all the $01$-composable monomials $a_1\cdot\ldots\cdot a_k$, $k\in\Z_{>0}$. (Note that the polynomials appearing in $\cA (\Lambda_0,\Lambda_1,\lambda)$ have no constant terms!). We will call $\cA (\Lambda_0,\Lambda_1,\lambda)$ the {\it $01$-subspace of $\cA (\Lambda,\lambda)$}.

Set
\[
\cA_r (\Lambda_0,\Lambda_1,\lambda) := \cA (\Lambda_0,\Lambda_1,\lambda) \cap \cA_r (\Lambda,\lambda).
\]

Let $J\in \cJ (\Lambda=\Lambda_0\sqcup\Lambda_1)$. Then, in particular, $J\in \cJ (\Lambda_0)\cap \cJ(\Lambda_1)$.

It is easy to see that $\cA (\Lambda_0,\Lambda_1,\lambda)$ is invariant under $\partial_J$ and hence so are the spaces $\cA_r (\Lambda_0,\Lambda_1,\lambda)$
for all $r$.
For each $r\in (-\infty,+\infty)$ define a vector space $V_r (\Lambda_0,\Lambda_1,\lambda)$ over $\Z_2$:
\[
V_r (\Lambda_0,\Lambda_1,\lambda,J) := \frac{\Ker \partial_J|_{\cA_r (\Lambda_0,\Lambda_1,\lambda)}}{\Image \partial_J|_{\cA_r (\Lambda_0,\Lambda_1,\lambda)}}.
\]
The inclusion maps $\cA_r (\Lambda_0,\Lambda_1,\lambda)\to \cA_{r'} (\Lambda_0,\Lambda_1,\lambda)$, $r\leq r'$, induce morphisms $V_r (\Lambda_0,\Lambda_1,\lambda,J)\to V_{r'} (\Lambda_0,\Lambda_1,\lambda,J)$.

It is easy to see that the vector spaces $V_r (\Lambda_0,\Lambda_1,\lambda,J)$, $r\in {(0,+\infty)}$, together with the morphisms between them, form a persistence module over ${(0,+\infty)}$.
We will denote this persistence module by $V (\Lambda_0,\Lambda_1,\lambda,J)$.
Below, whenever needed, we will also view $V (\Lambda_0,\Lambda_1,\lambda,J)$ as a persistence module over ${(-\infty,+\infty)}$ using the trivial extension as in Example~\ref{exam-extension-of-pers-modules}.

Similarly to the above,
one can show (see \cite{EP-big}) that

\smallskip
\noindent
- For fixed $\Lambda=\Lambda_0\sqcup\Lambda_1,\lambda$ a different choice of $J\in \cJ (\Lambda)$ does not change the isomorphism class of the persistence module $V (\Lambda_0,\Lambda_1,\lambda,J)$. The isomorphism class of this persistence module will be denoted by $V (\Lambda_0,\Lambda_1,\lambda)$.

\smallskip
\noindent
-
The dimension of the vector space
$V_\infty (\Lambda_0,\Lambda_1,\lambda)$ is invariant under
Legendrian isotopies of $\Lambda$.

\smallskip
Abusing the terminology we will call $V (\Lambda_0,\Lambda_1,\lambda)$ {\it the LCH persistence module associated to $(\Lambda,\lambda)$}, where {\it LCH} stands for ``Legendrian contact homology".

It is easy to see that if $c>0$ then
\[
V (\Lambda_0,\Lambda_1,c\lambda) = V^{[\times 1/c]} (\Lambda_0,\Lambda_1,\lambda),
\]
where $V^{[\times 1/c]} (\Lambda_0,\Lambda_1,\lambda)$ is the isomorphism class of persistence modules obtained from $V (\Lambda_0,\Lambda_1,\lambda)$ by the multiplicative shift by $1/c$ (see Remark~\ref{rem-shifts-isomorphisms}).

\subsection{Morphisms of persistence modules defined by Lag\-ran\-gian co\-bor\-disms}
\label{subsec-morphisms-pers-mod-defd-by-Lagr-cobs}

Let $L=L_0\sqcup L_1\subset \big(\Sigma\times [s_-,s_+],\omega=d(s\lambda)\big)$, $s_-< s_+$, be a two-part exact Lagrangian cobordism between two-part Legendrian submanifolds
$\Lambda^\pm = \Lambda^\pm_0\sqcup \Lambda^\pm_1\subset (\Sigma,\xi)$.

Assume the pairs $(\Lambda^\pm,\lambda)$ are non-degenerate -- in this case we will say that $L$ is {\it non-degenerate}.

For an appropriate -- a so-called {\it adap\-ted (to $L$ and $\omega$}) -- almost complex structure $I$ on $\Sigma\times\R_+$
define a unital algebra morphism
\[
\Phi^{L,I}: \cA (\Lambda^+,s_+\lambda,I^+) \to \cA (\Lambda^-,s_-\lambda,I^-)
\]
by prescribing its values on the generators:
\[
\Phi^{L,I} (1):=1
\]
and for any $a\in \cR (\Lambda^+,s_+\lambda)$
\[
\Phi^{L,I} (a) := \sum\limits_{\dim \cM_{L,I} (a;b_1,\ldots,b_m) = 0} |\cM_{L,I} (a; b_1,\ldots,b_m)| b_1\cdot\ldots\cdot b_m,
\]
where $\cM_{L,I} (a;b_1,\ldots,b_m)$ is the moduli space of $I$-holomorphic maps of a disk with one positive and $m\geq 0$ negative boundary punctures into $\Sigma\times\R_+$ that send the boundary of the disk to $\bL$ and converge near the positive puncture to a cylinder over $a$ and the negative punctures to the cylinders over the chords $b_1,\ldots,b_m\in\cR (\Lambda^-,s_-\lambda)$. Here $I^\pm$ are cylindrical almost complex structures on $\Sigma\times\R_+$ induced by the restrictions of $I$ at the ends of $\Sigma\times\R_+$. See \cite{Ekh-JEMS}, cf. \cite{Ekholm-Honda-Kalman}, for more details.

The set $\cI (L)$ of $I$ for which $\Phi^{L,I}$ is a well-defined unital algebra morphism is
dense in the space of all adapted almost complex structures
on $\Sigma\times\R_+$ and the map
\[
\cI (L)\to \cJ (\Lambda^+)\times \cJ (\Lambda^-),\ I\mapsto (I^+,I^-),
\]
is surjective \cite{Ekh-JEMS}, cf. \cite{Ekholm-Honda-Kalman}.

The map $\Phi^{L,I}$ is called {\it the cobordism map associated to $L,I$}.

\begin{rem}
\label{rem-cob-map-non-trivial-implies-cob-is-belted}
{\rm
Assume that $I\in \cI (L)$ and the restriction of $\Phi^{L,I}$ to the $01$-subspace $\cA (\Lambda^+_0,\Lambda^+_1,s_+\lambda)$ is not the zero map.
Then for some $01$-chord $a$ and some {\sl non-empty} set of chords $b_1,\ldots,b_m$ the moduli space $\cM_{L,I} (a;b_1,\ldots,b_m)$ is non-empty. Rescale an $I$-holomorphic map of a disk $D'$, with boundary punctures, that defines an element of the moduli space and obtain a map of $D'$ into $\Sigma\times [s_-,s_+]$. Concatenating the image of $\partial D'$ under the latter map with the chords $a,b_1,\ldots,b_m$ we get a belt path for the two-part Lagrangian cobordism $L$.
In other words, we have obtained that, unless $L$ is belted, $\Phi^{L,I}$ has to be the zero map.
}
\end{rem}

The following two claims are proved in \cite{EP-big} using \cite[Lemma 3.14]{Ekholm-Honda-Kalman}
(a chain homotopy result for the cobordism maps, which is a version of \cite[Lemma B.15]{Ekh-JEMS}) -- see the proof of Proposition~\ref{prop-Phi-L-two-part-Lagr-isotopy} below for a similar use of the same result.

\begin{prop}[\cite{EP-big}]
\label{prop-Phi-L-I-defines-morphism-of-pers-mod}
Assume that $L$ is a non-degenerate {belted} two-part exact Lagrangian cobordism.
Let $I\in \cI (L)$ and assume that the restriction of the cobordism map $\Phi^{L,I}$ to $\cA (\Lambda^+_0,\Lambda^+_1,s_+\lambda)$ is non-trivial.
Let $C$ be the gap of $L$ ({since $L$ is belted, the gap} is independent of the cobordism 1-form with respect to which $L$ is exact).

Then
$\Phi^{L,I}$ maps $\cA_r (\Lambda_0,\Lambda_1,s_+\lambda,I^+)$ into $\cA_{r-C} (\Lambda_0,\Lambda_1,s_-\lambda,I^-)$ for each $r\in (-\infty,+\infty)$
and, accordingly, defines a morphism of persistence modules over $(-\infty,+\infty)$:
\[
\Phi^{L,I}_* : V (\Lambda_0,\Lambda_1,s_+\lambda,I^+) \to V^{[-C]} (\Lambda_0,\Lambda_1,s_-\lambda,I^-),
\]
where $V^{[-C]} (\Lambda_0,\Lambda_1,s_-\lambda,I^-)$ is the isomorphism class of persistence modules over $(-\infty,+\infty)$ obtained from $V (\Lambda_0,\Lambda_1,s_-\lambda,I^-)$
by the additive shift by $-C$ (see Remark~\ref{rem-shifts-isomorphisms}).
\end{prop}
\bigskip

If the restriction of the cobordism map $\Phi^{L,I}$ to $\cA (\Lambda^+_0,\Lambda^+_1,s_+\lambda)$ is {the zero map}, we set $\Phi^{L,I}_*$ to be the zero morphism into the trivial persistence module.

\begin{prop}[\cite{EP-big}]
\label{prop-Phi-L-I-indep-of-I-up-to-equivalence}
The morphism
$\Phi^{L,I}_*$ is independent of $I\in \cI (L)$ up to the right-left equivalence in the category of persistence modules over $(-\infty,+\infty)$.
\end{prop}
\bigskip

Further on,
we will denote the equivalence class of the persistence module morphism $\Phi^{L,I}_*$ by $\Phi^L_*$.

Abusing the terminology, we will not distinguish between an isomorphism class of a persistence module and a specific persistence module representing it, as well as between an equivalence class of morphisms of persistence modules and a specific morphism representing it. In particular, we will write
\[
\Phi^L_* :  V (\Lambda_0,\Lambda_1,s_+\lambda) \to V^{[-C]} (\Lambda_0,\Lambda_1,s_-\lambda).
\]
Note that compositions of equivalence classes of morphisms of persistence modules are well-defined.

The following proposition is proved in \cite{EP-big} using the well-known results (see \cite{Ekh-JEMS}) about the cobordism map associated to a trivial exact Lagrangian cobordism.

\begin{prop}[\cite{EP-big}]
\label{prop-Phi-L-trivial-cobordisms}
Assume $L=(\Lambda_0\sqcup \Lambda_1)\times [s_-,s_+]$ is a trivial non-degenerate two-part exact Lagrangian cobordism. Set
\[
U:= V (\Lambda_0,\Lambda_1,s_+\lambda).
\]

Then
\[
\Phi^L_* = \Sh_U [\times s_+/s_-].
\]
\Qed
\end{prop}
\bigskip

Having recalled the needed preparational statements, we present now the key result of this section.

\begin{prop}
\label{prop-Phi-L-two-part-Lagr-isotopy}
With $L$ being a non-degenerate {belted} two-part exact Lagrangian cobordism as above,
suppose that
$\{ L^\tau = L^\tau_0\sqcup L^\tau_1,\theta_\tau\}_{0\leq \tau\leq T}$ is an exact Lagrangian cobordism isotopy of $L$ with fixed boundary.
Assume that $\Phi^L_*$ is non-trivial.

Let $C(\tau)$ be the gap of $L^\tau$, $0\leq \tau\leq T$ ({since $L$ is belted}, it is independent of the cobordism 1-form with respect to which $L^\tau$ is exact).
Let
\[
U:= V (\Lambda_0,\Lambda_1,s_+\lambda),
\]
\[
W:= V (\Lambda_0,\Lambda_1,s_-\lambda).
\]
\[
C_{min} := \min_{\tau\in [0,T]} C(\tau),
\]
\[
C_0 := C(0) - C_{min},\ C_T := C(T) - C_{min}.
\]

Then
\begin{equation}
\label{eqn-Phi-L0-LT-shifts}
\Sh_{W^{[-C(0)]}} [+C_0] \circ \Phi^{L^0}_* =
\Sh_{W^{[-C(T)]}} [+C_T] \circ \Phi^{L^T}_*.
\end{equation}
Here the equality is between equivalence classes of morphisms $U \to W^{[-C_{min}]}$.
\end{prop}

\bigskip
\noindent
{\bf Proof of Proposition~\ref{prop-Phi-L-two-part-Lagr-isotopy}:}
Pick $I_0\in \cI (L^0)$, $I_T\in \cI (L^T)$.
It follows from \cite[Lemma 3.14]{Ekholm-Honda-Kalman}
(which is a version of \cite[Lemma B.15]{Ekh-JEMS})
that there exists an $\Z_2$-linear map
\[
K: \cA (\Lambda^+,s_+\lambda)\to \cA (\Lambda^-,s_-\lambda)
\]
such that
\begin{equation}
\label{eqn-K-chain-homotopy}
K \circ \partial_{I_0^+} + \partial_{I_T^-}\circ K = \Phi^{L^T,I_T} - \Phi^{L^0,I_0}.
\end{equation}
The map $K$ is defined on each monomial $a_1\cdot\ldots\cdot a_k$, $a_1,\ldots,a_k\in \cR (\Lambda^+,s_+\lambda)$, as
\begin{equation}
\label{eqn-K-on-a-monomial-1}
K (a_1\cdot\ldots\cdot a_k) := \sum_{j=1}^k  \Phi^{L^T,I_T} (a_1\cdot\ldots\cdot a_{j-1}) K (a_j) \Phi^{L^0,I_0} (a_{j+1}\cdot\ldots\cdot a_k),
\end{equation}
so that for each $a\in \cR (\Lambda^+,s_+\lambda)$
\begin{equation}
\label{eqn-K-on-a-monomial-2}
K (a):= \sum_{b_1,\ldots, b_m\in \cR (\Lambda^-,s_-\lambda)} n_{\{ L^T\},\{I_T\}} (a;b_1,\ldots,b_m) b_1\cdot\ldots\cdot b_m,
\end{equation}
where $\{ I_\tau\}_{0\leq \tau\leq T}$ is a {generic} family of adapted almost complex structures on $M$ connecting $I_0$ and $I_T$, and
$n_{\{ L^T\},\{I_T\}} (a;b_1,\ldots,b_m)$ is the mod-2 number of elements (counted using an abstract perturbation scheme, see \cite{Ekh-JEMS} for details) of a certain moduli space which is non-empty only if
\[
\cup_{0\leq \tau\leq T} \cM_{L^\tau,I_\tau} (a;b_1,\ldots,b_m)\neq\emptyset.
\]
Here $\cM_{L^\tau,I_\tau} (a;b_1,\ldots,b_m)$ is the moduli space of pseudo-holomorphic maps of the disk {with boundary punctures} into $\Sigma\times\R_+$ used to define $\Phi^{L^\tau,I_\tau}$.
This moduli space is non-empty only for some finite set of $\tau\in [0,T]$ -- the corresponding pseudo-holomorphic maps have index $-1$. See \cite[Lemma 3.14]{Ekholm-Honda-Kalman} and \cite[Lemma B.15]{Ekh-JEMS} for the details.

Using the energy inequality and the Stokes theorem one can show that if $\cM_{L^\tau,I_\tau} (a;b_1,\ldots,b_m)\neq\emptyset$, then
\[
\sum\limits_{i=1}^m {s_-}l(b_i) < {s_+}l(a) - C(\tau)\ \textrm{for}\ a\in\cR (\Lambda_0^+,\Lambda_1^+,s_+\lambda),
\]
\[
\sum\limits_{i=1}^m {s_-}l(b_i) < {s_+}l(a) + C(\tau)\ \textrm{for}\ a\in\cR (\Lambda_1^+,\Lambda_0^+,s_+\lambda),
\]
\[
\sum\limits_{i=1}^m {s_-}l(b_i) < {s_+}l(a)\ \textrm{for}\ a\in\cR (\Lambda_i^+,\Lambda_i^+,s_+\lambda),\ i=0,1,
\]
where $l(\cdot)$ is the action (time-length) of a Reeb orbit.
Using these inequalities together with \eqref{eqn-K-on-a-monomial-1}, \eqref{eqn-K-on-a-monomial-2} it is not hard to see that $K$ maps the subspace $\cA_r (\Lambda_0^+,\Lambda_1^+,s_+\lambda)$ into $\cA_{r-C_{min}} (\Lambda_0^-,\Lambda_1^-,s_-\lambda)$ for any $r\in (-\infty,+\infty)$. Therefore the chain homotopy formula \eqref{eqn-K-chain-homotopy} implies that the restrictions of $\Phi^{L^T,I_T}$ and $\Phi^{L^0,I_0}$ to $\cA_r (\Lambda_0^+,\Lambda_1^+,s_+\lambda)$ induce the same map on homology -- that is, a map
\[
V_r (\Lambda_0^+,\Lambda_1^+,s_+\lambda,I_0^+)\to V_{r-C_{min}} (\Lambda_0^-,\Lambda_1^-,s_-\lambda,I_T^-)
\]
for any $r\in (-\infty,+\infty)$.
This latter map equals, on one hand, the composition of
\[
\Phi^{L^0}_{r,*}: V_r (\Lambda_0^+,\Lambda_1^+,s_+\lambda,I_0^+)\to V_{r-C(0)} (\Lambda_0^-,\Lambda_1^-,s_-\lambda,I_T^-)
\]
and the shift
\[
V_{r-C(0)} (\Lambda_0^-,\Lambda_1^-,s_-\lambda,I_T^-)\to V_{r-C_{min}} (\Lambda_0^-,\Lambda_1^-,s_-\lambda,I_T^-)
\]
and, on the other hand, the composition of
\[
\Phi^{L^T}_{r,*}: V_r (\Lambda_0^+,\Lambda_1^+,s_+\lambda,I_0^+)\to V_{r-C(T)} (\Lambda_0^-,\Lambda_1^-,s_-\lambda,I_T^-)
\]
and the shift
\[
V_{r-C(T)} (\Lambda_0^-,\Lambda_1^-,s_-\lambda,I_T^-)\to V_{r-C_{min}} (\Lambda_0^-,\Lambda_1^-,s_-\lambda,I_T^-).
\]
The equality between the two compositions yields \eqref{eqn-Phi-L0-LT-shifts}. This finishes the proof of the proposition.
\Qed

\subsection{An invariant of two-part Legendrians via LCH persistence modules}
\label{subsec-invt-of-two-part-Lagrs-via-pers-modules}

With $\Sigma$ and $\lambda$ as above,
assume $\Lambda = \Lambda_0\sqcup \Lambda_1$ is a two-part Legendrian submanifold in $(\Sigma,\xi=\Ker\lambda)$.

\begin{defin}
\label{def-l-min}
{\rm
If the pair $(\Lambda,\lambda)$ is non-degenerate, then for each $s>1$ define $l_{min,s} (\Lambda_0,\Lambda_1,\lambda)$ as the smallest left end of the bars of multiplicative length greater than $s$ in the barcode of $V(\Lambda_0,\Lambda_1,\lambda)$. Denote by $l_{min,\infty} (\Lambda_0,\Lambda_1,\lambda)$
the smallest left end of the infinite bars in the barcode of $V(\Lambda_0,\Lambda_1,\lambda)$.
If there are no such bars, set $l_{min,s} (\Lambda_0,\Lambda_1,\lambda) := +\infty$ or, respectively, $l_{min,\infty} (\Lambda_0,\Lambda_1,\lambda) := +\infty$.

For a general, possibly degenerate, pair $(\Lambda,\lambda)$ and $s>1$ define
\[
l_{min,s} (\Lambda_0,\Lambda_1,\lambda):=\liminf l_{min,s} (\Lambda'_0,\Lambda'_1,\lambda),
\]
\[
l_{min,\infty} (\Lambda_0,\Lambda_1,\lambda):=\liminf l_{min,\infty} (\Lambda'_0,\Lambda'_1,\lambda),
\]
where each $\Lambda' = \Lambda'_0\sqcup \Lambda'_1$
is a two-part Legendrian submanifold obtained from $\Lambda$ by a $C^\infty$-small Legendrian isotopy and such that the pair $(\Lambda',\lambda)$ is non-degenerate, and the liminf is taken over all such $\Lambda'$ as the $C^\infty$-size of the Legendrian isotopy converges to zero. (One can show -- see \cite{EP-big} -- that for a non-degenerate pair $(\Lambda=\Lambda_0\sqcup\Lambda_1,\lambda)$ both definitions yield the same $l_{min,s} (\Lambda_0,\Lambda_1,\lambda)$ and $l_{min,\infty} (\Lambda_0,\Lambda_1,\lambda)$).
}
\end{defin}

For an example where $l_{min,s}$ can be computed explicitly see Proposition~\ref{prop-twochords} below.

Note that $l_{min,s} (\Lambda_0,\Lambda_1,\lambda)$ is a non-decreasing function of $s$ with values in $(0,+\infty]$. In particular,
\[
l_{min,s} (\Lambda_0,\Lambda_1,\lambda)\leq l_{min,\infty} (\Lambda_0,\Lambda_1,\lambda)\ \textrm{for all}\ s\in (1,+\infty),
\]
and therefore if the number $l_{min,\infty} (\Lambda_0,\Lambda_1,\lambda)$ is finite, then so is $l_{min,s} (\Lambda_0,\Lambda_1,\lambda)$.

Recall from Section~\ref{subsec-homologically-bonded-Leg-submfds} that the {\it stabilization of $\Sigma$} is the manifold $\hSigma := \Sigma\times T^* \SP^1 (r,\tau)$, $\tau\in \SP^1$, equipped with the contact form $\hlambda := \lambda + rd \tau = dz + \vartheta + rd \tau$. Since $(\Sigma = P\times \R (z),\lambda = dz + \vartheta)$ is nice, so is $(\hSigma,\hlambda)$ (if $(P, d\vartheta)$ has bounded geometry at infinity, then so does $(P\times T^* \SP^1, d\vartheta + dr\wedge dt)$). For a two-part Legendrian submanifold $\Lambda_0\sqcup\Lambda_1\subset\Sigma$ and $s>1$ write
\[
\hLambda_i := \Lambda_i\times \{ r=0\}\subset \hSigma,\ i=0,1,
\]
\[
\hl_{min,s} (\Lambda_0,\Lambda_1,\lambda):= l_{min,s} (\hLambda_0,\hLambda_1,\hlambda),
\]
\[
\hl_{min,\infty} (\Lambda_0,\Lambda_1,\lambda):= l_{min,\infty} (\hLambda_0,\hLambda_1,\hlambda).
\]

\begin{defin}
\label{def-homol-lambda-bonded}
{\rm
We say that the pair $(\Lambda_0\sqcup\Lambda_1,\lambda)$ is {\it weakly homologically bonded}, if $l_{min,s} (\Lambda_0,\Lambda_1,\lambda) < +\infty$ for all $s>1$ -- that is, there are bars of arbitrarily large multiplicative length in the barcode of $V (\Lambda_0,\Lambda_1,\lambda)$.

We say that the pair $(\Lambda_0\sqcup\Lambda_1,\lambda)$ is {\it homologically bonded}, if $l_{min,\infty} (\Lambda_0,\Lambda_1,\lambda) < +\infty$  -- that is, there are infinite bars in the barcode of $V (\Lambda_0,\Lambda_1,\lambda)$.

We say that the pair $(\Lambda_0\sqcup\Lambda_1,\lambda)$ is {\it stably homologically bonded}, if \break $\hl_{min,\infty} (\Lambda_0,\Lambda_1,\lambda) < +\infty$.
}
\end{defin}
\bigskip

Clearly, homological bondedness implies weak homological bondedness. We do not know whether homological bondedness implies stable homological bondedness -- see Remark~\ref{rem-hom-bondedness-vs-stable-hom-bondedness}.

\begin{rem}
\label{rem-l-min-s-robust}
{\rm
Assume that $l_{min,s} (\Lambda_0,\Lambda_1,\lambda) < +\infty$ for some $s\in (1,+\infty]$. Let
$\Lambda' = \Lambda'_0\sqcup \Lambda'_1$ be a two-part Legendrian submanifold in $(\Sigma,\xi)$ obtained from $\Lambda = \Lambda_0\sqcup \Lambda_1$ by a
contact isotopy  $\{\phi_t\}$ with the conformal factor satisfying $||(\phi_t^*)^{-1}\lambda/\lambda -1 || < \delta$ for all $t$. If $\delta = \delta(s)$ is small enough, then $l_{min,s} (\Lambda'_0,\Lambda'_1,\lambda)$ is also a finite number which tends to $l_{min,s} (\Lambda_0,\Lambda_1,\lambda)$ as $\delta\to 0$.
The proof will appear in \cite{EP-big}.
}
\end{rem}
\bigskip

\subsection{A lower bound on $pb^+$ via LCH persistence modules}
\label{subsec-lower-bound-on-pb-via-pers-modules}

With $\Sigma$ as in \eqref{eq-contactization-vsp} and $0<s_-<s_+$, denote for brevity
\[
M:= \Sigma\times [s_-,s_+].
\]
Recall that $\omega:=d(s\lambda)$.
Assume $\Lambda = \Lambda_0\sqcup \Lambda_1$ is a two-part Legendrian submanifold in $(\Sigma,\xi)$.
Let
$L=\Lambda\times [s_-,s_+]$ be the corresponding trivial two-part exact Lagrangian cobordism in $(M= \Sigma\times [s_-,s_+],s\lambda)$.

Define an admissible quadruple $X_0,X_1,Y_0,Y_1\subset M$ as follows:
\begin{equation}
\label{eqn-def-X0-X1}
X_0 := \Lambda_0 \times [s_-,s_+],\ X_1 := \Lambda_1 \times [s_-,s_+],
\end{equation}
\begin{equation}
\label{eqn-def-Y0-Y1}
Y_0 := \Sigma \times s_-,\ Y_1 := \Sigma \times s_+.
\end{equation}

\begin{thm}
\label{thm-lower-bound-on-pb-via-pers-modules}
Assume that $l_{min,s_+/s_-} (\Lambda_0,\Lambda_1,\lambda) < +\infty$.

Then
\[
pb^+_M (X_0,X_1,Y_0,Y_1)\geq \frac{1}{(s_+ - s_-) l_{min,s_+/s_-} (\Lambda_0,\Lambda_1,\lambda)} > 0.
\]
\end{thm}

\bigskip
\noindent
{\bf Proof of Theorem~\ref{thm-lower-bound-on-pb-via-pers-modules}:}

Assume first that the pair $(\Lambda,\lambda)$ is non-degenerate.

Set
\[
V:= V (\Lambda_0,\Lambda_1,\lambda),
\]
\[
U:= V^{[\times 1/s_+]} = V (\Lambda_0,\Lambda_1,s_+\lambda),
\]
\[
W:= V^{[\times 1/s_-]} = U^{[\times s_+/s_-]} = V (\Lambda_0,\Lambda_1,s_-\lambda).
\]

Let $(F,G)\in{\cF''_M} (X_0,X_1,Y_0,Y_1)$.
We need to show that
\begin{equation}
\label{eqn-Poisson-bracket-F-G-lower-estimate}
{\sup_M} \{F,G\}\geq \frac{1}{(s_+ - s_-)l_{min,s_+/s_-} (\Lambda_0,\Lambda_1,\lambda) }.
\end{equation}

Following \cite{BEP}, consider the deformation
\[
\omega_\tau := \omega + \tau dF\wedge dG, \ \tau\in\R,
\]
of $\omega$.
A direct calculation shows that
\[
dF\wedge dG\wedge \omega^{n-1} = -\frac{1}{n}\{ F,G\} \omega^n.
\]
and therefore
\[
\omega_\tau^n = (1-\tau \{ F,G\})\omega^n.
\]
Thus $\omega_\tau$ is symplectic for any $\tau\in \textsf{I}:=[0, 1/{\sup_M} \{ F,G\})$. Note that $L$ is Lagrangian with respect to $\omega_\tau$ for all $\tau\in \textsf{I}$.

Following the idea underlying Moser's method \cite{Moser}, define a (time-de\-pen\-dent) vector field $v_\tau$, $\tau\in \textsf{I}$, on $M$ by
\[
FdG = -i_{v_\tau} \omega_\tau.
\]

One can check that
\[
v_\tau = \frac{F}{1-\tau\{ F,G\}} \sgrad G.
\]
Since $(F,G)\in\cF''_M (X_0,X_1,Y_0,Y_1)$, the vector field $F\sgrad G$ is complete. Also the function $1-\tau\{ F,G\}$ is bounded from below by the constant $1 - \tau \sup_M \{ F,G\}$ which is positive since  $\tau\in \textsf{I}$. Thus, for each $T\in \textsf{I}$ the time-dependent vector field $v_\tau$, $0\leq \tau\leq T$, equals to the product of $F\sgrad G$ with a non-negative function bounded from above by a constant depending on $F,G$ and $T$. Therefore the time-$[0,T]$ flow $\sigma_{\tau_0}: M\to M$ of $v_\tau$, $0\leq \tau\leq T$, is well-defined. It is identity near the boundary of $M$, because $v_\tau$ vanishes there (since $G$ is constant near the positive and negative boundaries of $M$).

{For each $\tau\in \textsf{I}$ define
\[
L^\tau := (\sigma_\tau)^{-1} (L).
\]
}
A direct check shows that $\sigma_\tau^* \omega_\tau = \omega$
for all $\tau\in \textsf{I}$. Therefore $L^\tau$ is Lagrangian with respect to $\omega$ for all $\tau\in \textsf{I}$.

For each $\tau\in \textsf{I}$ set
\[
\theta^\tau := \sigma_\tau^* (s\lambda+\tau FdG)
\]
Since $d(s\lambda+\tau FdG) = \omega_\tau$ and $\sigma_\tau^* \omega_\tau = \omega$ and since $G$ is constant and $\sigma_\tau$ is identity near the positive and negative boundaries of $M$, one readily gets that each $\theta^\tau$ is a cobordism 1-form.

Since the trivial cobordism $L$ is $(s\lambda)$-exact and $F$ is $0$ near $L_0$ and $1$ near $L_1$, we get that $(s\lambda+\tau FdG)|_{L_i} = df_i$,  $i=0,1$, for a smooth function $f_i: L_i\to\R$ equal to $0$ near the negative boundary of $L_i$ and to some constant near the positive boundary of $L_i$.
Consequently, for each $\tau\in \textsf{I}$ the two-part Lagrangian cobordism $L^\tau$ is $\theta^\tau$-exact.

Therefore for each $T\in \textsf{I}$ the family $\{ L^\tau,\theta^\tau\}_{0\leq\tau\leq T}$ is an exact Lagrangian cobordism isotopy.

Since $L^0=L$ is a trivial exact Lagrangian cobordism, Proposition~\ref{prop-Phi-L-trivial-cobordisms} yields
\[
\Phi^{L^0}_* = \Sh_U [\times s_+/s_-] : U\to W.
\]
{Note that since $L^\tau$ is belted, the gap $C(\tau)$ of $L^\tau$ is independent of the cobordism 1-form with respect to which $L^\tau$ is exact.}
Using the assumptions on $F$ and $G$ one easily sees that
\[
C(\tau)=\tau.
\]
Since $C_0=C_{min} = C(0)=0$ and $C_T=C(T)=T$, by Proposition~\ref{prop-Phi-L-two-part-Lagr-isotopy},
\[
\Sh_U [\times s_+/s_-] =
\Sh_{W^{[-T]}} [+T] \circ \Phi^{L^T}_*.
\]

Therefore, for each $t>0$ the morphism $U_t\to U_{ts_+/s_-}$ in the persistence module $U$ is a composition of a linear map $U_t\to U_{(t-T)s_+/s_-}$, given by $\Phi^{L^T}_*$, and a morphism $U_{(t-T)s_+/s_-}\to U_{ts_+/s_-}$ in the persistence module $U$.

Assume that for some $t>0$ both $t$ and $ts_+/s_-$ lie in a bar {$\J$} in the barcode of $U$. Denote by $\pi_{s,ts_+/s_-,\J}: U_t\cap Q(\J) = Q_t (\J)\to U_{ts_+/s_-}\cap Q(\J) = Q_{ts_+/s_-} (\J)$ the restriction of the morphism $U_t\to U_{ts_+/s_-}$ restricted to the interval persistence submodule $Q (\J)$ of $U_t$ (see Example \ref{exam-interval-module} above). The persistence morphism $\pi_{s,ts_+/s_-,\J}$ is an isomorphism. By the conclusion above, this isomorphism is the composition of a map $U_t\cap Q(\J) = Q_t (\J)\to U_{(t-T)s_+/s_-}$ and a morphism $U_{(t-T)s_+/s_-}\cap Q(\J)\to U_{ts_+/s_-}\cap Q(\J)$ in the persistence module $U_t\cap Q(\J)$, which also has to be an isomorphism. This is possible only if $(t-T)s_+/s_-\in\J$, which yields an upper bound on $T$ in the following way. Namely, take $\J$ to be a bar of multiplicative length greater than $s_+/s_-$ in the barcode of $U$ whose left end is $s_+ l_{min,s_+/s_-} (\Lambda_0,\Lambda_1,\lambda)$ (this is the smallest left end of the bars of multiplicative length greater than $s_+/s_-$ in the barcode of $U$) and let $t$ be arbitrarily close from above to the left end of $\J$ -- that is, to $s_+ l_{min,s_+/s_-} (\Lambda_0,\Lambda_1,\lambda)$. Then $t$ and $ts_+/s_-$ lie in the same bar $\J$
and therefore, as we have shown above, $(t-T)s_+/s_-$ lies in $\J$ and, in particular, is greater or equal than the left end of $\J$:
\[
(t-T)s_+/s_- \geq s_+ l_{min,s_+/s_-} (\Lambda_0,\Lambda_1,\lambda).
\]
Hence,
\[
T \leq (s_+ - s_-) l_{min,s_+/s_-} (\Lambda_0,\Lambda_1,\lambda).
\]

Since this is true for any $T\in \textsf{I}=[0, 1/\max_M \{ F,G\})$, we obtain \eqref{eqn-Poisson-bracket-F-G-lower-estimate}
as required. This finishes the proof of the theorem {for the case where the pair $(\Lambda,\lambda)$ is non-degenerate.}

{The general case now follows from the non-degenerate one by the semi-continuity of the Poisson bracket invariant -- see \eqref{eqn-Hausdorff-convergence}.}
\Qed

As in Section~\ref{subsec-invt-of-two-part-Lagrs-via-pers-modules}, let
\[
\hSigma := \Sigma\times \R (r)\times \SP^1 (\tau) = \Sigma\times T^* \SP^1 (r,\tau)
\]
be the stabilization of $\Sigma$
equipped with the contact form
\[
\hlambda := \lambda + rd \tau = dz + \vartheta + rd \tau.
\]
Set
\[
\tM := M\times T^* \SP^1 (r,\tau),\;\; M = \Sigma \times [s_-,s_+],
\]
and equip $\tM$ with the symplectic form $\omega + dr\wedge d\tau = d(s\lambda) + dr\wedge d\tau$.

For each $E>0$ define $\tM_E\subset \tM$ as
\[
\tM_E := M\times (-E,E)\times \SP^1\subset \tM.
\]

With the admissible quadruple $X_0,X_1,Y_0,Y_1\subset M$ as above, set
\[
\tX_0:= X_0\times \{ r=0\},\  \tX_1:= X_1\times  \{ r=0\}\subset\tM_E\subset\tM,
\]
\[
\tY_0 := Y_0\times T^* \SP^1,\ \tY_1:= Y_1\times T^* \SP^1\subset \tM,
\]
\[
\tY_0 (E) := Y_0\times (-E,E)\times \SP^1,\ \tY_1 (E) := Y_1\times (-E,E)\times \SP^1\subset \tM_E.
\]

\begin{thm}
\label{thm-lower-bound-on-hpb-via-pers-modules}
Assume that $\hl_{min,s_+/s_-} (\Lambda_0,\Lambda_1,\lambda) < +\infty$.

Then for each $E>0$
\[
pb^+_{\tM_E,comp} \left(\tX_0,\tX_1,\tY_0 (E),\tY_1 (E)\right) \geq pb^+_{\tM,comp} (\tX_0,\tX_1,\tY_0,\tY_1)\geq
\]
\[
\geq\frac{1}{(s_+ - s_-) \hl_{min,s_+/s_-} (\Lambda_0,\Lambda_1,\lambda)} > 0.
\]
\end{thm}

\bigskip
\noindent
{\bf Proof of Theorem~\ref{thm-lower-bound-on-hpb-via-pers-modules}:}
Consider the exact symplectic cobordism
\[
\hM := \left(\hSigma \times [s_-,s_+], d(s\hlambda)\right).
\]
Define $\hX_0, \hX_1, \hY_0, \hY_1\subset \hM$ by
\[
\hX_0 := \Lambda_0\times \SP^1\times [s_-,s_+] = \hLambda_0\times [s_-,s_+],
\]
\[
\hX_1 := \Lambda_1\times \SP^1\times [s_-,s_+] = \hLambda_1\times [s_-,s_+],
\]
\[
\hY_0 := \Sigma\times T^* \SP^1\times s_-,
\]
\[
\hY_1 := \Sigma\times T^* \SP^1\times s_+.
\]
Since $\hl_{min,s_+/s_-} (\Lambda_0,\Lambda_1,\lambda) < +\infty$, Theorem~\ref{thm-lower-bound-on-pb-via-pers-modules}, applied to $\hM$, $\hX_0$, $\hX_1$, $\hY_0$, $\hY_1$, together with the inequality $pb^+_{\hM,comp} (\hX_0,\hX_1,\hY_0,\hY_1)\geq pb^+_\hM (\hX_0,\hX_1,\hY_0,\hY_1)$, yields
\begin{equation}
\label{eqn-pb-hM}
pb^+_{\hM,comp} (\hX_0,\hX_1,\hY_0,\hY_1)\geq
\frac{1}{(s_+ - s_-) \hl_{min,s_+/s_-} (\Lambda_0,\Lambda_1,\lambda)} > 0.
\end{equation}

Consider a map
\[
\hM = \hSigma \times [s_-,s_+] = \Sigma\times T^* \SP^1 (r,\tau)\times [s_-,s_+] \to
\]
\[
\to \tM = \Sigma \times [s_-,s_+] \times T^* \SP^1 (u,\tau)
\]
that sends each $(x,r,\tau,s)\in \Sigma\times T^* \SP^1 (r,\tau)\times [s_-,s_+]$ to $(x,s,u=sr,\tau)\in \Sigma \times [s_-,s_+] \times T^* \SP^1$.
(We use two copies of $T^* \SP^1$ -- one with the coordinates $r,\tau$ and one with the coordinates $u,\tau$).
It is a symplectomorphism -- it identifies the symplectic form
\[
d(s\hlambda) = d \left( s (\lambda + rd\tau)\right) = d(s\lambda) + d(sr d\tau) = \omega + d(srd\tau)
\]
on $\hM$ with the symplectic form
\[
\omega + du\wedge d\tau
\]
on $\tM$.
This symplectomorphism maps the sets $\hX_0, \hX_1, \hY_0, \hY_1\subset\hM$, respectively, to the sets $\tX_0, \tX_1, \tY_0, \tY_1\subset \tM$. Therefore \eqref{eqn-pb-hM} yields
\[
pb^+_{\hM,comp} (\hX_0,\hX_1,\hY_0,\hY_1) = pb^+_{\tM,comp} (\tX_0,\tX_1,\tY_0,\tY_1) \geq
\]
\[
\geq\frac{1}{(s_+ - s_-) \hl_{min,s_+/s_-} (\Lambda_0,\Lambda_1,\lambda)} > 0.
\]
Combining this with the inequality
\[
pb^+_{\tM_E,comp} \left(\tX_0,\tX_1,\tY_0 (E),\tY_1 (E)\right) \geq pb^+_{\tM,comp} (\tX_0,\tX_1,\tY_0,\tY_1)
\]
(that follows from \eqref{eqn-pb-monotonicity-2})
finishes the proof of the theorem.
\Qed

Assume that $\left( M,d(s\lambda)\right)$ is a  codimension-zero symplectic submanifold with boundary of a symplectic manifold $(N,\Omega)$ {which is closed as a subset of $N$}. Consequently,
$X_0,X_1,Y_0,Y_1$ can be viewed as subsets of $N$. Let $H: N\times \SP^1 (t)\to\R$ be a complete Hamiltonian.

\begin{cor}
\label{cor-chord-of-autonomous-H-in-N}
Assume that the Hamiltonian $H$ is time-independent and $l_{min,s_+/s_-} (\Lambda_0,\Lambda_1,\lambda) < +\infty$.

Then the following claims hold:

\bigskip
\noindent
A.  If $\Delta (H; Y_0,Y_1) >0$, then there exists a chord of $H$ from $X_0$ to $X_1$ of time-length $\displaystyle \leq \frac{(s_+ - s_-) l_{min,s_+/s_-} (\Lambda_0,\Lambda_1,\lambda)}{\Delta (H; Y_0,Y_1)}$.

\bigskip
\noindent
B. If $\supp H\cap M$ is compact and $H|_{X_0}\geq 0$ and $\Delta (H; X_1,X_0) >0$, then there exists a chord of $H$ from $Y_0$ to $Y_1$ of time-length bounded from above by $\displaystyle \frac{(s_+ - s_-) l_{min,s_+/s_-} (\Lambda_0,\Lambda_1,\lambda)}{\Delta (H; X_1,X_0)}$.
\end{cor}

\bigskip
\noindent
{\bf Proof of Corollary~\ref{cor-chord-of-autonomous-H-in-N}:}
Follows directly from \eqref{eqn-pb-monotonicity-1}, Proposition~\ref{prop-pb4+-modified-chords} and Theorem~\ref{thm-lower-bound-on-pb-via-pers-modules}.
\Qed

Let us now consider the case where $H$ is time-dependent. For each $t_0 \leq t$ denote by $\phi_{t_0,t}: N\to N$ the time-$[t_0,t]$ flow of $H$.
Set
\[
\Delta := \Delta (H; Y_0, Y_1),
\]
\[
\hl_{min,s_+/s_-} := \hl_{min,s_+/s_-} (\Lambda_0,\Lambda_1,\lambda),
\]
\[
c_{min} := \min_{X_0\times \SP^1} H,\ c_{max} := \max_{X_0\times \SP^1} H.
\]

\begin{cor}
\label{cor-chord-of-non-autonomous-H-in-N}
Let $0<e<1/2$. Set
\[
E:= e\Delta,
\]
\[
T:= \frac{(s_+ - s_-)\hl_{min,s_+/s_-}}{(1-2e)\Delta}.
\]
Assume that

\bigskip
\noindent
(a)  $\hl_{min,s_+/s_-} < +\infty$,

\bigskip
\noindent
(b) $\Delta > 0$,

\bigskip
\noindent
(c) $\displaystyle \sup_{c_{min} - E \leq H\leq c_{max} + E} |\partial H/\partial t| < E/T = \frac{e(1-2e)\Delta^2}{(s_+ - s_-) \hl_{min,s_+/s_-}}$.

\bigskip
Then there exists a chord of $H$ from $X_0$ to $X_1$ of time-length bounded from above by $\displaystyle T=\frac{\hl_{min,s_+/s_-} (s_+ - s_-)}{(1-2e)\Delta}$.

\end{cor}

\bigskip
\noindent
{\bf Proof of Corollary~\ref{cor-chord-of-non-autonomous-H-in-N}:}
Define
\[
\tN_E := N\times (-E, E)\times \SP^1
\]
and equip it with the symplectic form $\Omega + dr\wedge d\tau$, where $r,\tau$ are, respectively, the coordinates on $(-E,E)$ and $\SP^1$.
Then $\tM_E = M\times (-E, E)\times \SP^1$ is a closed codimension-zero symplectic submanifold with boundary of $\tN_E$.
With $\tX_0,\tX_1,\tY_0 (E),\tY_1 (E)\subset \tM_E\subset \tN_E$ defined as above, we get, by \eqref{eqn-pb-monotonicity-1},\eqref{eqn-pb-monotonicity-1-comp} and Theorem~\ref{thm-lower-bound-on-hpb-via-pers-modules}, that
\[
\tp_E:=pb^+_{\tN_E,comp} \left(\tX_0,\tX_1,\tY_0 (E),\tY_1 (E)\right) \geq
\]
\[
\geq pb^+_{\tM_E,comp} \left(\tX_0,\tX_1,\tY_0 (E),\tY_1 (E)\right) \geq \frac{1}{(s_+ - s_-) \hl_{min,s_+/s_-}} > 0.
\]
Note that for any $x\in N$ and $t_0\in\R$
\[
\frac{d}{dt} H(\phi_{t_0,t} (x),t) = \frac{\partial H}{\partial t} (\phi_{t_0,t} (x),t).
\]
Therefore condition
(c) implies that for any $x\in X_0$ and any $t_0\in\R$, $t\in [t_0, t_0+T]$,
\[
\phi_{t_0,t} (x)\in \{ c_{min} - E \leq H\leq c_{max} + E\}
\]
and
\[
|\partial H/\partial t (\phi_{t_0,t} (x),t)| < E/T.
\]
In particular, this means that
\[
\sup_{t_0\in\R} \big(\sup_{t\in [t_0, t_0+T]} H(\phi_{t_0,t} (x), t) - \inf_{t\in [t_0, t_0+T]} H(\phi_{t_0,t} (x), t)\big) < \frac{E}{T} \cdot T = E.
\]
This inequality, together with the positivity of $\tp_E$ and of $\Delta$, allows to apply Proposition~\ref{prop-chords-of-time-dep-Ham-s} to $H$, $N$ and $\tN_E$, which yields the existence of the chord of $H$ from $X_0$ to $X_1$ of time-length $\displaystyle \leq \frac{1}{\tp_E (\Delta - 2E)} \leq T$.
\Qed

\begin{rem}{\rm
\label{rem-estimate-on-pb-in-thm-lower-bound-on-pb-via-pers-modules-robust}
The claim on the positivity of $pb^+_M (X_0,X_1,Y_0,Y_1)$ in Theorem~\ref{thm-lower-bound-on-pb-via-pers-modules}
can be generalized to arbitrary two-part exact Lagrangian cobordisms, with $(s_+ - s_-) l_{min,s_+/s_-} (\Lambda_0,\Lambda_1,\lambda)$ being replaced by a certain invariant associated to the morphism of the persistence modules defined by the cobordism. In particular, the positivity of $pb^+_M (X_0,X_1,Y_0,Y_1)$ remains true
if the sets $X_0$, $X_1$ as in Theorem~\ref{thm-lower-bound-on-pb-via-pers-modules} are perturbed as exact Lagrangian cobordisms in $\left(M,d(s\lambda)\right)$ (cylindrical near the boundaries) so that the Legendrian isotopies, induced on the boundaries by the exact Lagrangian isotopies, are sufficiently $C^1$-small. (Note that away from the boundary the perturbations may be arbitrarily large, as long as the perturbed $X_0$, $X_1$ are disjoint!). The lower bound on $pb^+_M (X_0,X_1,Y_0,Y_1)$ for the perturbed $X_0$, $X_1$ is then only slightly larger than the one appearing in Theorem~\ref{thm-lower-bound-on-pb-via-pers-modules} -- the difference between the bounds tends to zero as the sizes of the Legendrian isotopies above tend to zero. In the particular case where a trivial Lagrangian cobordism is deformed among trivial Lagrangian cobordisms the robustness follows from Theorem~\ref{thm-lower-bound-on-pb-via-pers-modules} and Remark~\ref{rem-l-min-s-robust}.

For the proofs and details see \cite{EP-big}.

Consequently, the results of Theorem~\ref{thm-lower-bound-on-hpb-via-pers-modules} and Corollaries~\ref{cor-chord-of-autonomous-H-in-N} and \ref{cor-chord-of-non-autonomous-H-in-N} are also robust with respect to the above perturbations.
}
\end{rem}

\section{Applications to contact dynamics}
\label{sec-applications-to-contact-dynamics}

With $\Sigma$ and $\lambda$ as above, assume that $\Lambda = \Lambda_0\sqcup \Lambda_1$ is a two-part Legendrian submanifold in $(\Sigma,\xi = \ker\Lambda)$.

Let $h: \Sigma\times\SP^1\to \R$ be a complete (time-dependent) contact Hamiltonian (with respect to $\lambda$). Set $h_t := h (\cdot, t): \Sigma\to\R$.
Let $v_t$, $t\in\SP^1$, denote the contact vector field of $h_t$. If $h$ and $v$ are time-independent we write $v$ instead of $v_t$.
 Denote by $\{ \varphi_t\}$ the {time-$[0,t]$} contact flow of $h$ -- that is, the {time-$[0,t]$} flow of $v_t$. The flow $\{ \varphi_t\}$ lifts to a Hamiltonian flow $\{ \phi_t\}$ on $(\Sigma\times \R_+, d(s\lambda))$ equivariant with respect to the multiplicative $\R_+$-action on $\Sigma\times \R_+$ and generated by the Hamiltonian $H:\Sigma\times\R_+\times\SP^1\to\R$, $H (y,s,t):=s\cdot h (y,t)$.
The flow $\{ \phi_t\}$ has the form
\begin{equation}
\label{eqn-psi-t}
\phi_t (y,s) = \left(\varphi_t \left(y\right), s   \frac{\left(\varphi_t^{-1}\right)^* \lambda \left(\varphi_t \left(y\right)\right)}{\lambda \left(\varphi_t \left(y\right)\right)} \right).
\end{equation}
Since the contact flow $\{ \varphi_t\}$ of $h$ is defined for all times, so is the Hamiltonian flow $\{ \phi_t\}$ of $H$, meaning that $H$ is complete.

Let us recall Definition~\ref{def-cooperative-ham}.

\begin{defin}
\label{def-cooperative-ham-gen-case}
{\rm
Let us say that $h: \Sigma\times\SP^1\to\R$ is {\it $C$-cooperative with  $\Lambda_0$, $\Lambda_1$} for $C>0$ if either of the following conditions holds:

\bigskip
\noindent
(a)
$h < C$ on $\Lambda_1\times \SP^1$ and either the set $\{ h\geq C\} = \bigcup_{t\in\SP^1} \{ h_t\geq C\}$ is empty or $dh_t (R)\geq 0$ on $\{ h_t\geq C\}$ for all $t\in\SP^1$.

\bigskip
\noindent
(b)
$h < C$ on $\Lambda_0\times \SP^1$ and either the set $\{ h\geq C\} = \bigcup_{t\in\SP^1} \{ h_t\geq C\}$ is empty or $dh_t (R)\leq 0$ on $\{ h_t\geq C\}$ for all $t\in\SP^1$.

\bigskip
\noindent
We will say that $h$ is {\it cooperative with  $\Lambda_0$, $\Lambda_1$} if it is $C$-cooperative with  $\Lambda_0$, $\Lambda_1$ for some $C>0$.
}
\end{defin}
\bigskip

\subsection{Largeness of the conformal factor of $\varphi_t$}
\label{subsec-conformal-factor}

\begin{thm}
\label{thm-conformal-factor-gen-case}
Assume that $h$ is time-independent and compactly supported. Assume also that
\[
h|_{\Lambda_0} {\geq} 0,\ h|_{\Lambda_1} {<}0.
\]
and the pair $(\Lambda_0\sqcup\Lambda_1,\lambda)$ is {weakly} homologically bonded.

{
Then the conformal factor of $\varphi_t$ takes arbitrarily large values as $t$ varies between $0$ and $+\infty$:
\[
\inf\limits_{t\in (0,+\infty), y\in\Sigma} {\frac{\left(\varphi_t^{-1}\right)^* \lambda \left(\varphi_t \left(y\right)\right)}{\lambda \left(\varphi_t \left(y\right)\right)} = +\infty}.
\]
}
\end{thm}

\bigskip
\noindent
\begin{rem}
\label{rem-conf-factor-contact-hams-with-non-compact-support}
{\rm
It would be interesting to generalize Theorem~\ref{thm-conformal-factor-gen-case} to contact Hamiltonians $h$ that are not necessarily compactly supported but rather are constant outside a compact set $K\subset\Sigma$, meaning that the contact Hamiltonian flow of such an $h$ outside $K$ is a reparameterized Reeb flow and the conformal factor of the flow outside $K$ is identically equal to $1$, since the Reeb flow preserves the contact form. (Recall that, by our assumptions, the Reeb flow is defined for all times so such an $h$ would be complete). Such a generalization would be true if {\sl in this particular setting} -- for $Y_0=\Sigma\times s_-$, $Y_1=\Sigma\times s_+$ and the vector field $\sgrad (sh)$ on $\left(\Sigma\times\R_+, d(s\lambda)\right)$ --  Fathi's Theorem~\ref{thm-fathi} would provide a function $G:\Sigma\times\R_+\to\R$, $G\in \cS' (Y_0,Y_1)$ (see Proposition~\ref{prop-L-vs-L-c}), so that the flow of $\sgrad G$ on $\Sigma\times [s_-,s_+]$ is complete.
}
\end{rem}
\bigskip

\bigskip
\noindent
{\bf Proof of Theorem~\ref{thm-conformal-factor-gen-case}:}
Pick $0<s_- < s_+$. Let $X_0,X_1,Y_0,Y_1\subset \Sigma\times [s_-,s_+]$ be the admissible quadruple defined for $s_-, s_+$ as in \eqref{eqn-def-X0-X1},\eqref{eqn-def-Y0-Y1}.

Clearly, $H|_{X_0} \geq 0$, $\Delta (H; X_1,X_0) >0$.

We claim that there is a chord of $H$ from $Y_0$ to $Y_1$. Indeed, consider a cut-off function $\chi:\R_+\to [0,1]$ which
equals $1$ on $[s_-,s_+]$ and $0$ outside $(s_- -\epsilon, s_+ + \epsilon)$ for some $0<\epsilon<s_-$. Since $h:\Sigma\to\R$ is compactly supported, the Hamiltonian $\chi (s) h: \Sigma\times\R_+\to\R$ is also compactly supported and coincides with $H=sh$ on $\Sigma\times [s_-,s_+]$. In particular, $\chi (s)H|_{X_0} \geq 0$, $\Delta (\chi (s)H; X_1,X_0)>0$.
Together with the assumption that the pair $(\Lambda_0\sqcup\Lambda_1,\lambda)$ is weakly homologically bonded, this allows to apply part B of Corollary~\ref{cor-chord-of-autonomous-H-in-N} to the symplectic manifold $(N,\Omega)=\left(\Sigma\times\R_+, d(s\lambda)\right)$ and the Hamiltonian $\chi (s)h$. This yields the existence of a chord of $\chi(s) h$ from $Y_0$ to $Y_1$ in $N$. An easy topological argument then yields that there exists a chord of $\chi(s) h$ from $Y_0$ to $Y_1$ in $N$ that lies in $\Sigma\times [s_-,s_+]$ and therefore is a chord $H$. This proves the claim.

The existence of the chord of $H$ from $Y_0$ to $Y_1$, together with \eqref{eqn-psi-t}, implies that
\[
\left(\varphi_t^{-1}\right)^* \lambda \left(\varphi_t \left(y\right)\right)/\lambda \left(\varphi_t \left(y\right)\right) \geq s_+/s_-.
\]
Since ${s_+/s_-}$ can be made arbitrarily {large}, we get that
\[
\inf\limits_{t\in (0,+\infty), y\in\Sigma} {\frac{\left(\varphi_t^{-1}\right)^* \lambda \left(\varphi_t \left(y\right)\right)}{\lambda \left(\varphi_t \left(y\right)\right)} = +\infty},
\]
which finishes the proof of the theorem.
\Qed

\subsection{Existence of chords of $h$ from $\Lambda_0$ to $\Lambda_1$}
\label{subsec-existence-of-chords-from-Lambda0-to-Lambda1}

\begin{thm}[Cf. Rem. 1.14 in \cite{EP-tetragons}]
\label{thm-Reeb-chords-pos-Ham-gen-case}
Assume
\[
0<\inf_{\Sigma\times\SP^1} h \leq \sup_{\Sigma\times\SP^1} h < +\infty,
\]
and let
\[
s_+ > \frac{\sup_{\Sigma\times\SP^1} h}{\inf_{\Sigma\times\SP^1} h}\geq 1.
\]

Then the following claims hold:

\bigskip
\noindent
A. Assume that $h$ is time-independent and $l_{min,s_+} (\Lambda_0,\Lambda_1,\lambda) =: l_{min,s_+}< +\infty$.

Then there exists a chord of $h$ from $\Lambda_0$ to $\Lambda_1$ of time-length bounded from above by
$\displaystyle \frac{ (s_+ - 1) l_{min,s_+}}{s_+ \inf_\Sigma h - \sup_\Sigma h}$.

\bigskip
\noindent
B.
Assume that $\hl_{min,s_+} (\Lambda_0,\Lambda_1,\lambda)  =: \hl_{min,s_+} < +\infty$. Let
\[
\Delta_{s_+} := s_+\inf_{\Sigma\times \SP^1} h - \sup_{\Sigma\times \SP^1} h.
\]
Assume also that for some $0<e<1/2$
\begin{equation}
\label{eqn-upper-bound-on-partial-h-over-partial-t}
\sup_{\Sigma\times \SP^1} |\partial h/\partial t| < \frac{(1-2e)e\Delta_{s_+}^2 \inf_{\Sigma\times \SP^1} h} {(s_+ - 1)(s_+ \max_{\Lambda_0\times \SP^1} h + e\Delta_{s_+})  \hl_{min,s_+} }.
\end{equation}

Then there exists a chord of $h$ from $\Lambda_0$ to $\Lambda_1$ of time-length bounded from above by
$\displaystyle \frac{(s_+ -1)\hl_{min,s_+}}{(1-2e)\Delta_{s_+}}$.
\end{thm}

\bigskip
\noindent
{\bf Proof of Theorem~\ref{thm-Reeb-chords-pos-Ham-gen-case}:}
Pick $s_- := 1 < s_+$. Let $X_0,X_1,Y_0,Y_1\subset \Sigma\times [1,s_+]$ be the admissible quadruple defined for $s_- = 1, s_+$ as in \eqref{eqn-def-X0-X1},\eqref{eqn-def-Y0-Y1}.
Clearly,
\[
\Delta(H;Y_0,Y_1) := \Delta_{s_+} = s_+ \inf_\Sigma h - \sup_\Sigma h.
\]
Since, by the hypothesis of the theorem,
\[
0<\inf_{\Sigma\times \SP^1} h \leq \sup_{\Sigma\times \SP^1} h < +\infty,
\]
we get that $\Delta(H;Y_0,Y_1) = \Delta_{s_+} >0$ if $\displaystyle s_+ > \frac{\sup_{\Sigma\times \SP^1} h}{\inf_{\Sigma\times \SP^1} h}$.

Let us now prove part A of the theorem. Its hypothesis allows to apply
part A of Corollary~\ref{cor-chord-of-autonomous-H-in-N} to $(N,\Omega)=\left(\Sigma\times\R_+, d(s\lambda)\right)$ and the Hamiltonian $H$ on it as long as $\displaystyle s_+ > \frac{\sup_{\Sigma\times \SP^1} h}{\inf_{\Sigma\times \SP^1} h}$. This yields the existence of a Hamiltonian chord of $H$ from $X_0$ to $X_1$ of time-length bounded from above by
\[
\frac{(s_+ - 1) l_{min,s_+} (\Lambda_0,\Lambda_1,\lambda)}{s_+ \inf_\Sigma h - \sup_\Sigma h}.
\]
The projection of this Hamiltonian chord to $\Sigma$ is a chord of $h$ from $\Lambda_0$ to $\Lambda_1$ of the same time-length.
This finishes the proof of part A the theorem.

Let us prove part B of the theorem. Similarly to the setting of Corollary~\ref{cor-chord-of-non-autonomous-H-in-N}, for a given $0<e<1/2$ define
\[
E:= e\Delta_{s_+},
\]
\[
T:= \frac{(s_+ - 1)\hl_{min,s}}{(1-2e)\Delta_{s_+}},
\]
\[
c_{min} := \min_{X_0\times \SP^1} H = \min_{\Lambda_0\times \SP^1} h.
\]
\[
c_{max} := \max_{X_0\times \SP^1} H = s_+ \max_{\Lambda_0\times \SP^1} h.
\]

Consider the set $S:=\{c_{min} - E \leq H=sh\leq c_{max} + E\}$.
We would like to apply Corollary~\ref{cor-chord-of-non-autonomous-H-in-N} and in order to this we need to verify that the upper bound
on the restriction of the function $|\partial H/\partial t|$ to $S$, required in Corollary~\ref{cor-chord-of-non-autonomous-H-in-N}, does hold in our case.

Note that on $S$
\[
s\leq \frac{c_{max} + E}{\inf_{\Sigma\times \SP^1} h}.
\]
Together with the upper bound
on the restriction of the function $|\partial H/\partial t|= s |\partial h/\partial t|$ to $S$
in the hypothesis of part B of the theorem, this yields the following upper bound on the function $|\partial H/\partial t| = s |\partial h/\partial t|$ on the set $S$:
\[
\sup_S s |\partial h/\partial t| \leq \frac{c_{max} + E}{\inf_{\Sigma\times \SP^1} h} \cdot \sup_{\Sigma\times \SP^1} |\partial h/\partial t| <
\]
\[
< \frac{c_{max} + E}{\inf_{\Sigma\times \SP^1} h} \cdot \frac{(1-2e)e\Delta_{s_+}^2 \inf_{\Sigma\times \SP^1} h} {(s_+ - 1) (s_+ \max_{\Lambda_0\times \SP^1} h + e\Delta_{s_+}) \hl_{min,s_+}} =
\]
\[
= \frac{c_{max} + E}{\inf_{\Sigma\times \SP^1} h} \cdot \frac{(1-2e)e\Delta_{s_+}^2 \inf_{\Sigma\times \SP^1} h} {(s_+ - 1) (c_{max} + E) \hl_{min,s_+}} =
\]
\[
= \frac{(1-2e)e\Delta_{s_+}^2} {(s_+ - 1) \hl_{min,s_+}} = \frac{E}{T},
\]
yielding the bound required in Corollary~\ref{cor-chord-of-non-autonomous-H-in-N}. Thus, Corollary~\ref{cor-chord-of-non-autonomous-H-in-N}
can be applied to $(N,\Omega)=\left(\Sigma\times\R_+, d(s\lambda)\right)$ and the Hamiltonian $H$ on it
(since, by our assumptions, $\hl_{min,s_+} < +\infty$ and $\Delta_{s_+} > 0$), which yields the existence of
a chord of $h$ from $\Lambda_0$ to $\Lambda_1$ of time-length
$\displaystyle \leq  T=\frac{(s_+ - 1)\hl_{min,s}}{(1-2e)\Delta_{s_+}}$.

This finishes the proof of part B of the theorem.
\Qed

\begin{cor}
\label{cor-homol-bonded-autonom-contact-dynamics-general-case}
Assume that $h$ is time-independent and $\inf_\Sigma h >0$.

If $h$ is $C$-cooperative with $\Lambda_0, \Lambda_1$ for some $C>\inf_\Sigma h$ (see Definition~\ref{def-cooperative-ham}) and the pair $(\Lambda_0\sqcup\Lambda_1,\lambda)$ is {weakly} homologically bonded, then
there exists a chord of $h$ from $\Lambda_0$ to $\Lambda_1$ of time-length
$\displaystyle \leq\inf_{s> C/\inf_\Sigma h} \frac{(s - 1) l_{min,s} (\Lambda_0,\Lambda_1,\lambda)}{s \inf_\Sigma h - C}$.

Furthermore, if $h$ is cooperative with $\Lambda_0, \Lambda_1$ and the pair $(\Lambda_0\sqcup\Lambda_1,\lambda)$ is homologically bonded, the time-length of the chord can be also bounded from above by $\mu:=l_{min,\infty} (\Lambda_0,\Lambda_1,\lambda)/\inf_\Sigma h$. In particular, the pair $(\Lambda_0,\Lambda_1)$ is $\mu$-interlinked.
\end{cor}

\bigskip
\noindent
{\bf Proof of Corollary~\ref{cor-homol-bonded-autonom-contact-dynamics-general-case}:}
Let us assume that {condition (a) from Definition~\ref{def-cooperative-ham}} of $C$-cooperativeness is satisfied -- {that is,
$h < C$ on $\Lambda_1$ and either the set $\{ h\geq C\}$ is empty or $dh (R)\geq 0$ on $\{ h\geq C\}$
(the case of condition (b) from the same definition is similar).}

Consider a smooth increasing function $\chi: \R_+\to\R$ such that $\chi (s) = s$ for $s \in [0,C]$ and $\chi (s) = C+\epsilon$ for some $\epsilon >0$ and all sufficiently large $s\in\R_+$. Consider the time-independent contact Hamiltonian $\tih := \chi \circ h$. One readily sees that it is complete and satisfies $0 < \inf_\Sigma \tih \leq \sup_\Sigma \tih \leq C+\epsilon < +\infty$. Since the pair $(\Lambda_0\sqcup\Lambda_1,\lambda)$ is {weakly} homologically bonded, part A of Theorem~\ref{thm-Reeb-chords-pos-Ham-gen-case}, applied to $\tih$, shows that there exists a chord $\gamma: [0,T] \to \Sigma$ of $\tih$ such that $\gamma (0)\in \Lambda_0$ and $\gamma (T)\in\Lambda_1$ for
$\displaystyle T\leq T_0:= \inf_s \frac{(s - 1) l_{min,s} (\Lambda_0,\Lambda_1,\lambda)}{s \inf_\Sigma \tih - \sup_\Sigma \tih }$, where the infimum is taken over all $\displaystyle s> \frac{\sup_\Sigma \tih}{\inf_\Sigma \tih}$.

We claim that $\gamma ([0,T])$ lies in the set $\{ \tih {\geq} C\}$. Indeed, for all $t$
\[
d(\tih\circ \gamma)/dt = d\tih (R)\cdot  \tih = (\chi'\circ h)\cdot dh (R)\cdot \tih.
\]
Therefore, if $\tih \left(\gamma (t_0)\right) {>} C$ for some $t_0\in [0,T]$, then $\tih \left(\gamma (t)\right) \geq C$ for all $t\in [t_0,T]$, in contradiction to $\tih \left(\gamma (T)\right) = h \left(\gamma (T)\right)< C$ (the latter holds since $\gamma (T)\in\Lambda_1$  and $h< C$ on $\Lambda_1$). Thus, $\gamma ([0,T])$ lies in $\{ \tih {\leq} C\}$ where $\tih$ coincides with $h$, meaning that $\gamma$ is, in fact, the chord of $h$ of time-length $T$ bounded from above by $T_0$. Since $\inf_\Sigma h = \inf_\Sigma \tih$, $\sup_\Sigma \tih \leq C+\epsilon$ and $\epsilon$ can be taken arbitrarily small, we get that $\displaystyle T_0\leq \inf_{s>C/\inf_\Sigma h} \frac{(s - 1) l_{min,s} (\Lambda_0,\Lambda_1,\lambda)}{s \inf_\Sigma h - C }$, yielding the required upper bound on the time-length of the chord.

If the pair $(\Lambda_0\sqcup\Lambda_1,\lambda)$ is homologically bonded, then we can replace in the bound above $l_{min,s} (\Lambda_0,\Lambda_1,\lambda)$
by $l_{min,\infty} (\Lambda_0,\Lambda_1,\lambda)$, remove the infimum and let $s$ go to $+\infty$. This shows that the time-length of the chord is bounded from above by $l_{min,\infty} (\Lambda_0,\Lambda_1,\lambda)/\inf_\Sigma h$.

This finishes the proof of the corollary.
\Qed

\begin{cor}
\label{cor-stably-homol-bonded-non-autonom-contact-dynamics-general-case}
Assume that $\inf_{\Sigma\times\SP^1} h >0$, $h$ is cooperative with $\Lambda_0, \Lambda_1$ and the pair $(\Lambda_0\sqcup\Lambda_1,\lambda)$ is stably homologically bonded. Denote $\hl_{min,\infty} :=\hl_{min,\infty} (\Lambda_0,\Lambda_1,\lambda)>0$. Assume also that for some $0<e<1/2$
\[
\sup_{\Sigma\times \SP^1} |\partial h/\partial t| < \frac{(1-2e)e\big(\inf_{\Sigma\times \SP^1} h\big)^3} {\big(\max_{\Lambda_0\times \SP^1} h + e\inf_{\Sigma\times \SP^1} h  \big)\hl_{min,\infty} }.
\]

Then there exists a chord of $h$ from $\Lambda_0$ to $\Lambda_1$ of time-length bounded from above by
$\displaystyle \frac{\hl_{min,\infty}}{(1-2e)\inf_{\Sigma\times \SP^1} h}$.
\end{cor}

\bigskip
\noindent
{\bf Proof of Corollary~\ref{cor-stably-homol-bonded-non-autonom-contact-dynamics-general-case}:}
Let us assume that $h$ is $C$-cooperative with $\Lambda_0, \Lambda_1$ for some $C>0$
and condition (a) from Definition~\ref{def-cooperative-ham} of $C$-cooperativeness is satisfied
(the case of condition (b) from the same definition is similar). Without loss of generality, assume $C > \inf_{\Sigma\times\SP^1} h$.

Similarly to the proof of Corollary~\ref{cor-homol-bonded-autonom-contact-dynamics-general-case},
consider a smooth increasing function $\chi: \R_+\to\R$ such that $\chi (s) = s$ for $s \in [0,C]$ and $\chi (s) = C+\epsilon$ for some $\epsilon >0$ and all sufficiently large $s\in\R_+$. Consider the time-dependent contact Hamiltonian $\tih := \chi \circ h$. One readily sees that it is complete and satisfies
\[
0 < \inf_{\Sigma\times\SP^1} h = \inf_{\Sigma\times\SP^1} \tih \leq \sup_{\Sigma\times\SP^1} \tih \leq C+\epsilon < +\infty.
\]
Note that for any sufficiently large $s_+ >1$
\[
\Delta_{s_+} := s_+  \inf_{\Sigma\times\SP^1} \tih - \sup_{\Sigma\times\SP^1} \tih >0.
\]
Also note that
\[
\lim\limits_{s_+\to +\infty} \frac{(1-2e)e\Delta_{s_+}^2 \inf_{\Sigma\times \SP^1} h} {(s_+ - 1)(s_+ \max_{\Lambda_0\times \SP^1} h + e\Delta_{s_+})} =
\frac{(1-2e)e\big(\inf_{\Sigma\times \SP^1} h\big)^3} {\max_{\Lambda_0\times \SP^1} h + e\inf_{\Sigma\times \SP^1} h}
\]
and for all $s_+>1$
\[
\hl_{min,s_+} (\Lambda_0,\Lambda_1,\lambda)\leq \hl_{min,\infty} (\Lambda_0,\Lambda_1,\lambda).
\]
Therefore the upper bound on $\sup_{\Sigma\times \SP^1} |\partial h/\partial t|$ in the hypothesis of the corollary implies that for any sufficiently large $s_+$ we can bound $\sup_{\Sigma\times \SP^1} |\partial \tih/\partial t|$ from above as in \eqref{eqn-upper-bound-on-partial-h-over-partial-t}:
\[
\sup_{\Sigma\times \SP^1} |\partial \tih/\partial t| < \frac{(1-2e)e\Delta_{s_+}^2 \inf_{\Sigma\times \SP^1} h} {(s_+ - 1)(s_+ \max_{\Lambda_0\times \SP^1} h + e\Delta_{s_+})  \hl_{min,{s_+}} }.
\]
Since the pair $(\Lambda_0\sqcup\Lambda_1,\lambda)$ is stably homologically bonded, part B of Theorem~\ref{thm-Reeb-chords-pos-Ham-gen-case}, applied to $\tih$, shows that there exists a chord $\gamma_{s_+}: [t_{s_+},t_{s_+}+T_{s_+}] \to \Sigma$ of $\tih$, for some $t_{s_+}\in\R$, so that $\gamma_{s_+} (t_{s_+})\in \Lambda_0$ and $\gamma_{s_+} (t_{s_+}+T_{s_+})\in\Lambda_1$ for $\displaystyle T_{s_+}\leq \displaystyle \frac{(s_+ - 1) \hl_{min,s_+}}{(1-2e)\Delta_{s_+}}$.
Similarly to the proof of Corollary~\ref{cor-homol-bonded-autonom-contact-dynamics-general-case}, we get that $\gamma_{s_+}$ is in fact a chord of $h$.

Note that
\[
\lim_{s_+\to +\infty} \frac{ (s_+ - 1) \hl_{min,s_+}}{(1-2e)\Delta_{s_+}} \leq \lim_{s_+\to +\infty} \frac{ (s_+ - 1) \hl_{min,\infty}}{(1-2e)\Delta_{s_+}}
= \frac{\hl_{min,\infty}}{(1-2e)\inf_{\Sigma\times \SP^1} h}.
\]
Also note that since $h$ is time-periodic with period $1$, we can assume that $t_{s_+}\in [0,1]$ for all $s_+$.
Now, since $\Lambda_0$ is compact, a standard compactness argument allows to obtain the existence of a chord of $h$ from $X_0$ to $X_1$  of
time-length $\displaystyle \leq\frac{ (s_+ - 1) \hl_{min,\infty}}{(1-2e)\Delta_{s_+}}$.
\Qed

The following corollary yields the existence of a chord of $h$ in case where $h$ is not necessarily everywhere positive.

\begin{cor}
\label{cor-non-positive-hamiltonians-separating-hypersurface-gen-case}
Assume there exists a (possibly non-compact or disconnected) closed codimension 0 submanifold $\Xi\subset \Sigma$ with a (possibly non-compact or disconnected) boundary $\partial\Xi$,
so that

\bigskip
\noindent
(1) $\inf_{\Xi\times\SP^1} h >0$ (but $h$ may be negative outside $\Xi \times\SP^1$).

\bigskip
\noindent
(2) $\sup_{\partial\Xi\times\SP^1} h < +\infty$.

\bigskip
\noindent
(3) For each $t\in\SP^1$
the {contact Hamiltonian vector field $v_t$ of $h$} is transverse to $\partial\Xi$ (in particular, $\partial\Xi$ is a convex surface in the sense of contact topology -- see \cite{Giroux-CMH1991})
and either points inside $\Xi$ everywhere on $\partial\Xi$ or points outside $\Xi$ everywhere on $\partial\Xi$.

\bigskip
\noindent
(4)
Both $\Lambda_0$ and $\Lambda_1$ lie in $\Xi$.

\bigskip
Then the following claims hold:

\bigskip
\noindent
(I) Assume the pair $(\Lambda_0\sqcup\Lambda_1,\lambda)$ is weakly homologically bonded. Assume also that $h$ is time-independent and $C$-cooperative with $\Lambda_0,\Lambda_1$ for $C > \inf_\Xi h$.

Then there exists a chord of $h$ from $\Lambda_0$ to $\Lambda_1$ of time-length bounded from above by
$\displaystyle \inf_{s>C/\inf_\Xi h} \frac{(s - 1) l_{min,s} (\Lambda_0,\Lambda_1,\lambda)}{s \inf_\Xi h - C}$.

If the pair $(\Lambda_0\sqcup\Lambda_1,\lambda)$ is homologically bonded, then the time-length of the chord can be bounded from above by
$\displaystyle \frac{l_{min,\infty} (\Lambda_0,\Lambda_1,\lambda)}{\inf_\Xi h}$.

\bigskip
\noindent
(II) Assume that the pair $(\Lambda_0\sqcup\Lambda_1,\lambda)$ is stably homologically bonded and set $\hl_{min,\infty} := \hl_{min,\infty} (\Lambda_0,\Lambda_1,\lambda)$. Assume also that $h$ is time-dependent and cooperative with $\Lambda_0,\Lambda_1$ and for some $0<e<1/2$
\[
\sup_{\Xi\times \SP^1} |\partial h/\partial t| < \frac{(1-2e)e\big(\inf_{\Xi\times \SP^1} h\big)^3} {\big(\max_{\Lambda_0\times \SP^1} h + e\inf_{\Xi\times \SP^1} h  \big)\hl_{min,\infty} }.
\]

Then there exists a chord of $h$ from $\Lambda_0$ to $\Lambda_1$ whose time-length is bounded from above by $\displaystyle \frac{\hl_{min,\infty}}{(1-2e)\inf_{\Xi\times \SP^1} h}$.

\end{cor}
\bigskip

\bigskip
\noindent
{\bf Proof of Corollary~\ref{cor-non-positive-hamiltonians-separating-hypersurface-gen-case}:}
Let us assume that $h$ is $C$-cooperative with $\Lambda_0$, $\Lambda_1$  for $C\geq \sup_{\partial\Xi\times\SP^1} h$ (this is possible since, by (2), $\sup_{\partial\Xi\times\SP^1} h < +\infty$).

For any sufficiently small $\epsilon >0$ one can find a new contact Hamiltonian $\tih_\epsilon: \Sigma\times\SP^1\to\R$ so that $\tih = h$ on a neighborhood of $\Xi$ and
$\inf_{\Xi\times\SP^1} h -\epsilon\leq \tih_\epsilon \leq \sup_{\Xi\times\SP^1} h$
on $(\Sigma\setminus\Xi)\times\SP^1$. Since $C\geq \sup_{\partial\Xi\times\SP^1} h$, the contact Hamiltonian
$\tih_\epsilon$ satisfies the assumptions of Corollary~\ref{cor-homol-bonded-autonom-contact-dynamics-general-case} (in case (I)), or of Corollary~\ref{cor-stably-homol-bonded-non-autonom-contact-dynamics-general-case} (in case (II)). Consequently, by these corollaries,
there exists a chord $\gamma_\epsilon (t)$ of $\tih_\epsilon$, $\gamma_\epsilon (t_\epsilon)\in\Lambda_0$, $\gamma_\epsilon (t_\epsilon+T_\epsilon)\in\Lambda_1$, so that

\bigskip
\noindent
- In case (I):
\[
T_\epsilon\leq\inf_s \frac{(s - 1) l_{min,s} (\Lambda_0,\Lambda_1,\lambda)}{s \inf_\Sigma \tih_\epsilon - C}\leq \inf_s \frac{(s - 1) l_{min,s} (\Lambda_0,\Lambda_1,\lambda)}{s (\inf_\Xi h -\epsilon) - C}
\]
(the infimum is taken over all $s>C/\inf_\Sigma \tih_\epsilon$), and
\[
T_\epsilon\leq\frac{l_{min,\infty} (\Lambda_0,\Lambda_1,\lambda)}{\inf_\Sigma \tih_\epsilon}\leq \frac{l_{min,\infty} (\Lambda_0,\Lambda_1,\lambda)}{\inf_\Xi h - \epsilon}.
\]

\bigskip
\noindent
- In case (II):
\[
T_\epsilon\leq\frac{\hl_{min,\infty}}{(1-2e)\inf_{\Sigma\times \SP^1} \tih_\epsilon}\leq \frac{\hl_{min,\infty}}{(1-2e)\inf_{\Xi\times \SP^1} h -\epsilon}.
\]

\bigskip
\noindent
Here in both cases we have used that $\inf_\Sigma \tih_\epsilon \geq \inf_\Xi h - \epsilon$ for all (sufficiently small) $\epsilon>0$.

The chord $\gamma_\epsilon$ cannot cross $\partial\Xi$. Indeed, by (3), if it had crossed $\partial\Xi$, it would have had to cross it transversally from $\Xi$ to $M\setminus \Xi$. This would mean that $v_t$ {(the contact Hamiltonian vector field of $\tih_\epsilon|_\Xi=h|_\Xi$)} points outside $\Xi$ everywhere on $\partial\Xi$ for all $t\in\SP^1$. But then the chord would not have been able to return to $\Xi$ to reach $\Lambda_1$. Thus the chord lies in the interior of $\Xi$ and is, in fact, a chord of $h$ from $\Lambda_0$ to $\Lambda_1$.

We have such a chord $\gamma_\epsilon$ of $h$ from $\Lambda_0$ to $\Lambda_1$ for any sufficiently small $\epsilon>0$ and the time-lengths of the chords admit a bound continuous in $\epsilon$. We can also assume that for all $\epsilon>0$ $t_{s_\epsilon} =0$ in case (I) (since $h$ is time-independent) and
$t_{s_\epsilon} \in [0,1]$ in case (II) (since $h$ is 1-periodic in time). Now, since $\Lambda_0$ is compact, a standard compactness argument allows to
obtain the existence of the chord of $h$ from $X_0$ to $X_1$ of time-length $T$, so that

\bigskip
\noindent
- In case (I):
\[
T\leq\inf_{s>C/\inf_\Xi h} \frac{(s - 1) l_{min,s} (\Lambda_0,\Lambda_1,\lambda)}{s \inf_\Xi h - C}
\]
and if the pair $(\Lambda_0\sqcup\Lambda_1,\lambda)$ is homologically bonded, then
\[
T\leq\frac{l_{min,\infty} (\Lambda_0,\Lambda_1,\lambda)}{\inf_\Xi h}.
\]

\bigskip
\noindent
- In case (II)
\[
T\leq\frac{\hl_{min,\infty}}{(1-2e)\inf_{\Xi\times \SP^1} h}.
\]

This finishes the proof of the corollary.
\Qed

\bigskip

\begin{rem}
\label{rem-existence-of-chords-of-h-from-Lambda0-to-Lambda1-robust}
{\rm
The claim of Theorem~\ref{thm-Reeb-chords-pos-Ham-gen-case} is robust -- for a fixed $h$ -- with respect to perturbations of $\Lambda=\Lambda_0\sqcup\Lambda_1$ by Legendrian isotopies, as long as the perturbation is sufficiently {$C^1$-small}, depending on {$s_+$} -- this follows from Remark~\ref{rem-l-min-s-robust}.

Accordingly, the claims of Corollaries~\ref{cor-homol-bonded-autonom-contact-dynamics-general-case}, \ref{cor-stably-homol-bonded-non-autonom-contact-dynamics-general-case} and \ref{cor-non-positive-hamiltonians-separating-hypersurface-gen-case}
are robust -- for a fixed $h$ -- with respect to  perturbations of $\Lambda=\Lambda_0\sqcup\Lambda_1$ by Legendrian isotopies, as long as the perturbation is sufficiently {$C^1$-small}, depending on $\displaystyle \frac{C}{\inf_{\Sigma\times\SP^1} h}$, {where $C$ is the constant such that $h$ is $C$-cooperative with $\Lambda_0$, $\Lambda_1$.}

Namely, if the pair $(\Lambda_0\sqcup\Lambda_1,\lambda)$ is {weakly} homologically bonded, then $l_{min,s} (\Lambda_0,\Lambda_1,\lambda)< +\infty$ for any $s> \displaystyle \frac{C}{\inf_{\Sigma\times\SP^1} h}$. Fix such an $s$. Then, by Remark~\ref{rem-l-min-s-robust}, $l_{min,s} (\Lambda'_0,\Lambda'_1,\lambda)$ is finite and close to $l_{min,s} (\Lambda_0,\Lambda_1,\lambda)$ for any $\Lambda' = \Lambda'_0\sqcup\Lambda'_1$ obtained from $\Lambda$ by a Legendrian isotopy, as long as the isotopy is small (depending on the chosen $s$).
The proofs of Corollaries~\ref{cor-homol-bonded-autonom-contact-dynamics-general-case}, \ref{cor-stably-homol-bonded-non-autonom-contact-dynamics-general-case} and \ref{cor-non-positive-hamiltonians-separating-hypersurface-gen-case} then go through for $\Lambda'$ instead of $\Lambda$ and yield the existence of a chord of $h$ between $\Lambda'_0$ and $\Lambda'_1$.
{The cases when the pair $(\Lambda_0\sqcup\Lambda_1,\lambda)$ is homologically bonded or stably homologically bonded are similar.}
}
\end{rem}

\section{The case of $J^1 Q$}
\label{sec-the-1-jet-space}

In this section let $\Sigma = J^1 Q = T^* Q\times \R (z)$ be the 1-jet space of a smooth manifold $Q$, together with the standard contact form $\lambda$ on it. The Reeb flow of $\lambda$ is the shift in the $z$-coordinate. This is a nice contact manifold.

Let $\Lambda_0$ be the zero section
of $J^1 Q$.

\begin{prop}
\label{prop-lmin-for-J1Q-1} Assume that $\Lambda_1 \subset \Sigma=J^1 Q$ is a Legendrian submanifold
such that $V_\infty (\Lambda_1,\lambda)\neq 0$ and satisfying the
  following property: there is unique Reeb chord starting on $\Lambda_0$ and ending on $\Lambda_1$, and this chord is non-degenerate in the sense of \eqref{eq-nondeg}.

Then the pair $(\Lambda_0\sqcup\Lambda_1,\lambda)$ is homologically bonded.
\end{prop}

\begin{proof} Denote by $a$ the unique Reeb chord going from $\Lambda_0$ to $\Lambda_1$.

For each $i,j=0,1$ write $\cA_{ij}$ for the subalgebra of $\cA (\Lambda_0\sqcup \Lambda_1,\lambda)$
generated by all $ij$-composable monomials. We call a monomial in $\cA (\Lambda_0\sqcup \Lambda_1,\lambda)$ {\it good} if it is $11$-composable and contains no $a$. Note that each such monomial lies in $\cA (\Lambda_1,\lambda)$.

Let $J\in \cJ (\Lambda_0\sqcup \Lambda_1)$. Then $J$ also lies in $\cJ (\Lambda_1)$. Consequently, $J$ defines both a differential $d:=\partial_J$ on $\cA (\Lambda_0\sqcup \Lambda_1,\lambda)$ and a differential $d'$ on $\cA (\Lambda_1,\lambda)$.

Note that $da=0$. Indeed, $d$ should map $a$ into a sum of $01$-composable monomials whose actions are smaller than the action of $a$ but there are no such monomials.

\begin{lemma}
\label{lem-no-b-with db-equal-1}
There is no $x\in \cA_{11}$ such that $dx=1$.
\end{lemma}

\smallskip
\noindent
{\bf Proof of Lemma~\ref{lem-no-b-with db-equal-1}.}
Since $V_\infty (\Lambda_1,\lambda)\neq 0$, there exists no $x'\in \cA (\Lambda_1,\lambda)$ such that $d' x' = 1$. Indeed, otherwise every $d'$-closed element $y$ is $d'$-exact: $d'(yx')=y$.

Now assume by contradiction that there exists $x\in \cA_{11}$ such that $dx=1$. Then $x$ is a sum of $11$-composable monomials of the either of the following two types: either
$b_1 a b_2 a\ldots b_k a$ (type I) for some $k>0$, or $b_1 a b_2 a\ldots b_k a c$ (type II) for some $k\geq 0$, where $b_1,\ldots, b_k$ are $10$-composable monomials not containing $a$ and $c$ is a good monomial.

Since $da=0$ and $\cA_{10}$ is $d$-invariant, applying $d$ to a monomial $b_1 a b_2 a\ldots b_k a$ of type I, we get that $d (b_1 a b_2 a\ldots b_k a)$ is again a sum of monomials of type I, or zero.
Similarly, applying $d$ to a monomial $b_1 a b_2 a\ldots b_k a c$ of type II, we get that $d (b_1 a b_2 a\ldots b_k a c)$ is a sum of monomials of type II, with at least one $a$ in each of them, or zero -- unless $k=0$ and the original monomial of type II to which $d$ is applied is just a good monomial $c$. Since $dx=1$, we get that $x$ can be written as $x=x'+x''$, where $x'$ is a non-trivial sum of good monomials and $x''$ is a sum of monomials (of types I and II) containing $a$.

Each good monomial $c$ lies in $\cA (\Lambda_1,\lambda)$ and therefore $x'$, which is a sum of good monomials, also lies in $\cA (\Lambda_1,\lambda)$. For each good monomial $c$ we can write $dc = d'c + y$, where $y$ is a sum of $11$-composable monomials containing $a$. At the same time, the discussion above shows that $dx''$ is a (possibly trivial) sum of $11$-composable monomials containing $a$. Since $dx = 1$ and $d'c$ is a sum of $11$-composable monomials that do not contain $a$ and cannot cancel out monomials containing $a$, we get that $d'x'=1$, which yields a contradiction.

This finishes the proof of the lemma.
\Qed
\smallskip

Let us now finish the proof of the proposition.

Note that there are no $00$-chords (since $\Lambda_0$ is the zero section). Therefore any $01$-composable monomial has to be of either of the following two types: either $a B_1 a B_2\ldots a B_k a$ (type 1) for some $k>0$, or $a B_1 a B_2\ldots a B_k a C$ (type 2) for some $k\geq 0$,
where $B_1,\ldots, B_k$ are $10$-composable monomials containing no $a$ and $C$ is a good monomial.

We claim that there exists no $z\in \cA (\Lambda_0, \Lambda_1,\lambda)$ such that $dz=a$ -- since $da=0$, this would readily imply the proposition.

Let us prove the claim.
Assume by contradiction, that such a $z$ exists. It is a sum of $01$-composable monomials, each of which is either of type 1 or of type 2.

Since $da=0$, the differential of any monomial of type 1 or 2 is a (possibly zero) sum of monomials of the same type each of which contains {\sl more than one} $a$, except for the case where the monomial to which $d$ is applied is a monomial of type 2 of the form $aC$, where $C$ is a good monomial -- then $d(aC)= a\cdot dC$. Thus, the only way $dz$ can contain a monomial $a$ is that $z$ is a sum of monomials one of which is of the form $aC$, where $C$ is a good monomial such that $dC=1$. This leads to a contradiction with Lemma~\ref{lem-no-b-with db-equal-1}.

This finishes the proof of the claim and of the proposition.
\end{proof}

\medskip
Further on in this section,
assume that $l>0$ and $\Lambda_1$ is the image of $\Lambda_0$ under the time-$l$ Reeb flow.

\begin{prop}
\label{prop-lmin-for-J1Q}
The pair $(\Lambda_0\sqcup\Lambda_1,\lambda)$ is homologically bonded and
$$l_{min,s} (\Lambda_0,\Lambda_1,\lambda) \leq l$$
for all $s\in (1,+\infty]$.
\end{prop}

\bigskip
\noindent
{\bf Proof of Proposition~\ref{prop-lmin-for-J1Q}:}
The pair $(\Lambda_0\sqcup \Lambda_1,\lambda)$ is degenerate and we have to perturb, say, $\Lambda_1$ to make it non-degenerate.

Namely, consider a Morse function $f:Q\to\R$ which is a $C^\infty$-small perturbation of the constant function $z_+$. Its 1-jet is a Legendrian submanifold $\Lambda'_1\subset J^1 Q$ which is a small perturbation of $\Lambda_1$. It is not hard to check that the pair $(\Lambda_0\sqcup \Lambda'_1,\lambda)$ is non-degenerate: the Reeb chords of $\Lambda=\Lambda_0\sqcup\Lambda'_1$ are then the Reeb chords from $\Lambda_0$ to $\Lambda'_1$ corresponding to the critical points of $f$ and their actions are the critical values of $f$ shifted down by $z_+$. The $01$-subspace is then the $\Z_2$-span of the Reeb chords of $\Lambda$.

One can show that, under the identification between the Reeb chords of $\Lambda$ and the critical points of $f$, the differential in the Chekanov-Eliashberg algebra is identified with the Morse differential in the Morse chain complex of $f$ (over $\Z_2$). More precisely, the differential in the Chekanov-Eliashberg algebra is identified with the differential in the Lagrangian Floer complex of the projections of $\Lambda_0$ and $\Lambda'_1$ to $T^* Q$ (that are {\sl embedded} Lagrangian submanifolds) and the latter is identified with the Morse differential by the original work of Floer \cite{Floer-JDG88} -- see e.g. the comment preceding Cor. 1.10 in \cite{Chantraine-et-al-pos-Leg-isotopies}.

Thus, the persistence module associated to $(\Lambda_0,{\Lambda'_1},\lambda)$ is the Morse homology persistence module associated to $f$. The corresponding barcode contains infinite bars -- their number is equal to the sum of the Betti numbers of $Q$ over $\Z_2$. Hence, $l_{min,s} (\Lambda_0,\Lambda'_1,\lambda)\leq \max_Q$ for all $s\in (1,+\infty]$. Letting $f$ converge uniformly to the constant function $l$ we readily get that $l_{min,s} (\Lambda_0,\Lambda_1,\lambda) \leq l$
for all $s\in (1,+\infty]$, meaning that the pair $(\Lambda_0\sqcup\Lambda_1,\lambda)$ is homologically bonded.
\Qed

\bigskip
\noindent
{\bf Proof of Theorem~\ref{thm-1-jet-space-0-section-and-its-Reeb-shift-homol-bonded}:}
The Legendrian submanifolds appearing in the formulation of the theorem
are homologically bonded. For item (i) of the theorem, this follows from Proposition~\ref{prop-lmin-for-J1Q} combined with the fact that the property of being homologically bonded is invariant under a Legendrian isotopy of a pair. For item (ii) this follows from Proposition~\ref{prop-lmin-for-J1Q-1}. Therefore, the pair $(\Lambda_0,\Lambda_1)$ is interlinked by Corollary~\ref{cor-homol-bonded-autonom-contact-dynamics-general-case}. \qed

\bigskip
\begin{prop}
\label{prop-hlmin-for-J1Q}
The pair $(\Lambda_0\sqcup\Lambda_1,\lambda)$ is stably homologically bonded and
$\hl_{min,\infty} (\Lambda_0,\Lambda_1,\lambda) \leq l$.
\end{prop}

\bigskip
\noindent
{\bf Proof of Proposition~\ref{prop-hlmin-for-J1Q}:}
There is a natural identification $\hSigma= J^1 Q\times T^* \SP^1 = J^1 (Q\times \SP^1)$ identifying the contact forms, and hence the contact structures.
The Legendrian submanifold $\hLambda_0$ is then the zero-section of $J^1 (Q\times \SP^1)$ and $\hLambda_1$ is its image under the time-$l$ Reeb flow.
Now the result follows from Proposition~\ref{prop-lmin-for-J1Q}.\Qed

Let $\psi$ be a positive function on $\Lambda_0$, and let $\Lambda:= \{z=\psi(q), p=\psi'(q)\}$ be the graph of its 1-jet in $J^1 Q$.
Assume that $K$ is a Legendrian submanifold of $J^1 Q$ Legendrian isotopic to $\Lambda$ outside the zero section $\Lambda_0$, so that there exist exactly two Reeb chords $A,a$ starting on $K$ and ending on $\Lambda_0$, both non-degenerate, and with their time-lengths $|A|,|a|$ satisfying
\begin{equation}
\label{eq-twochords-1}
0 < |A|-|a| < |b|,
\end{equation}
for every Reeb chord $b$ starting and ending on $K \sqcup \Lambda_0$.

\begin{prop}
\label{prop-twochords}
For any $s < |A|/|a|$,
\begin{equation}
\label{eq-twochords-el}
l_{min,s} (K,\Lambda_0,\lambda) = |a|.
\end{equation}
\end{prop}

\begin{proof}
For each $i,j=0,1$ write $\cA_{ij}$ for the subalgebra of $\cA (K\sqcup \Lambda_0,\lambda)$
generated by all $ij$-composable monomials. {\bf Warning:} with our notation $01$-chords are the ones going from $K$ to $\Lambda_0$ etc. Observe that
$$\cA_{01} = \text{Span} \{ ua, uA, u\in \cA_{00}\}, \cA_{11} = 0.$$

By the invariance of $V_\infty$ under Legendrian isotopies of the two-part Legendrians and by Proposition~\ref{prop-lmin-for-J1Q}, we get
$$V_\infty(K,\Lambda_0,\lambda) = V_\infty(\Lambda,\Lambda_0,\lambda) = 0,$$
$$V_\infty(\Lambda_0,K,\lambda) = V_\infty(\Lambda_0,\Lambda,\lambda) = V_\infty(\Lambda_0,\Lambda_1,\lambda) \neq 0,$$
where $\Lambda_1$ is defined before Proposition~\ref{prop-lmin-for-J1Q}.
Let $J\in \cJ (K\sqcup \Lambda_0)$ and let $d:=\partial_J$ be the corresponding differential on $\cA (K\sqcup \Lambda_0,\lambda)$.

Note that $da =0$ as there is no $01$-composable monomial with a smaller action.
In what follows we write $|b|$ for the action of a monomial $b$.

Since there are no $11$-chords, we can write $dA=ua+vA$ with $u,v \in \cA_{00}$. Note that $v=0$, since otherwise $|vA| \geq |A|$
while $d$ lowers the action.
By assumption \eqref{eq-twochords-1}, $|ua| > |A|$ for a non-scalar $u$,
and hence either $u=0$ or $u=1$ (recall that the base field is $\Z_2$).

\medskip
\noindent
{\sc Case 1:} $u=0$, i.e.,\ $dA=0$. We claim that in this case $a \neq dx$ for any $x$.
Indeed, otherwise write $x = pa+qA$, where $p,q \in \cA_{00}$. Then $a = (dp)a+(dq)A$ yielding $dp=1$. But then for every closed $y \in \cA_{10}$ we have $y= d(yp)$, meaning that $y$ is exact, in
contradiction to $V_\infty(\Lambda_0,K,\lambda) \neq 0$.

\medskip
\noindent
{\sc Case 2:} $u=1$, i.e. $dA=a$. Note that by assumption \eqref{eq-twochords-1},
$A$ has the minimal action among all $01$-monomials of action $> |a|$. It follows
that the barcode of the persistence module $V(K,\Lambda_0,\lambda)$ contains
a bar $(|a|, |A|]$. Moreover, since $|a|$ is the minimal action among all $01$-chords,
we get that \eqref{eq-twochords-el}. This completes the proof.
\end{proof}

\medskip
\noindent
\medskip
\noindent
{\bf Proof of Theorem~\ref{thm-twochords}:} Using if necessary small perturbations of the Legendrians we may assume without loss of generality that all the chords are non-degenerate.
The theorem immediately follows from Proposition~\ref{prop-twochords} and Theorem~\ref{thm-Reeb-chords-pos-Ham-gen-case}.A.
\qed

\section{The case of $ST^* \R^n$}
\label{sec-the-unit-cotangent-bundle-of-Rn}

{In this section let}
\[
\Sigma := ST^* \R^n,\ n>1.
\]
Denote by $\xi$ the standard contact structure on $\Sigma$ defined by the contact form
\[
\lambda := pdq.
\]

Let $x_0,x_1\in \R^n (q)$, $x_0\neq x_1$. Consider the Legendrian submanifolds
\[
\Lambda_i := \{\ q=x_i,\ |p|=1\ \},\ i=0,1,
\]
of $(\Sigma,\xi)$.
Set
\[
\Lambda := \Lambda_0\sqcup \Lambda_1.
\]
{One easily checks that the pair $(\Lambda = \Lambda_0\sqcup \Lambda_1,\lambda)$ is non-degenerate.}

Consider also the manifold
\[
\hSigma := \Sigma \times T^* \SP^1 (r,\tau),
\]
and the 1-form
\[
\hlambda := pdq - rd\tau,
\]
where $r\in\R$, $\tau\in\SP^1$,
are the standard coordinates on
$T^* \SP^1 = \R \times \SP^1$. Let $\SP^1 := \{ r=0\} \subset T^* \SP^1$ be the zero-section.
One easily sees that $\hlambda$ is a contact form, defining a contact structure $\hxi$ on $\hSigma$ and the Reeb vector field of $\hlambda$ can be described as follows: its projection to the
$ST^* \R^n$ factor is the Reeb vector field of $\lambda$, while its projection to the $T^* \SP^1$ factor is zero.
Denote
\[
\hLambda_0 :=\Lambda_0\times \SP^1,\ \hLambda_1 :=\Lambda_1\times \SP^1.
\]
\[
\hLambda := \hLambda_0\sqcup\hLambda_1 = \Lambda\times \SP^1.
\]
For each $\delta>0$ denote
\[
\hLambda_{1,\delta} := \Lambda_1\times \graph (\delta df)\subset \hSigma,
\]
\[
\hLambda_\delta := \hLambda_0\sqcup\hLambda_{1,\delta},
\]
where $\graph (\delta df)\subset T^* \SP^1$ is the graph of $d(\delta f)$ for a Morse function $f: \SP^1\to\R$ that has only two critical points.
The set $\hLambda_\delta$ is a two-part Legendrian submanifold of $(\Sigma,\ker \hlambda)$.
{Since} the pair $(\Lambda,\lambda)$ is non-degenerate, so is the pair $(\hLambda_\delta,\hlambda)$.

\begin{prop}
\label{prop-lmin-for-union-of-two-cotangent-fibers-of-Rn}
The barcode of the persistence module $V(\Lambda_0,\Lambda_1,\lambda)$
consists of infinitely many infinite bars, of multiplicity $1$, whose left ends are $(2k-1) |x_0-x_1|$, $k\in\Z_{>0}$.

Consequently, the pair $(\Lambda_0\sqcup\Lambda_1,\lambda)$ is homologically bonded and
\[
l_{min,s} (\Lambda_0,\Lambda_1,\lambda) {= l_{min,\infty} (\Lambda_0,\Lambda_1,\lambda)} = |x_0-x_1|
\]
for all {$s\in (1,+\infty]$}.
\end{prop}

\bigskip
\noindent
{\bf Proof of Proposition~\ref{prop-lmin-for-union-of-two-cotangent-fibers-of-Rn}:}
The set $\cR (\Lambda, \lambda)$ consists only of two Reeb chords: a $01$-chord $a$ and a $10$-chord $b$, which are the lifts of the Euclidean geodesics in $\R^n$ going from $x_0$ to $x_1$ and from $x_1$ to $x_0$. Thus, $\cA (\Lambda_0,\Lambda_1,\lambda)$ is spanned over $\Z_2$ by the monomials of the form $abab\cdot\ldots\cdot ba$, where each monomial contains $k$ factors $a$ and $k-1$ factors $b$, for $k\in\Z_{>0}$. For each $k\in\Z_{>0}$ there is exactly one such monomial of length {$2k-1$}. The actions of $a$ and $b$ are both equal to $|x_0-x_1|$ and therefore the action of the monomial of length {$2k-1$} is $(2k-1)|x_0-x_1|$.

Let $J\in \cJ (\Lambda)$. Since the differential $\partial_J$ on $\cA (\Lambda_0,\Lambda_1,\lambda)$ lowers the action, we immediately get that $\partial_J (a) = \partial_J (b) =0$ and therefore $\partial_J$ is identically zero on $\cA (\Lambda_0,\Lambda_1,\lambda)$. Therefore the barcode of the persistence module $V(\Lambda_0,\Lambda_1,\lambda)$ consists of infinitely many infinite bars of multiplicity $1$ whose left ends are $(2k-1) |x_0-x_1|$, $k\in\Z_{>0}$. This finishes the proof.
\Qed

\begin{prop}
\label{prop-hlmin-for-union-of-two-cotangent-fibers-of-Rn}
The barcode of the persistence module $V(\hLambda_0,\hLambda_{1,\delta},\hlambda)$
consists of infinitely many infinite bars whose left ends are $(2k-1) |x_0-x_1|$, $k\in\Z_{>0}$. The multiplicity of the bar with the left end
$(2k-1) |x_0-x_1|$ is $2^{2k-1}$.

Consequently, the pair $(\Lambda_0\sqcup\Lambda_1,\lambda)$ is stably homologically bonded and
\[
\hl_{min,s} (\Lambda_0,\Lambda_1,\lambda) = {\hl_{min,\infty} (\Lambda_0,\Lambda_1,\lambda)} = |x_0-x_1|
\]
for all {$s\in (1,+\infty]$}.
\end{prop}

\bigskip
\noindent
{\bf Proof of Proposition~\ref{prop-hlmin-for-union-of-two-cotangent-fibers-of-Rn}:}
Let $a,b\in \cR (\Lambda, \lambda)$ be the $01$-chord and the $10$-chord of $\Lambda$ as above. The Reeb chords of $\hLambda_\delta$ are the direct products of the Reeb chords of $\cR (\Lambda, \lambda)$ lying in $\Sigma$ and the constant paths in $T^* \SP^1$ corresponding to the $2$ critical points of $f$ (that is, the intersections of the graph of $df$ with the zero-section). Thus, $\cR (\hLambda_\delta, \hlambda)$ consists only of exactly $4$ Reeb chords: $01$-chords $a_1,a_2$ and $10$-chords $b_1,b_2$, where $a_1,a_2$ project onto $a$ and $b_1,b_2$ project onto $b$ under the projection $\hSigma = \Sigma \times T^* \SP^1\to\Sigma$.

Thus, $\cA (\hLambda_0,\hLambda_{1,\delta},\lambda)$ is spanned over $\Z_2$ by the monomials of the form $a_{j_1} b_{j_2}\cdot\ldots\cdot b_{j_{2k-2}}a_{j_{2k-1}}$, where each monomial contains $k$ factors $a_{j_i}$, $j_i=1,2$, and $k-1$ factors $b_{j_i}$, $j_i=1,2$, for $k\in\Z_{>0}$. There are $2^{2k-1}$ such monomials of length $2k-1$. The actions of all $a_j$ and $b_j$, $j=1,2$, are equal to $|x_0-x_1|$ and therefore the action of a monomial as above is $(2k-1)|x_0-x_1|$.

Let $J\in \cJ (\hLambda_\delta)$. Since the differential $\partial_J$ on $\cA (\hLambda_0,\hLambda_{1,\delta},\hlambda)$ lowers the action, we immediately get that $\partial_J (a_j) = \partial_J (b_j) =0$ for all $j=1,2$ and therefore $\partial_J$ is identically zero on $\cA (\hLambda_0,\hLambda_\delta,\hlambda)$. Therefore the barcode of the persistence module $V(\hLambda_0,\hLambda_{1,\delta}\hlambda)$ consists of infinitely many infinite bars whose left ends are $(2k-1) |x_0-x_1|$, $k\in\Z_{>0}$. The multiplicity of the bar with the left end
$(2k-1) |x_0-x_1|$ is $2^{2k-1}$.

Consequently, for any {$s\in (1,+\infty]$} and $\delta>0$
\[
l_{min,s} (\hLambda_0,\hLambda_{1,\delta},\hlambda) = |x_0-x_1|
\]
and
\[
\hl_{min,s} (\Lambda_0,\Lambda_1,\lambda) = \liminf_{\delta\to 0} l_{min,s} (\hLambda_0,\hLambda_{1,\delta},\hlambda) = |x_0-x_1|.
\]
This finishes the proof.
\Qed

\bigskip
\noindent
{\bf Proof of Theorem~\ref{thm-main-R2n}:}
In the case $n=1$ the claim follows from the results in \cite{EP-tetragons}. Namely, in this case the sets $X_0,X_1,Y_0,Y_1$ form a Lagrangian tetragon
in $(\R^2 (p,q), dp\wedge dq)$ built from the point $x_0\in\R$ for $T= x_1 - x_0$ (see \cite[Sec. 5.1]{EP-tetragons}).  In the terminology of \cite{EP-tetragons}, this Lagrangian tetragon is stably $\kappa$-interlinked, for $\kappa = |x_0-x_1| (s_+ - s_-)$
-- this follows e.g. from Cor. 5.3 and Thm. 5.8 in \cite{EP-tetragons} (see \cite[Thm. 1.20, Prop.1.21]{BEP} for a different approach to the proof). By the definition of a $\kappa$-interlinked Lagrangian tetragon (see \cite[Sec. 1.2]{EP-tetragons}), this yields the dynamical claims of Theorem~\ref{thm-main-R2n} in the case $n=1$.

Assume now that $n>1$.

By Propositions~\ref{prop-lmin-for-union-of-two-cotangent-fibers-of-Rn}, \ref{prop-hlmin-for-union-of-two-cotangent-fibers-of-Rn},
\[
l_{min,s_+/s_-} (\Lambda_0,\Lambda_1,\lambda) = \hl_{min,s_+/s_-} (\Lambda_0,\Lambda_1,\lambda) = |x_0-x_1|.
\]

The symplectization $\Sigma \times \R_+ (s)$ is identified symplectically
with $\R^{2n}(p,q)\setminus \{ p=0\}$ by the map $(p,q,s) \mapsto (sp,q)$. Thus, $\left(\Sigma\times [s_-,s_+], d(s\lambda)\right)$ can be viewed
as a codimension-zero submanifold with boundary of $(N=\R^{2n},\Omega=dp\wedge dq)$.

Now the claims of the theorem follow from Corollaries~\ref{cor-chord-of-autonomous-H-in-N}, \ref{cor-chord-of-non-autonomous-H-in-N} applied to the Hamiltonian $H$ on $(N,\Omega)$.

This finishes the proof of the theorem.
\Qed

\bigskip
\noindent
{\bf Proofs of Theorems~\ref{thm-conformal-factor-unit-cot-bundle}, \ref{thm-main-existence-of-connecting-Reeb-chords-unit-cot-bundle} and Corollary~\ref{cor-non-positive-hamiltonians-separating-hypersurface-unit-cot-bundle}:}
We need to prove Theorems~\ref{thm-conformal-factor-unit-cot-bundle}, \ref{thm-main-existence-of-connecting-Reeb-chords-unit-cot-bundle} and Corollary~\ref{cor-non-positive-hamiltonians-separating-hypersurface-unit-cot-bundle} in the following three cases:

\bigskip
\noindent
(I) $\Lambda_0$ is the zero-section of $J^1 Q$ and $\Lambda_1$ is its image under the time-$l$ Reeb flow.

\bigskip
\noindent
(II) $\Lambda_0$ is a cotangent unit sphere in $ST^* \R^n$ and $\Lambda_1$ is its image under the time-$l$ Reeb flow.

\bigskip
\noindent
(III) $\Lambda_0, \Lambda_1\subset ST^* \R^n$ are the cotangent unit spheres at $x_0,x_1\in\R^n$, $|x_0-x_1|=l$.

\bigskip
The needed results in cases (I), (II), (III) follow from the corresponding general results in Theorem~\ref{thm-conformal-factor-gen-case}, Corollaries~\ref{cor-homol-bonded-autonom-contact-dynamics-general-case}, \ref{cor-stably-homol-bonded-non-autonom-contact-dynamics-general-case}
and Corollary~\ref{cor-non-positive-hamiltonians-separating-hypersurface-gen-case}, as soon as we check that in all the three cases
$l_{min,\infty} (\Lambda_0,\Lambda_1,\lambda)\geq l$, $\hl_{min,\infty} (\Lambda_0,\Lambda_1,\lambda)\geq l$.

In case (I) these inequalities are proved in Propositions~\ref{prop-lmin-for-J1Q}, \ref{prop-hlmin-for-J1Q}.

In the case (II) they also follow from Propositions~\ref{prop-lmin-for-J1Q}, \ref{prop-hlmin-for-J1Q}, because the cases (I) and (II) can be
identified by a contactomorphism preserving the contact forms -- namely,
the contact identification of $ST^* \R^n$ and $J^1 \SP^{n-1}$ (see \eqref{eq-hodograph}) can be easily adjusted so that $\Lambda_0\subset ST^* \R^n$ is identified with the zero-section of $J^1 \SP^{n-1}$.

In case (III) the inequalities follow from Propositions~\ref{prop-lmin-for-union-of-two-cotangent-fibers-of-Rn}, \ref{prop-hlmin-for-union-of-two-cotangent-fibers-of-Rn}.
\Qed

\bibliographystyle{alpha}

\end{document}